%% file: BLF-PLF-August-2026.tex
\documentclass[10pt]{article}  
\usepackage[utf8]{inputenc}
\usepackage[T1]{fontenc}
\usepackage{mathrsfs}
\usepackage{amsbsy} 
\usepackage{amsmath, amsfonts, amsthm, amssymb, amscd}
\usepackage{accents, verbatim}
\usepackage{pdfpages}
\input{epsf}
\usepackage{graphicx}
\usepackage{xcolor}
\usepackage{microtype}
\usepackage[nottoc]{tocbibind} 
\usepackage[normalem]{ulem}
\usepackage{zref-perpage} \zmakeperpage{footnote} 
\usepackage{float}
\usepackage{booktabs}
\usepackage{bookmark,hyperref}
\usepackage{a4wide}
\usepackage{titlesec}
\input{macros-DEUX-v22.tex}
\renewcommand{\thefootnote}{\arabic{footnote}}
\titleformat{\subsubsection}[runin]{\normalfont\bfseries}{}{.5em}{}[\ifnum\spacefactor<1001.\fi\ \ ]
\titlespacing{\subsubsection}{0pt}{2ex plus.1ex minus.2ex}{0pt}
\titleformat{\paragraph}[runin]{\normalfont\itshape}{}{.5em}{\hspace{-1pt}}[\ifnum\spacefactor<1001.\fi\ \ ]
\titlespacing{\paragraph}{0pt}{1ex plus.1ex minus.2ex}{0pt}

\usepackage{calligra}
\usepackage[mathscr]{euscript}
\usepackage{bbm}
\numberwithin{table}{section}

\usepackage{needspace}
\usepackage{etoolbox}

\pretocmd{\section}{\needspace{8\baselineskip}}{}{}

\begin{document}

\makeatletter
\ifnum\@ptsize=2 \pretocmd{\@maketitle}{\vspace*{-4cm}\@gobble}\fi
\makeatother

\title{\bf Finite-energy spacetimes with torus symmetry.  
\\
\Large 
Cauchy stability of Einstein--Euler and
\\
Einstein--Navier--Stokes areal flows}

\author{Bruno Le Floch\thanks{Laboratoire de Physique Théorique et Hautes Énergies, Sorbonne Universit\'e \& Centre National de la Recherche Scientifique, 4 Place Jussieu, 75252 Paris, France. Email: {\tt blefloch@lpthe.jussieu.fr}.}
\, 
and Philippe G. LeFloch\thanks{Laboratoire Jacques-Louis Lions, Sorbonne Universit\'e \& Centre National de la Recherche Scientifique, 4~Place Jussieu, 75252 Paris, France. Email: {\tt contact@philippelefloch.org}. 
}
}
\date{}
\maketitle

\

\begingroup
\renewcommand{\thefootnote}{\ensuremath{\ddagger}}
\begin{center}
\vspace*{-.9cm plus 5cm}
{\scshape\small In memory of Yvonne Choquet--Bruhat\footnote{\normalfont In his contribution to the collective tribute \emph{Yvonne Choquet--Bruhat 1923--2025}, Preprint arXiv:2605.01956, 2026, pp.~15--16, the second author offers a personal recollection of Yvonne's generous encouragement, her regular participation in the \emph{S\'eminaire de Relativit\'e Math\'ematique} in Paris, and the many stimulating scientific conversations they shared.
\newline
\phantom{.}  \hfill August 2026}
}
\vspace*{1cm plus 5cm}
\end{center}
\addtocounter{footnote}{-1}
\endgroup


\begin{abstract}
We establish finite-energy Cauchy stability for vacuum and matter spacetimes with $\Tbb^2$ symmetry on $\Tbb^3$, under general constitutive equations satisfying hyperbolicity and mild asymptotic conditions at vacuum and large mass-energy density. First, using the area of the symmetry orbits as a time function, we introduce the \emph{\JKL{} formulation of Einstein areal flows}, a first-order evolution--constraint system coupling a wave-map structure for the geometry to hyperbolic matter balance laws and weighted transport equations for the twists and the momentum tangent to the symmetry orbits. Second, we also define two classes of hyperbolic Einstein--Navier--Stokes models, constructed from (as we call them) a Navier--Stokes potential and a relaxation rate map. The proposed \emph{particle-production} model has divergence-free matter stress and non-negative particle-number production, while the proposed \emph{dissipated-energy} model conserves particle number and dissipates fluid mass-energy, while an auxiliary stress restores the total stress-energy conservation required by the contracted Bianchi identity. For regular finite-energy (Einstein, Euler, Navier--Stokes) flows, we establish maximum and invariant-domain principles, geometric and energy monotonicity formulas, spacelike and timelike energy estimates, and total-variation estimates. Third, together, these estimates yield finite-energy Cauchy stability in the future expanding regime and, in the contracting regime, until the spacelike volume collapses. They are independent of the viscosity (or relaxation) mechanisms and cover vacuum, unbounded mass-energy density, and arbitrary finite rapidity. They also extend to weakly regular Einstein--Euler areal flows satisfying a reference mathematical entropy inequality.
\end{abstract}

\clearpage 

{\small
 
\setcounter{secnumdepth}{2}
\setcounter{tocdepth}{2}
\tableofcontents
}

\section{Introduction}
\label{section-1}

\subsection{The class of spacetimes with torus symmetry}

\label{section-1-1}

\subsubsection{A finite-energy formulation of the Einstein--matter system}

We investigate the class of possibly non-vacuum spacetimes with two commuting spacelike Killing fields and compact Cauchy hypersurfaces. Under the dominant energy condition, the area of the symmetry orbits has a timelike gradient, and the Einstein Cauchy developments considered here admit a spacelike foliation whose time function coincides with the area of the torus-symmetry orbits~\cite{Rendall-book}. This choice is geometrically natural, and it also has several analytic advantages: it reduces Einstein's field equations to a coupled hyperbolic--elliptic system of partial differential equations (PDEs) on the two-dimensional quotient of the spacetime by the torus symmetry group. Our purpose in this paper is to establish the structural and stability properties that are relevant at the level of finite-energy regularity, in the sense explained below. At this regularity, the Einstein equations require a suitable weak formulation, which was introduced in the companion paper~\cite{TorusT2-1-2026}. The present paper establishes the propagation of finite-energy regularity from an initial areal hypersurface for Einstein--matter systems, in which the matter is either a perfect fluid governed by a general equation of state or a dissipative fluid satisfying one of the Einstein--Navier--Stokes systems introduced below. This work is part of a broader project devoted to the analytical and geometric properties of spacetimes with torus symmetry and to the weak convergence properties of sequences of such spacetimes~\cite{TorusT2-1-2026,TorusT2-3-2026,TorusT2-4-2026}. 


\subsubsection{Einstein equations with general matter models}

We thus consider $(1+3)$-dimensional spacetimes $(M,g)$, where \(M\) is a smooth, time-oriented four-dimensional manifold and \(g\) is a Lorentzian metric with signature \((-,+,+,+)\), satisfying
\be
G_{\alpha\beta} = T_{\alpha\beta}. 
\ee
Here, $G_{\alpha\beta}\coloneqq R_{\alpha\beta}-\frac12R g_{\alpha\beta}$ is the Einstein tensor, determined from the Ricci curvature $R_{\alpha\beta}$ and scalar curvature $R$ of the metric $g$. Greek indices \(\alpha,\beta,\gamma,\ldots\) are abstract Penrose indices, while Latin indices \(m,n,\ldots\), ranging over \(\{0,1,2,3\}\), will later denote components in an orthonormal frame. Indices are raised and lowered with \(g\) and its inverse \(g^{-1}=(g^{\alpha\beta})\), and repeated upper and lower indices are summed. We denote by \(\nabla\) the Levi-Civita connection of \(g\), and Einstein's gravitational coupling constant is absorbed into the normalization.

Here, $T_{\alpha\beta}$ is a symmetric stress-energy tensor invariant under the torus action. More precisely, if \(X_2,X_3\) denote the two Killing fields generating this action, invariance means that \(\mathcal L_{X_2}T=\mathcal L_{X_3}T=0\). The contracted Bianchi identity imposes the matter balance laws
\be
\nabla^\alpha T_{\alpha\beta}=0.
\ee
Recall that \(\nabla^\alpha\coloneqq g^{\alpha\gamma}\nabla_\gamma\), and this equation represents four balance laws in \(3+1\) dimensions. At finite-energy regularity, the matter balance laws are imposed in the sense of distributions, while Einstein's equations are interpreted through the first-order weak formulation described below. 

We first keep the constitutive law general, and the geometric part of the proposed reduction of the Einstein equations requires only the frame components of $T_{\alpha\beta}$, their balance laws, their local integrability, and natural positivity properties required by our finite-energy estimates. In turn, the framework applies to several matter models, including vacuum, scalar fields, compressible fluids, fluids with phase transitions, and relativistic Navier--Stokes systems. Additional constitutive variables are specified when needed, and the relevant mathematical entropy inequality is specified for each fluid model when needed.


\subsubsection{The proposed $(\Jbb,\Kpar,\Lbb)$ formulation}

An important class covered by the theory consists of perfect fluids governed by an equation of state $p=p(\mu)$, where $\mu\geq0$ denotes the mass-energy density. For a fluid state with future-directed unit velocity $u^\alpha$, normalized by \(g_{\alpha\beta}u^\alpha u^\beta=-1\), the stress-tensor is $T^{\alpha\beta}=(\mu+p(\mu))u^\alpha u^\beta+p(\mu)g^{\alpha\beta}$.  We introduce the matter momentum one-form
\be
\Jbb\coloneqq\sqrt{2\bigl(\mu+p(\mu)\bigr)} u^\flat,
\ee
where \(u^\flat_\alpha\coloneqq g_{\alpha\beta}u^\beta\). Consequently, \(g^{-1}(\Jbb,\Jbb)=-2\bigl(\mu+p(\mu)\bigr)\leq0\), and the metric-dual vector \(\Jbb^\sharp\) is future directed and causal. At finite-energy regularity, we  treat $\Jbb$ as a primary variable ranging over the closed future causal cone. Importantly, it remains meaningful at \emph{vacuum}, where the velocity~$u$ need not be defined; indeed, \(\Jbb\) becomes null at vacuum, so that the (closed) finite-energy state space also contains zero and non-zero null states.

In torus symmetry, the geometric unknowns are the so-called twist variable $\Kpar$, which measures the failure of the distribution orthogonal to the symmetry orbits to be integrable, and a first-order variable $\Lbb$ describing the geometry of the orbits, together with the lapse function $\Omega$ and the conformal length density $\amdeux$ associated with the areal foliation. More precisely, the complex-valued field 
\be
\Kpar=K_2+iK_3
\ee
collects the two twist functions, while the complex-valued field 
\be
\Lbb=\Pbb+i\Qbb=dP+i e^P dQ
\ee
collects the first derivatives of the two metric coefficients denoted by \(P,Q\) and describing the conformal geometry of the symmetry orbits. The variables \(\Kpar,\Lbb,\Omega,\amdeux\) descend to the two-dimensional quotient of spacetime by the action of the symmetry group.

Collecting the matter and first-order geometric fields $(\Jbb,\Kpar,\Lbb)$, together with $(\Omega,\amdeux)$, defines the \emph{\JKL{} formulation of Einstein areal flows}, as we call these solutions to the Einstein equations. In these variables, the evolution subsystem is hyperbolic (away from its vacuum degeneracy), while the constraints determine the remaining fields through transport and reconstruction equations. Interestingly, our  formulation is robust enough to be extended continuously to the vacuum boundary in terms of~\(\Jbb\). The nonlinear terms controlled by the energy are quadratic and enjoy a partial null structure.\!\footnote{The structure of the Euler equations \emph{without} symmetry restriction was thoughroughly investigated by Speck and his collaborators~\cite{AbbresciaSpeck-2025,DisconziLuoMazzoneSpeck-2022,Speck-2019}.} This structure naturally leads us to introduce an $L^2$-based theory in which the finite-energy variables $(\Jbb,\Kpar,\Lbb)$ are square-integrable on spacelike hypersurfaces with respect to their induced volume measure.  The lapse function and the conformal-length density are instead absolutely continuous or of bounded variation.


\subsubsection{Perfect fluids}

We first treat perfect fluids whose equation of state satisfies the hyperbolicity and subluminal-sound-speed conditions, together with mild asymptotic assumptions near vacuum and at large density. Thus, the mass-energy density may approach zero, corresponding to cavitation, or become arbitrarily large, corresponding to concentration. Compressible-fluid evolution may form \emph{shock waves} even from smooth initial data~\cite{Lax-1957,Lax-1971,Dafermos-book}. The finite-energy requirement that $\Jbb$ be square-integrable is compatible with jump discontinuities in the matter momentum. The balance laws are then imposed in the sense of distributions, while a mathematical entropy inequality selects the admissible shocks. We distinguish between two weak formulations of the Einstein--Euler equations, namely the \emph{particle-production formulation} and the \emph{dissipated-energy formulation}.


\subsubsection{Dissipative fluids with a mathematical entropy}

We also propose and study two hyperbolic models of dissipative relativistic fluids. Earlier treatments of relativistic fluid flows and their symmetric formulations include~\cite{Barnes-2004,LeFlochStewart-2005,LeFlochUkai-2009}, while several covariant dissipative theories have been proposed in~\cite{GerochLindblom-1990,BDN2019,ReintjesChaddha}. Crucially, covariance alone does not determine a unique relativistic Navier--Stokes system. Since we require finite propagation speed, we consider relaxation systems whose principal matrices are symmetric and define symmetric-hyperbolic equations, and whose characteristic cones are contained in the Lorentzian cone. The construction proposed in this paper builds on the theory of hyperbolic relaxation~\cite{ChenLevermoreLiu} and the covariant symmetric formulation of relativistic fluids~\cite{RuggeriStrumia-1981}, as well as the mathematical entropy structure going back to Godunov and Friedrichs--Lax's pioneering works~\cite{Godunov-1961,FriedrichsLax-1971}.

Both systems are defined from the same covariant scalar potential $\chi=\chi(\beta,\zeta)$. Here, $\beta$  is a rescaled fluid momentum that determines the mass-energy density and velocity, while the additional field $\zeta$ measures the departure from the Euler state and vanishes at equilibrium. Second-order derivatives of $\chi$ define a symmetric matter stress tensor $T_\chi$, an auxiliary current $A_\chi$, and a particle current $\vNumb_\chi$. On the equilibrium manifold $\zeta=0$, the tensors $T_\chi$ and $\vNumb_\chi$ agree with the perfect-fluid stress tensor and particle current, respectively. The divergence equation for $A_\chi$ determines the evolution of $\zeta$, and its source drives this variable toward the Euler equilibrium manifold.

In both formulations, we denote by $T^{\mathrm{NS}}_{\alpha\beta}$ the total symmetric stress-energy tensor of the dissipative fluid. It replaces the perfect-fluid stress tensor in Einstein's equations and reduces to it on the Euler equilibrium manifold $\zeta=0$. In the particle-production formulation, the matter stress tensor defined from the potential is divergence-free, while the particle current has non-negative divergence. In the dissipated-energy formulation, the particle current is divergence-free, while the matter stress tensor defined from the potential satisfies a balance law with a signed source. The tensor $T^{\mathrm{NS}}_{\alpha\beta}$ then includes an (explicitly determined) symmetric tensor tangent to the symmetry orbits whose divergence cancels this source. Consequently, both proposed Einstein--Navier--Stokes systems satisfy
\be
G_{\alpha\beta}=T^{\mathrm{NS}}_{\alpha\beta},
\qquad
\nabla^\alpha T^{\mathrm{NS}}_{\alpha\beta}=0,
\ee
as required by the contracted Bianchi identity. The two formulations therefore preserve the same total stress-energy balance law, but they lead to different mathematical entropy inequalities: one describes particle production, while the other describes the dissipation of mass-energy.


\subsection{Formulation and stability of the Einstein--matter equations}
\label{section-1-2}

\subsubsection{The  \JKL{} formulation}

We establish two main results concerning the formulation and Cauchy stability of \emph{Einstein--matter areal flows}.  First of all, we provide a \emph{first-order} formulation of the Einstein equations in areal gauge, designed to exhibit the hyperbolicity, constraint propagation, and a quadratic nonlinearity structure.
and to establish finite-energy properties of Einstein Cauchy developments. Our formulation organizes the geometry and matter variables into two classes with different regularity requirements. Namely, we set
\bel{equa-phipsi}
\Phi=(\Jperp,\Jpar,\Pbb,\Qbb),
\qquad
\Psi=(\ell,\log\Omega,\Kpar).
\ee
The variables in $\Phi$ arise in a divergence--curl system and in matter balance laws and are naturally controlled by a quadratic geometry-matter energy, whereas the remaining variables in $\Psi$ satisfy transport and constraint equations. In addition to an $L^2$ control of the variables $\Phi$, we require bounded-variation (or absolute-continuity) regularity of $\Psi$. In this setting, the torus-symmetric Einstein--Euler equations consist of twelve first-order nonlinear evolution balance laws and three constraints. The equations for $(\Pbb,\Qbb)$ have a divergence--curl structure, and the equations for the momentum tangent to the symmetry orbits
and the twists are transport laws. Crucially, since vacuum generally can form in the evolution of the Einstein--matter system, 
the differential operators introduced in this paper remain meaningful at \emph{vacuum}.
The same first-order presentation of the Einstein equations applies to the two classes of Einstein--Navier--Stokes systems introduced below. 
The state space of the fluid momentum $\Jbb$ is the future causal cone, whose boundary corresponds to vanishing mass density $\mu=0$.  The Einstein--Navier--Stokes systems are conveniently expressed in terms of a rescaled fluid momentum~$\beta$ that is only defined away from vacuum, but extend to the vacuum in~$\Jbb$ variables.

\begin{theorem}[The \JKL{} formulation of Einstein areal flows]
\label{theo-main-results}
Consider four-dimensional spacetimes $(\Mcal,\guntrois)$ with spatial topology $\Tbb^3$ and two commuting spacelike Killing fields with closed orbits, admitting a foliation by the area~$|t|$ of the symmetry orbits. Consider a constitutive law $p=p(\mu)$ (pressure vs.\ mass-energy density) satisfying hyperbolicity and asymptotic conditions (stated in~\eqref{hyperbolic-eos-zero}, below). Then the following properties hold.
\bei

\item \textbf{Perfect-fluid matter model.} In the class of regular flows, the Einstein--Euler system 
\bel{equa-49F}
G^{\alpha\beta} = T^{\alpha\beta},
\qquad
\nabla_\alpha T^{\alpha\beta}=0
\ee
is equivalent to an \emph{evolution--constraint, first-order differential system} in terms of the variables \eqref{equa-phipsi}, as  given in \eqref{eq:T2-1234-def}--\eqref{equa-9385}, below, and referred to as the \JKL{} system. Its evolution subsystem has essentially \emph{quadratic nonlinearities} with a partial \emph{null structure} and, away from vacuum, is strongly hyperbolic.   All differential operators in the formulation 
remain defined on the full (closed) state space, including cavitation and near-light speed fluid propagation, although the hyperbolicity degenerates at non-zero null states.  

\item \textbf{Particle-production dissipative model.} Analogous properties hold for the following class of Einstein--Navier--Stokes systems
\bel{eq:pot-production-system0}
G^{\alpha\beta}=T_\chi^{\alpha\beta},
\qquad
\nabla_\alpha T_\chi^{\alpha\beta}=0,
\qquad
\nabla_\alpha A_\chi^{\alpha\beta\gamma}
=-\mathscr B^{\beta\gamma}, 
\ee
namely an extension of the Einstein--Euler system by adjoining to the fluid variables a symmetric two-tensor relaxation field~$\zeta$. The tensors $T_\chi$ and $A_\chi$ are defined from second-order derivatives
of a \emph{Navier--Stokes potential} denoted by~$\chi$, which is assumed to satisfy suitable equilibrium, dominant-energy, and causal-hyperbolicity conditions in \autoref{def-pot-hyperbolic}. The relaxation rate tensor~$\mathscr B$ is prescribed and satisfies suitable stable-equilibrium and positivity conditions in \autoref{def-pot-source}. In particular, the coupled evolution equations are symmetric hyperbolic in the interior of the state space, and both Einstein equations and Bianchi equations are satisfied.  The associated particle-number current~$\vNumb_\chi$ satisfies $\nabla_\alpha \vNumb_\chi^\alpha = \zeta_{\alpha\beta}\mathscr B^{\alpha\beta}\geq 0$.

\item \textbf{Dissipated-energy dissipative model.} 
Analogous properties also hold for the following class of Einstein--Navier--Stokes systems
\bel{eq:pot-dissipated-system0}
\begin{aligned}
G^{\alpha\beta}
&=T_\chi^{\alpha\beta}+T_{\rm aux}^{\alpha\beta}, \quad
&
\nabla_\alpha T_\chi^{\alpha\beta}
&=-\sigma n^\beta, \quad
&
\nabla_\alpha A_\chi^{\alpha\beta\gamma}
&=-\mathscr B^{\beta\gamma},
\end{aligned}
\ee 
with 
\be
\begin{aligned}
T_{\rm aux}^{\alpha\beta}
&=\sigma \underline\pi^{\alpha\beta},
&
\sigma & \coloneqq \frac{\zeta_{\alpha\beta} \mathscr B^{\alpha\beta}}{-n_\alpha\beta^\alpha} \geq 0. 
\end{aligned}
\ee 
Here $\underline\pi^{\alpha\beta}$ is the (rescaled) projection onto the symmetry orbits, $n_\alpha$ is the future-directed normal to the areal foliation, and $\beta$ is the rescaled fluid momentum.
The model includes an auxiliary stress tensor~$T_{\rm aux}$ (\autoref{def-diss-en}), which compensates the mass-energy dissipated by~$T_\chi$. In particular, the coupled evolution equations are symmetric hyperbolic in the interior of the state space, and both Einstein equations and Bianchi equations are satisfied.
The associated particle-number current~$\vNumb_\chi$ is divergence-free.
\eei
\noindent For all three systems, the Einstein constraints propagate from one areal hypersurface, and the metric variables are reconstructed from the transport, constraint, and periodicity equations. 
\end{theorem}

Propositions~\ref{theo-struct} and~\ref{theo-struct-bis} establish these properties for the Einstein--Euler system, and \autoref{prop-pot-JKL} establishes them for both Einstein--Navier--Stokes models.

The formulation also has a direct extension to weakly regular finite-energy Einstein--Euler flows, as proved earlier in~\cite{TorusT2-1-2026}. Indeed, the Einstein equations remain meaningful in the sense of distributions, and may contain an additional measure-valued tensor field, a non-negative stress-energy corrector~\cite{LeFlochLeFloch-1,LeFlochLeFloch-port}, representing the possible  oscillatory contributions arising from weak limits of quadratic geometry terms~\cite{TorusT2-1-2026}. The variables and first-order equations in our formulation therefore apply robustly to regular solutions, limits thereof, and weak solutions, to either Einstein--Euler or Einstein--Navier--Stokes systems.


\subsubsection{Stability of future Cauchy developments}

Our second main result concerns the initial value problem for geometry and matter data prescribed on an areal hypersurface. We control the corresponding future Cauchy development in both regimes: the area of the symmetry orbits increases when \(t_0>0\) and decreases when \(t_0<0\).
We prove stability in the sense of propagation of finite-energy bounds from the initial hypersurface; stability here does not refer to a Lipschitz estimate for the difference of two developments.
The estimates rely on several geometric and analytical properties of torus-symmetric spacetimes. The volume and conformal length satisfy monotonicity formulas, while weighted energy identities yield spacelike and timelike integral estimates. Exact transport equations provide pointwise control. For regular Einstein--Euler flows and dissipated-energy flows, the initial ranges of the momentum components tangent to the symmetry orbits and of the twist variables are preserved. For regular positive-density flows governed by the particle-production system, the normalized parallel momentum is contracted radially toward the origin.
These estimates are robustly expressed as inequalities between stress-tensor momentum components along the symmetry orbits (spanned by the Killing vectors $Y,Z$) and the particle number: the parallel momentum per particle is said to be bounded by a constant $C^\parallel>0$ provided
\bel{intro-parall-perpart-bounded}
|n^\alpha T_{\alpha\beta}Y^\beta|+|n^\alpha T_{\alpha\beta}Z^\beta|\leq C^\parallel (-n_\alpha\vNumb_\chi^\alpha) .
\ee
For the selected matter system, let \(\mathcal E(t)\) denote the energy on the areal hypersurface at time~\(t\), as defined in~\eqref{V-L-E-expr}. It combines the spatial integral of the normal matter energy \(n^\alpha T_{\chi\,\alpha\beta}n^\beta\), the quadratic \(L^2\) energy of the first-order geometric variables \((\Kpar,\Pbb,\Qbb)\), and a weighted conformal-length term involving~\(\mathcal L(t)\). Since the reduced fields depend on a single spatial variable, the bounded-variation norms appearing in \autoref{theo-intro-stability} also control the corresponding supremum norms.

In the geometric description by an initial data set \((\mathring g,\mathring k,\mathring\Jbb,\mathring\zeta)\), the sign of~\(t_0\) is determined by the mean curvature. More precisely, under the dominant energy condition, the Einstein constraints imply that \(\Tr_{\mathring g}\mathring k\) has a constant sign on the initial hypersurface, and our orientation convention gives
\be
\sgn(t_0)=-\sgn\bigl(\Tr_{\mathring g}\mathring k\bigr).
\ee
By contrast, the description by the \JKL{} variables \((\mathring\Phi,\mathring\Psi,\mathring\zeta)\) does not encode this sign, which must therefore be prescribed separately.  Two integration constants must also be specified to fully reconstruct the geometry from \JKL{} variables, namely the conformal modulus \(\tau\in\CC\) with \(\Im(\tau)>0\) of a single $\Tbb^2$-symmetry orbit.

\begin{theorem}[The Cauchy stability of finite-energy  Einstein areal flows]
\label{theo-intro-stability}
Consider one of the three Einstein--matter models in \autoref{theo-main-results} with $\Tbb^2$ symmetry on~$\Tbb^3$.
In the Einstein--Navier--Stokes case, assume that the model is admissible and torus-adapted, and that its matter stress tensor~$T_\chi$ obeys the coercive stress condition in \autoref{def-Tchi-coercive}.
In the Einstein--Euler case, we use the conventions $T_\chi=T_{\rm E}$, $\vNumb_\chi=\vNumb_\Jbb$, and $\zeta=0$.
Prescribe\footnote{Equivalently, prescribe a signed areal time~$t_0$ and initial data $(\mathring\Phi,\mathring\Psi,\mathring\zeta)$ for the \JKL{} formulation of the areal flow, together with the conformal modulus of one symmetry orbit.}
 an initial data set $(\mathring g,\mathring k,\mathring\Jbb,\mathring\zeta)$, consisting of a Riemannian metric, second fundamental form, fluid momentum, and relaxation field (absent for the Einstein--Euler model) satisfying Einstein's constraint equations and invariant under two commuting Killing fields $Y,Z$ with closed orbits, and assume that the orbit area constant and equal to $|t_0|\neq 0$.

Consider a regular Cauchy development $(\Phi,\Psi,\zeta)$ on $\Tbb^3\times[t_0,t_1]$ of these initial data for the chosen Einstein--Euler or Einstein--Navier--Stokes model in the areal foliation, with an areal time interval $I = [t_0,t_1] \subset \RR\setminus\{0\}$.
\bei
\item Assume that the initial spatial volume, the initial energy norm~$\mathcal E(t_0)$, and the bounded-variation norm of $\mathring\Psi=(\mathring\ell,\log\mathring\Omega,\mathring\Kpar)$ on~$\Tbb^3$ are bounded by a constant~$\mathcal I_0$.

\item In the future-contracting regime, assume that the conformal length has a positive lower bound $\mathcal L_{\min}$ on $I$; in the future-expanding regime, this lower bound is $\mathcal L_{\min}=\mathcal L(t_0)>0$, which is part of the initial data controlled by~$\mathcal I_0$.

\item Assume either that the fluid velocity is orthogonal to the symmetry orbits or, more generally, that the parallel momentum per particle is uniformly bounded in the sense that~\eqref{intro-parall-perpart-bounded} holds at $t=t_0$ for some constant $C^\parallel>0$.
\eei
Then the following Cauchy stability properties hold.

\begin{enumerate}

\item The parallel momentum per particle remains uniformly bounded at all times by~$C^\parallel$ in the sense that~\eqref{intro-parall-perpart-bounded} holds on $[t_0,t_1]\times\Tbb^3$.

\item There exists a constant
\be
C=C\bigl(\mathcal I_0,t_0,t_1,\mathcal L_{\min}^{-1},C^\parallel\bigr)
\ee
such that the $L^2$ norms of~$\Phi$ and bounded-variation norms of~$\Psi$ on areal time slices (constant~$t$) and on timelike slices orthogonal to the areal foliation (constant~$x$ in the \JKL{} formulation) are controlled as
\bel{equajKD93}
\aligned
& \sup_{t\in I}\bigl(
\|\Phi\|_{L^2(\Tbb^3,\dVtrois)}
+ \|\Psi\|_{\BV(\Tbb^3)}
\bigr)
\\
& + \sup_{x\in\Sbb^1}\bigl(
\|\Phi\|_{L^2(\Interval\times\Tbb^2,\dVundeux)}
+ \|\Psi\|_{\BV(\Interval\times\Tbb^2)}
\bigr)
+\int_{I\times\Tbb^3}\!\!\mathfrak d_*\,\dVuntrois
\leq C.
\endaligned
\ee
Here, $\dVuntrois,\dVtrois,\dVundeux$ are volume forms of the (Lorentzian) metric~$g$ and of (Riemannian) induced metrics on spatial and temporal slices.

\item The last term in \eqref{equajKD93} represents the total non-negative production measure and vanishes for the Einstein--Euler system.  For the particle-production Einstein--Navier--Stokes system it is the produced particle number density $\mathfrak d_*=\zeta_{\alpha\beta}\mathscr B^{\alpha\beta}$, while for the dissipated-energy Einstein--Navier--Stokes system it is the dissipated mass-energy density $\mathfrak d_*=\zeta_{\alpha\beta}\mathscr B^{\alpha\beta}/(-n_\alpha\beta^\alpha)$.

\item Moreover, these conclusions allow vacuum, arbitrarily large mass-energy density, and arbitrary finite rapidity. For the two Einstein--Navier--Stokes systems, the constant $C$ is independent of the relaxation rate map~$\mathscr B$.
\end{enumerate}
\end{theorem}

\autoref{part-two} is this paper is devoted to the proof of this theorem, and is outlined next.


\subsection{Organization of the paper}
\label{section-1-3}

\subsubsection{\autoref{part-one}: \JKL{} formulation of finite-energy Einstein areal flows} 

\autoref{part-one} consists of Sections~\ref{section-2} to~\ref{section-6}. 
 
\begin{itemize}

\item \autoref{section-2} introduces the areal gauge, the adapted frame, the first-order geometry and matter variables, and the pressure assumptions at vacuum and at large density.

\item \autoref{section-3} defines the \JKL{} operators, the flow notation $(\Phi,\Psi)$, and the Einstein--Euler evolution--constraint system.

\item \autoref{section-4} constructs the particle current and proves the convexity of the two reference currents, leading to the particle-production and dissipated-energy mathematical entropy inequalities.

\item \autoref{section-5} introduce two classes of Einstein--Navier--Stokes system, based on a notion of Navier--Stokes potential denoted by $\chi$. Remarkably, 
the Einstein, Bianchi, and constraint equations are all satisfied, and a mathematical entropy structure is readily available. 

\item \autoref{section-6} carries out the torus reduction and determines the \JKL{} equations for the proposed Einstein--Navier--Stokes models. 

\end{itemize}


\subsubsection{\autoref{part-two}: Cauchy stability of finite-energy Einstein areal flows}
 
\autoref{part-two} consists of Sections~\ref{section-7} to~\ref{section-9}.

\begin{itemize}

\item \autoref{section-7} proves the Einstein--Euler maximum and bounded-variation principles, together with their extension to the
 dissipated-energy and particle-production systems.

\item \autoref{section-8} derives geometric monotonicity properies, spacelike and timelike energy estimates for~$\Phi$, and bounded-variation estimates for~$\Psi$, and also analyzes total-stress positivity for both selected systems.

\item \autoref{section-9} presents some perspectives.

\end{itemize}

\autoref{section-A} derives the \JKL{} equations and their structural features. \autoref{section-C} establishes the spacelike and timelike trace results used at weak regularity.
In addition, we present the notation used in this paper in three tables, encompassing the main notation for the geometric variables, the matter variables, and the Einstein flows, respectively.  

\begin{table}[H]
\centering
\small
\begin{tabular}{l p{0.64\textwidth}}
\toprule
$(\Mcal,\guntrois)$ & spacetime with signature $(-,+,+,+)$ and $\Tbb^2$ symmetry \\
$\Ncal= \Mcal/\Tbb^2$ & quotient manifold with coordinates $(t,x)$ \\
$\Omega,\amdeux$ & positive lapse and conformal length density \\
$\Pbb=(P_0,P_1),\ \Qbb=(Q_0,Q_1)$ & first-order modular variables \\
$\Lbb= \Pbb+i\Qbb,\ \Kpar=K_2+iK_3$ & complex geometry variable and twist variable\\
\bottomrule
\end{tabular}
\caption{Geometric variables.}
\label{table:notation-geometry}
\end{table}

\begin{table}[H]
\centering
\small
\begin{tabular}{l p{0.64\textwidth}}
\toprule
$\mu,p(\mu),q(\mu)$ & mass-energy density, pressure, and pressure ratio \\
$\Jbb=(J_0,J_1,J_2,J_3)$ & matter momentum, with $\Jperp=(J_0,J_1)$ and $\Jpar=J_2+iJ_3$ \\
$\Ebf_0,\Ebf_1$ & quadratic energy and energy-flux forms \\
$\Mbf^{00},\Mbf^{01},\Mbf^{11}$ & reduced stress-energy components\\
\bottomrule
\end{tabular}
\caption{Matter variables.}
\label{table:notation-fluid}
\end{table} 

\begin{table}[H]
\centering
\small
\begin{tabular}{l p{0.64\textwidth}}
\toprule
$\Phi,\Psi$ & Einstein--Euler variables \\
$\opL,\opJ,\opC,\opO,\opK$ & geometry, fluid, length, lapse, and twist operators \\
$\divuntrois,\curluntrois$ & weighted quotient divergence and curl \\
$\dVuntrois,\dVtrois$ & spacetime and spatial-slice volume forms\\
\bottomrule
\end{tabular}
\caption{Flow notation.}
\label{table:notation-operators}
\end{table}


\

\part{\JKL{} formulation of finite-energy Einstein areal flows} 
\label{part-one}

\section{A choice of geometry and matter variables}
\label{section-2}

\subsection{Areal geometry and an adapted frame}
\label{section-2-1}

\subsubsection{The geometric normalization}

This section introduces the geometric and matter variables used throughout the paper. We first describe the areal gauge and the quotient geometry, then construct an adapted orthonormal frame and calculate its spin connection. We next replace the original fluid variables by a matter momentum which remains meaningful at vacuum, and we gather all notation associated with the pressure law and the quadratic energy forms. The first-order equations themselves are stated in \autoref{section-3}.


\subsubsection{Global areal time}

Throughout, we consider a spacetime $(\Mcal,\guntrois)$ that is $\Tbb^2$-symmetric with spatial topology $\Tbb^3$. By $\Tbb^2$-symmetry, we mean that the spacetime is endowed with two commuting spacelike Killing fields with closed orbits. In particular, the symmetry orbits are two-tori, and the quotient by the symmetry is globally well-defined. We foliate this spacetime by spacelike hypersurfaces of constant areal time, denoted by~$t$. Namely, we introduce a global coordinate $t\colon \Mcal \to \Interval \subset \RR\setminus\{0\}$ that coincides, up to a sign, with the area of the $\Tbb^2$-orbits of symmetry. Here, $\Interval$ denotes a compact or non-compact interval that does not contain~$0$. Most of our definitions and results will be stated on a compact interval $\Interval=[t_0, t_1]$; however, our definition extends immediately to semi-open intervals $\Interval = [t_0, t_*) \subset (0, + \infty)$ or $\Interval= [t_0, t_*) \subset (- \infty, 0)$. Under the standard non-degeneracy and dominant energy condition \cite{LeFlochRendall-2011}, Einstein's constraint equations imply that the gradient of the orbit-area function is timelike; this is the usual motivation for seeking a global areal foliation (cf.~\cite{Rendall-book} and the references cited therein). We work throughout on such an areal development and choose the sign of~$t$ so that $t$ increases toward the future. With our signature and the metric convention~\eqref{quotient-metric-gN} below, this is equivalently the requirement that $-\nabla t$ be future-directed. Hence, positive and negative values of~$t$ correspond to \emph{future-expanding} and \emph{future-contracting} spacetimes, respectively. In this section, we assume that all objects under consideration are sufficiently regular.


\subsubsection{Modular parameter}

By symmetry, each $\Tbb^2$~orbit is isometric to a discrete quotient of~$\CC$ by a rank~$2$ lattice. Thus, besides the area~$|t|$, the metric induced on each orbit is characterized {\cadmiumgreen by} the \emph{modular parameter}, denoted by $\tau\in\CC$ with $\Im\tau>0$, which describes the conformal structure as that of the discrete quotient $\CC/(\ZZ+\tau\ZZ)$.  The parameter is defined up to discrete transformations 
\be
\tau\mapsto \frac{a\tau+b}{c\tau+d} \quad \text{with }  \begin{pmatrix}a&b\\c&d\end{pmatrix} \in SL(2,\ZZ), 
\ee
which can be gauge-fixed by choosing a pair of reference one-cycles inside~$\Tbb^2$ (with intersection number~$1$). We focus our attention on spacelike hypersurfaces with $\Tbb^3$ topology, which ensures a global choice of such cycles; thus the modular parameter $\tau$ is periodic in space.  We decompose 
\be
\tau \eqqcolon Q+ie^{-P}
\ee
into its real and positive imaginary parts. Although we assume $\Tbb^3$ topology, our method is general and should also apply to the case where Cauchy surfaces are realized as certain non-trivial $\Tbb^2$ bundles over the circle~$\Sbb^1$, including the comparatively simpler case of (vacuum) \emph{twisted Gowdy spacetimes} introduced in~\cite{Rendall-twisted}. In that setting, the spatial periodicity of~$\tau$ is replaced by periodicity up to an $SL(2,\ZZ)$ transformation.


\subsubsection{Lapse function and conformal length density}

We consider the (two-dimensional) \emph{quotient manifold}
\be
\Ncal \coloneqq \Mcal/\Tbb^2
\ee
and the associated \emph{quotient metric}~$g_{\Ncal}$, induced by the four-dimensional metric through the action of the symmetry group~$\Tbb^2$. By construction, the areal time~$t$ remains constant along the $\Tbb^2$-orbits, and thus serves as a time coordinate on this $(1+1)$-dimensional Lorentzian manifold. We define the \emph{conformal factor} $\Omega>0$ through the squared norm of~$dt$, namely
\be
\Omega^{-2} \coloneqq - g_{\Ncal}^{-1}(dt,dt).
\ee
In spatial coordinates adapted to the areal foliation, the quotient metric $g_{\Ncal}$ takes the form
\bel{quotient-metric-gN}
g_{\Ncal} = \Omega^2 (- dt^2 + \amdeux^2 dx^2).
\ee

More precisely, once the time coordinate~$t$ has been fixed, the spatial coordinate~$x$ is defined, up to an overall constant shift and orientation, by choosing a parametrization of the initial quotient circle $\{t=t_0\}\simeq \Tbb^3/\Tbb^2\simeq \Sbb^1$ and requiring that $x$ be periodic with period~$1$. Equivalently, we may prescribe the initial profile $\amdeux|_{t=t_0}$, which determines the spatial gauge\footnote{We may in particular choose $\amdeux|_{t=t_0}$ to be constant. In the vacuum case under Gowdy symmetry, this normalization is especially natural, since $\amdeux$ then remains constant throughout the evolution. In general, however, $\amdeux$ evolves nontrivially, so this choice is merely an initial gauge normalization and depends on the selected initial slice.} on the initial slice. In the regularity classes considered below, we work with areal coordinates $(t,x)$ for which the quotient metric admits the representation \eqref{quotient-metric-gN}. The conformal coefficient $\Omega$ is commonly referred to as the \emph{lapse function} of the foliation, while we call the function $\amdeux$ the \emph{conformal length density}. The null directions of the quotient metric satisfy
\be
\frac{dx}{dt}= \pm \amdeux^{-1},
\ee
so that $\amdeux$ plays the role of an inverse local propagation speed in the quotient geometry. For this reason, it appears as a characteristic speed in the differential operators governing the geometric variables.


\subsubsection{Adapted moving frame}

We now construct an \emph{adapted moving frame}, namely an \emph{orthonormal basis} of vector fields on~$\Mcal$,
\be
e_0,e_1,e_2,e_3, \text{ invariant under the } \Tbb^2\text{ symmetry,}
\ee
as follows. First, $e_0$ is the future-directed unit vector normal to the constant-$t$ hypersurfaces, while $e_1$ is a unit vector orthogonal to $e_0$ and to the $\Tbb^2$~symmetry orbits. Next, $e_2$ is a unit vector tangent to the $\Tbb^2$~orbits (hence orthogonal to $e_0$ and $e_1$) and aligned with the first reference one-cycle mentioned above. Finally, $e_3$ is a unit vector orthogonal to the other three basis vectors. For any given spacetime with $\Tbb^2$~symmetry on~$\Tbb^3$, the choice of adapted moving frame depends only on \emph{discrete data:} the cycle of~$\Tbb^2$ along which $e_2$ is aligned (with possible choices differing by $SL(2,\ZZ)$ transformations), and the orientations of $e_3$ and~$e_1$. The modular parameter~$\tau$ depends also on the second reference cycle.


\subsubsection{Metric decomposition}

In the so-called areal gauge, with the coordinates $t,x$ introduced above and after choosing angular coordinates {$y,z$} along the closed symmetry orbits, any $\Tbb^2$-symmetric spacetime metric $\guntrois$ on the torus $\Tbb^3$ can be written in the form 
\bel{metric:areal}
\guntrois = \Omega^2 (- dt^2 + \amdeux^2 dx^2) + \abs{t} e^P \bigl( dy + Q\,dz + (G+QH) \, dx\bigr)^2 + \abs{t} e^{-P} (dz + H\,dx)^2, 
\ee  
where the metric coefficients $P,Q,\Omega, \amdeux, G, H$, with $\Omega,\amdeux>0$, depend only on $t \in \Interval$ and $x \in \Sbb^1 \simeq [0,1]$, whereas the remaining `parallel' variables $y,z$ describe coordinates on $\Tbb^2 \simeq [0,1]^2$. Observe that the vector fields $\del/\del y$ and $\del / \del z$ are Killing fields; for the existence of this metric decomposition, we refer to Rendall~\cite{Rendall-book}. In our analysis, we will sometimes find it convenient to specialize the general results to the following important cases. 
\bei 

\item  \emph{Gowdy-symmetric spacetimes} are characterized geometrically by vanishing twists (introduced below). In the areal gauge used here, after a choice of orbit coordinates, we represent this subclass by\footnote{In view of \eqref{eq:221a}--\eqref{eq:221b} below, the vanishing of twists also implies the vanishing of the parallel fluid variables~$J_2,J_3$.} $G=H=0$ throughout the spacetime. In addition, \emph{polarized} Gowdy-symmetric spacetimes are further characterized by $Q=0$.

\item \emph{Vacuum $\Tbb^2$-symmetric spacetimes} are defined by setting $\mu = 0$ throughout the spacetime, in which case the Euler equations are automatically satisfied for an arbitrary (and physically irrelevant) velocity field~$u$. 

\eei


\subsection{First-order variables for the quotient geometry}
\label{section-2-2}

\subsubsection{A choice of frame}

While the present work focuses mainly on regular solutions, we express the field equations for the metric~\eqref{metric:areal} in a robust form that is meaningful for finite-energy (weakly regular) solutions and whose nonlinearities can converge along weakly convergent limits~\cite{TorusT2-1-2026,TorusT2-4-2026}. First of all, we rewrite the metric and its inverse as 
\bel{equa-metric-inverse}
\guntrois_{\alpha\beta} = - e^0_\alpha e^0_\beta + e^1_\alpha e^1_\beta + e^2_\alpha e^2_\beta + e^3_\alpha e^3_\beta,
\qquad 
\guntrois{}^{\alpha\beta} 
 = - e_0^\alpha e_0^\beta + e_1^\alpha e_1^\beta + e_2^\alpha e_2^\beta + e_3^\alpha e_3^\beta, 
\ee 
namely in terms of the adapted moving frame defined above, also called \emph{Vielbein} and inverse Vielbein, respectively. 

In areal gauge, the Vielbein one-forms and vectors associated with the decomposition~\eqref{equa-metric-inverse} of the spacetime metric~$\guntrois$ take the explicit form
\bel{eq:ourframe}
\aligned
e^0 & \coloneqq \Omega\,dt, \
& e^1 & \coloneqq \Omega\, \amdeux dx, \
\\
e^2 & \coloneqq |t|^{1/2} e^{P/2} \bigl( dy + Q\,dz + (G+QH) \, dx\bigr), \
& e^3 & \coloneqq |t|^{1/2} e^{-P/2} (dz + H\,dx),
\\
e_0 & \coloneqq \Omega^{-1} \del_t, \
& e_1 & \coloneqq \Omega^{-1} \amdeux^{-1} (\del_x - G\del_y - H\del_z), \ 
\\
e_2 & \coloneqq |t|^{-1/2} e^{-P/2} \del_y, \
& e_3 & \coloneqq |t|^{-1/2} e^{P/2} (\del_z - Q \del_y).
\endaligned
\ee 
Denoting by $m,n, \ldots= 0, 1, 2, 3$ the frame indices corresponding to $(e_m)$, we observe that the components $\omega^{mn}_p \coloneqq e_p^\alpha e_\beta^m\nabla_\alpha e^{\beta n}=- \omega^{nm}_p$ of the \emph{spin connection} take a particularly simple form in terms of the variables introduced in~\eqref{eq:definevar-01} below: 
\bse\label{equa-touslesomega}
\be
\aligned
& \omega^{20}_2 = \frac{1}{2} \Bigl(P_0+ \frac{1}{t} \Omega^{-1} \Bigr), \,
&& \omega^{03}_3 = \frac{1}{2} \Bigl(P_0- \frac{1}{t} \Omega^{-1} \Bigr), \, 
&& \omega^{12}_2 = \omega^{31}_3 = \frac{1}{2} P_1, 
\\
& \omega^{20}_3 = \omega^{30}_2 = \omega^{23}_0 = \frac{1}{2} Q_0, \,
&& \omega^{12}_3 = \omega^{13}_2 = \omega^{23}_1 = \frac{1}{2} Q_1, \, 
&& 
\\
& \omega^{10}_2 = \omega^{20}_1 = \omega^{21}_0 = \frac{1}{2} K_2, \quad 
&& \omega^{10}_3 = \omega^{30}_1 = \omega^{31}_0 = \frac{1}{2} K_3, \quad 
\\
& \omega^{01}_0 = \amdeux^{-1}  (\Omega^{-1})_x, \quad
&&
\omega^{01}_1 = \amdeux (\Omega^{-1} \amdeux^{-1} )_t. 
\endaligned
\ee
All other coefficients not determined by the above relations and by antisymmetry vanish, namely
\be
\aligned   
&& \omega^{02}_0 = \omega^{03}_0 = \omega^{12}_1 = \omega^{13}_1 = \omega^{23}_2 = \omega^{23}_3 = 0.
\endaligned
\ee
\ese

\subsubsection{A choice of geometry variables}

Motivated by components of the spin connection, we introduce the notation 
\bel{eq:definevar-01}
\aligned 
P_0 & \coloneqq e_0(P) = \Omega^{-1} P_t, 
\,
& Q_0 & \coloneqq e^P e_0(Q) = \Omega^{-1} e^P Q_t, 
\\
P_1 & \coloneqq e_1(P) = \Omega^{-1} \amdeux^{-1} \, P_x,  
\,
& Q_1 & \coloneqq e^P e_1(Q) = \Omega^{-1}\amdeux^{-1} \, e^P Q_x,  
\\
K_2 & \coloneqq \Omega^{-2}   \amdeux^{-1}  \, |t|^{1/2} e^{P/2} \bigl ( G_t + QH_t \bigr),  
\, \qquad
& K_3 & \coloneqq \Omega^{-2}  \amdeux^{-1} \, |t|^{1/2} e^{-P/2} H_t,
\endaligned
\ee  
in which $K_2, K_3$ are referred to as the \emph{twist functions}, also defined geometrically as \eqref{equa-K2K3} below. 


The geometric unknowns used throughout our work to express the Einstein equations are as follows: the lapse~$\Omega$ and the conformal length density~$\amdeux$ (which describe the quotient metric~\eqref{quotient-metric-gN}), the two one-forms $\Pbb$ and $\Qbb$ (whose components in the Vielbein basis are $(P_0, P_1)$ and $(Q_0, Q_1)$), and a complex scalar $\Kpar$ (associated with $K_2$ and $K_3$). Specifically, the \emph{first-order geometry variables} are the triple $(\Kpar,\Pbb,\Qbb)$, where
\bse\label{eq:definevar-KL}
\bel{eq:definevar}
\aligned
\Pbb & \coloneqq (P_0, P_1) = e_0(P) e^0 + e_1(P) e^1 = dP , 
\\
\Qbb & \coloneqq (Q_0, Q_1) = e^P e_0(Q) e^0 + e^P e_1(Q) e^1 = e^P dQ ,
\endaligned
\ee
together with the \emph{twist functions} 
\bel{equa-K2K3}
K_2  \coloneqq  \guntrois\bigl(e_2,[e_1,e_0]\bigr), 
\quad
K_3  \coloneqq  \guntrois\bigl(e_3,[e_1,e_0]\bigr),
\ee
\ese
where $[\, \cdot\,,\, \cdot\,]$ denotes the commutator of two vector fields on~$\Mcal$.
Throughout, we tacitly assume the \emph{vanishing integral conditions}
\bel{equa-P1Q1}
\int_{\Tbb^3} P_1 \, \dVtrois = \int_{\Tbb^3} e^{-P} Q_1 \, \dVtrois = 0 \quad \text{  at  } t \in \Interval,  
\ee 
where $\dVtrois = |t|\,\Omega\,\amdeux\,dx\,dy\,dz$. These conditions are required for periodicity of $P,Q$. Using the evolution equations derived below, we easily check that these conditions hold for all times provided they hold on an initial Cauchy surface

\subsubsection{Complex-valued variables}

We may also introduce the following \emph{complex-valued} geometry variables: a \emph{one-form}~$\Lbb$ and the \emph{scalar field}~$\Kpar$ on~$\Ncal$, defined by
\be
\aligned
\Lbb & \coloneqq \Pbb + i\,\Qbb = \frac{i\,d\tau}{\Im\tau},
\qquad\quad
\Kpar \coloneqq K_2 + iK_3 = \guntrois\bigl(e_2+ie_3,[e_1,e_0]\bigr). 
\endaligned
\ee 
Recall that $\tau=Q+ie^{-P}$ is the modular parameter.  With some abuse of notation, we also write 
\be
\Lbb  \coloneqq  (\Pbb, \Qbb), \qquad \Kpar = (K_2, K_3). 
\ee


\subsection{Pressure law and quadratic fluid variables}
\label{section-2-3}

\subsubsection{Pressure law}

The matter model is determined by a pressure law $p=p(\mu)$, where $\mu\geq 0$ is the mass-energy density. We assume that $p\in C^1([0,+\infty))\cap C^2((0,+\infty))$, together with  
\bse\label{hyperbolic-eos-zero}
\bel{hyperbolic-eos}
p(0)=0, \qquad
0\leq p'(0)<1, \qquad
0<p'(\mu)<1 \quad \text{for } \mu>0 ,
\ee 
which express \emph{strict hyperbolicity away from vacuum} and the fact that the sound speed stays \emph{below the speed of light}, normalized to unit. We also impose two asymptotic conditions, one near vacuum and one at large density, stated as follows.
\bei

\item 
\emph{Asymptotically polytropic or isothermal near vacuum.} We assume that, for some $\gamma>1$,
\bel{equa-asym}
\lim_{\mu\to 0}\frac{p(\mu)}{\mu^\gamma} \in (0, +\infty)
\quad\text{ or }\quad
p'(0) \in (0,1).
\ee
\item
\emph{Asymptotically isothermal at large density.} 
We assume that
\bel{hyperbolic-eos-2}
\lim_{\mu\to+\infty}p'(\mu) \in (0,1).
\ee
\eei
\noindent
The isothermal law $p(\mu)=k^2\mu$, with $k\in(0,1)$, is the main special case considered later; throughout, the phrase \emph{admissible pressure law} refers to a law satisfying the hyperbolicity condition and the two asymptotic conditions above. 
\ese 
%


\subsubsection{A choice of fluid variables}

\bse\label{Jdef}
Throughout, the matter flow is assumed to enjoy the same $\Tbb^2$~symmetry as the spacetime, so that the fluid unknowns depend on the variables $t,x$ only. However, \emph{all four components} of the velocity vector may be non-vanishing. Using the orthonormal frame $(e_0, e_1, e_2, e_3)$, which was defined above in geometric terms and is adapted to the assumed symmetry, we decompose the \emph{fluid velocity} $u$ as $u=u^m e_m$, with frame components $(u^0,u^1,u^2,u^3)$. We define the \emph{matter momentum one-form} $\Jbb$, by lowering the index with the metric~$\guntrois$ and rescaling the velocity vector by the scalar factor $\sqrt{2(\mu+p)}$, namely
\bel{appendixJmatter}
\Jbb \coloneqq (J_0, J_1, J_2, J_3), 
\qquad
J_m \coloneqq \sqrt{2(\mu+p)} \, u_m, 
\ee
Hence, we have 
\bel{equa-211b}
- \Jbb \cdot \Jbb = J_0^2 - J_1^2 - J_2^2 - J_3^2 = 2 (\mu+p(\mu)). 
\ee
Since $\mu+p(\mu) \geq 0$, this field obeys the \emph{fluid causality inequalities}
\bel{eq:J0J} 
J_0 \leq 0,
\qquad 
- \Jbb \cdot \Jbb \coloneqq J_0^2 - J_1^2 - J_2^2 - J_3^2 \geq 0, 
\ee
namely, $\Jbb$ is future-oriented and causal, possibly vanishing. 

\emph{By extension,} our theory is formulated directly in terms of the current $\Jbb$, which continues to make sense in the degenerate vacuum regime $\mu=0$, even though the velocity field $u$ is then no longer defined. This extension is required for stability under evolution and weak convergence.

In contrast to the geometric variables, which primarily reside on the quotient manifold~$\Ncal$, the one-form $\Jbb$ is naturally defined on the full spacetime~$\Mcal$.  When projected onto~$\Ncal$, the one-form decomposes into components that are \emph{transverse} or \emph{parallel} to the $\Tbb^2$~orbits of symmetry, namely, a (real) one-form $\Jperp$ and a complex scalar field $\Jpar$, defined as
\be
\Jperp \coloneqq (J_0, J_1), \qquad \Jpar \coloneqq J_2+iJ_3. 
\ee
\ese
The formulation of the Einstein--Euler system proposed here is based on the fluid and geometric variables $(\Jbb,\Kpar,\Lbb)$ together with $(\Omega,\amdeux)$. 

Through~\eqref{equa-211b}, the density is uniquely recovered from the matter momentum vector as
\bel{eq:section-2-density-change}
\mu = y^{-1}(-\Jbb\cdot\Jbb) , \qquad y(\mu)\coloneqq2(\mu+p(\mu)) ,
\ee
in terms of a map $y$ that is a strictly increasing bijection from $[0,+\infty)$ onto itself and that has a continuously differentiable inverse on $[0,+\infty)$.  Indeed, $y(0)=0$, $y(\mu)\to+\infty$, and $y'(\mu)=2(1+p'(\mu))>0$ for $\mu\geq 0$.

The pressure ratio used below is the function
\bse
\bel{eq:defq}
q(\mu) \coloneqq \frac{\mu-p(\mu)}{\mu+p(\mu)} \in (0,1), \qquad \mu>0,
\ee
and we write $q_\Jbb=q(\mu)$ when the density is recovered from the momentum vector  through~\eqref{equa-211b}. Moreover, we define
\be
q(0) \coloneqq \lim_{\mu\to 0}q(\mu)=\frac{1-p'(0)}{1+p'(0)}\in(0,1],
\ee
\ese
so $q(0)=1$ whenever $p'(0)=0$, as in the polytropic case. The asymptotic assumptions imply the uniform lower bound
\bel{equa-qminqmax-1}
0 < q_{\min} \coloneqq \inf_{[0,+\infty)} q \leq q(\mu) \quad \text{for all }\mu\in[0,+\infty),
\ee
and the high-density upper bound
\bel{equa-qminqmax-2}
q(\mu) \leq q_\infty^+ < 1 \quad \text{for all sufficiently large }\mu.
\ee
These properties follow directly from the pressure hypotheses. Indeed, $p(0)=0$ and $0<p'(\mu)<1$ imply $0<p(\mu)<\mu$ for $\mu>0$ so $q(\mu)\in(0,1)$. The two alternatives in~\eqref{equa-asym} imply $q(0)\in(0,1]$, while $p'(\mu)\to k_\infty^2$ implies $p(\mu)/\mu\to k_\infty^2$ and therefore $q(\mu)\to(1-k_\infty^2)/(1+k_\infty^2)\in(0,1)$ as $\mu\to+\infty$. The preceding limits, together with continuity on compact density intervals, give the stated bounds.


\subsubsection{Quadratic forms}

We now introduce nonlinear quantities that play a crucial role in our analysis below. For any one-form~$\Xbb$, chosen among $\Pbb, \Qbb, \Jperp$, we define 
\bse
\label{eq:E0E1-T2-3lignes}
\bel{eq:E0E1-T2}
\aligned
\Ebf_0(\Xbb) & \coloneqq \frac{1}{2} (X_0^2 + X_1^2),
& \qquad
\Ebf_1(\Xbb) & \coloneqq X_0 X_1. 
\endaligned
\ee
It is convenient to also set $\Ebf_0(\Lbb)= \Ebf_0(\Pbb,\Qbb)$ and use the shorthand notation $\Ebf_0(\Xbb, \Ybb) \coloneqq \Ebf_0(\Xbb) + \Ebf_0(\Ybb)$, and similarly for additional arguments and for the quadratic form~$\Ebf_1$. For $\Xpar=X_2+iX_3$, chosen among $\Kpar, \Jpar$, we use
\bel{equa-Ezero-etc}
\Ebf_0(\Xpar) \coloneqq \frac{1}{2}(X_2^2 + X_3^2), \qquad
\Re(\Xpar^2)=X_2^2-X_3^2, \qquad \Im(\Xpar^2)=2 X_2X_3.
\ee
We also write
\be
\Ebf_0(\Jbb) \coloneqq \Ebf_0(\Jperp) + \Ebf_0(\Jpar) = \frac{1}{2}(J_0^2 + J_1^2 + J_2^2 + J_3^2).
\ee
We collectively refer to these expressions as energy and energy-flux terms.
\ese
\bse 
We will also use the following \emph{null} quadratic forms in $\Xbb, \Ybb$, chosen among $\Pbb, \Qbb, \Jperp$, 
\be 
- \Xbb \cdot \Ybb \coloneqq X_0 Y_0 - X_1 Y_1, 
\qquad 
\Xbb \wedge \Ybb \coloneqq X_0 Y_1 - X_1 Y_0.
\ee 
Finally, for the matter momentum $\Jbb$ and its orthogonal part $\Jperp$, we have
\bel{equa-2122}
- \Jperp \cdot \Jperp = J_0^2 - J_1^2 \geq - \Jbb \cdot \Jbb 
= J_0^2 - J_1^2 - J_2^2 - J_3^2 \geq 0.
\ee 
\ese
%


\section{The \JKL{} formulation of Einstein--Euler areal flows}
\label{section-3}

\subsection{Evolution equations and constraints}
\label{section-3-1}

\subsubsection{Divergence and curl operators}

\bse
\label{equa-div-curl}
To proceed, we introduce notation for several operators defined on $(\Mcal, \guntrois)$. We consider one-form fields $\Xbb= (X_0,X_1)$ on the quotient manifold~$\Ncal$, and we introduce their \emph{divergence with respect to the quotient metric}~$g_{\Ncal}$ 
\bel{equa-divN} 
\divN(\Xbb)
\coloneqq \amdeux^{-1} \Omega^{-2}  \bigl(-( \amdeux \Omega X_0)_t + (\Omega X_1)_x\bigr).  
\ee
The divergence operator depends only on the metric coefficients $(\Omega,\amdeux)$. This formula differs slightly from the divergence with respect to the four-dimensional metric. The \emph{divergence with respect to the spacetime metric} $\guntrois$ reads 
\bel{equa-def-div13}
\divuntrois \Xbb \coloneqq |t|^{-1} \divN(|t| \, \Xbb).
\ee
The factor of $|t|$ arises from the $\Tbb^2$~directions, with an absolute value to account for both time orientations. As usual, we implicitly identify scalar fields on~$\Ncal$ with $\Tbb^2$-invariant scalar fields on~$\Mcal$. We are using the same notation $\Xbb$ (by a mild abuse) for the one-form field on~$\Mcal$ with the same components $X_0,X_1$ and $X_2=X_3=0$ in the Vielbein basis. 
\ese

We identify vector fields and one-forms on the quotient manifold~$\Ncal$ using the metric~$g_{\Ncal}$. This identification maps a one-form $\Xbb$ with components $X_0,X_1$ in the inverse Vielbein basis to a vector with the same components in the Vielbein basis, up to a sign, namely ${X^0=-X_0}$ and ${X^1=X_1}$. Hence, the divergence operator can be applied equally well to vector fields or one-forms. When working with contravariant components, the time derivative thus has the \emph{opposite sign} from the one-form formula~\eqref{equa-divN}.  Explicitly,
\bel{divuntrois-vector}
\divuntrois \Xbb = |t|^{-1} \amdeux^{-1} \Omega^{-2}  \bigl( ( |t|\amdeux \Omega X^0)_t + (|t|\Omega X^1)_x\bigr).
\ee

\bse\label{equa-def-curl-0}
To the one-form field $\Xbb$, we associate its differential $d\Xbb$, a two-form field, which has a single non-vanishing component, namely
\be
d\Xbb= \bigl( \curl_\Ncal  \Xbb \bigr) \, e^0\wedge e^1, 
\ee
which leads us to define $\curl_\Ncal (\Xbb)= (d\Xbb)_{01}$. This is equivalent to applying the Hodge star operator ($X_0\leftrightarrow -X_1$) before taking the divergence, and yields the \emph{curl operator:}
\be
\curl_\Ncal\Xbb 
\coloneqq \amdeux^{-1} \Omega^{-2}  \bigl(( \amdeux \Omega X_1)_t - (\Omega X_0)_x\bigr),  
\ee
which, by analogy with the expression of the divergence above, leads us to set 
\bel{equa-def-curl}
\curluntrois \Xbb \coloneqq |t|^{-1} \curl_\Ncal( |t| \, \Xbb).
\ee
In the following, it will be convenient to use the spacetime notation $\divuntrois $ and $\curluntrois$ when stating the field equations. 
\ese
%


\subsubsection{Ricci and matter components in the adapted frame}

We denote by $(\eta_{mn})$ the Minkowski metric with signature $(-1, 1, 1, 1)$. In the adapted frame introduced above, the Einstein equations take the following form for $m,n = 0, 1, 2, 3$: 
\bel{eq:Einsframe}
\aligned
  e_m^\alpha R_{\alpha\beta} e_n^\beta 
& = 
e_m^\alpha T_{\alpha\beta} e_n^\beta - \frac{1}{4-2}\bigl(T_{\alpha\beta} {\guntrois{}^{\alpha\beta}}\bigr)  \eta_{mn} 
\\
& = (\mu+p) \, u_m u_n + \frac{1}{2} (\mu-p) \eta_{mn}
= \frac{1}{2} J_m J_n + \frac{q_\Jbb}{4} (-\Jbb\cdot\Jbb) \eta_{mn} , 
\endaligned
\ee
where $R_{\alpha\beta}$ denotes the components of the Ricci curvature $\Ric$. When expressed in terms of the vector~$\Jbb$, the last term $\frac{1}{2} J_m J_n$ \emph{does not} explicitly depend on the equation of state, whereas the second term still depends on the equation of state. 


Moreover, we define the tensor field
\be
\Mbf= \Mbf(\Jbb, \Kpar, \Lbb) = \Mbf(\Jperp, \Jpar, \Kpar, \Pbb, \Qbb), 
\ee
which incorporates both geometric and matter contributions, via its components\footnote{Observe that $\Mbf^{01}(\Jbb, \Kpar,\Lbb)$ does not depend on~$\Kpar$, hence this argument will be freely omitted.}
\bel{eq:T2-Mdef-0}
\begin{alignedat}{2}
\Mbf^{00} & \coloneqq \Ebf_0(\Jbb, \Kpar, \Lbb) + \frac{q_\Jbb}{2}  \, (- \Jbb \cdot \Jbb)
&& = \Ebf_0(\Jperp, \Pbb, \Qbb) + \Ebf_0(\Jpar, \Kpar) + \frac{q_\Jbb}{2}  \, (- \Jbb \cdot \Jbb), 
\\
\Mbf^{01} & \coloneqq - \Ebf_1(\Jperp, \Lbb) 
&& = - \Ebf_1(\Jperp) - \Ebf_1(\Pbb) - \Ebf_1(\Qbb),
\\
\Mbf^{11} & \coloneqq \Ebf_0(\Jperp, \Lbb) - \Ebf_0(\Jpar, \Kpar) - \frac{q_\Jbb}{2}  \, (- \Jbb \cdot \Jbb)
&& = \Ebf_0(\Jperp, \Pbb, \Qbb) - \Ebf_0(\Jpar, \Kpar) - \frac{q_\Jbb}{2}  \, (- \Jbb \cdot \Jbb),
\end{alignedat}
\ee
together with $\Mbf^{10} = \Mbf^{01}$. Observe that $\Mbf^{00} \geq 0$. Moreover, these functions obey the inequalities
\bel{eq:T2-Mdef-1}
2 \, |\Mbf^{01}|\leq \Mbf^{00}+\Mbf^{11}, \qquad \quad - \Ebf_0(\Kpar)\leq \Mbf^{11}\leq \Mbf^{00}.
\ee 
To write the Euler equations, we will also need the notation
\bel{eq:T2-Euler-perp2-def0}
\aligned
\Msource(\Jbb, \Kpar, \Lbb)
& \coloneqq \Ebf_0(\Jpar) + 5 \, \Ebf_0(\Kpar)
- \Pbb \cdot \Pbb -  \Qbb \cdot \Qbb - \Jperp \cdot \Jperp 
+ \frac{q_\Jbb}{2}  \, \Jbb \cdot \Jbb, 
\endaligned
\ee
which obeys the inequality (proven in~\autoref{section-A-5})
\be
\abs{\Msource}\leq 5 \, \Mbf^{00}.
\ee
 

\subsubsection{Perfect-fluid stress-energy tensor}

 The components of the stress-energy tensor in the orthonormal frame~$(e_m)_{m=0,1,2,3}$ are given by $T^{mn}= \frac{1}{2}J^m J^n + \frac{1}{4} (1-q_\Jbb)(-\Jbb\cdot\Jbb) \eta^{mn}$, and in particular for $m,n=0,1$ they coincide with the fluid part of $(1/2)\Mbf^{mn}$: 
\bel{eq:T2-Texpr}
\begin{alignedat}{2}
T^{00} & = \frac{1}{2} \Mbf^{00}(\Jbb, 0, 0) = \frac{1}{2} J_0^2 - \frac{1-q_\Jbb}{4} (- \Jbb \cdot \Jbb), \qquad
& \Mbf^{00}(\Jbb, \Kpar, \Lbb) & = 2 T^{00} + \Ebf_0(\Kpar, \Lbb),
\\
T^{01} & = \frac{1}{2} \Mbf^{01}(\Jbb, 0, 0) = - \frac{1}{2} J_0 J_1, \qquad
& \Mbf^{01}(\Jbb, \Kpar, \Lbb) & = 2 T^{01} - \Ebf_1(\Lbb),
    \\
    T^{11} & = \frac{1}{2} \Mbf^{11}(\Jbb, 0, 0) = \frac{1}{2} J_1^2 + \frac{1-q_\Jbb}{4} (- \Jbb \cdot \Jbb), \qquad
    & \Mbf^{11}(\Jbb, \Kpar, \Lbb) & = 2 T^{11} + \Ebf_0(\Lbb) - \Ebf_0(\Kpar).
  \end{alignedat}
\ee
Throughout this work, the fluid contributions appearing in the equations could equally well be expressed in terms of the components of the stress-energy tensor. This observation also indicates how the calculations extend to more general matter models, as discussed later on. Since the fluid mass-energy density may vanish on regions of the spacetime, it is essential for our theory to allow for such vacuum regions.


\subsubsection{Field equations}

In terms of the variables introduced in~\autoref{section-2-2}, we now exhibit the basic algebraic and differential structure of the Einstein equations.  The field equations form a coupled system\footnote{A single equation is not included in this list (namely \eqref{eq:waveconffac}, below), since it follows from the other equations and will not be included in our first-order formulation.} involving the following partial differential operators. 
\bei 

\item Evolution of the first-order geometry: 
\bse\label{eq:T2-1234-def}
\begin{align}
\aligned 
\opL_{00} & \coloneqq
\divuntrois(\Pbb)
-  \Qbb \cdot \Qbb + \frac{1}{2} \Re(\Kpar^2 + \Jpar^2),
\\ 
\opL_{01}  & \coloneqq \curluntrois (|t|^{-1} \Pbb),
\label{eq:T2-1234-def-a}
\endaligned
\\
\aligned 
\opL_{10}  
& \coloneqq \divuntrois(\Qbb)
+ \Pbb \cdot \Qbb + \frac{1}{2} \Im(\Kpar^2 + \Jpar^2),
\\ 
\opL_{11} & \coloneqq \curluntrois(|t|^{-1} \Qbb)
- |t|^{-1} \Pbb \wedge \Qbb.
\endaligned
\label{eq:T2-1234-def-b}
\end{align}%
\ese

\par\vspace{-\lastskip}

\item Evolution of the matter:
\bse
\label{eq:T2-Euler-perp-def}
\begin{align} 
\opJ^0 & \coloneqq \divuntrois \bigl( \Omega \, \Mbf^{0 \bullet}(\Jbb, \Kpar, \Lbb) \bigr)
+ \frac{1}{2t} \Msource(\Jbb, \Kpar, \Lbb),
\label{eq:T2-Euler-perp-def-a}
\\ 
\opJ^1 & \coloneqq \divuntrois \bigl( \Omega \, \Mbf^{1 \bullet}(\Jbb, \Kpar, \Lbb) \bigr),
\label{eq:T2-Euler-perp-def-b}
\\ 
\opJ^2 & \coloneqq \divuntrois \bigl(|t|^{1/2} J_2 \Jperp\bigr)
+ \curluntrois  \bigl(|t|^{1/2} K_2 \Pbb / 2 \bigr),
\label{eq:T2-fluid23-def-a}
\\ 
\opJ^3 & \coloneqq \divuntrois  \bigl(|t|^{1/2} J_3 \Jperp\bigr)
+ \curluntrois  \bigl(|t|^{1/2} ( K_2 \Qbb - K_3 \Pbb / 2)\bigr).
\label{eq:T2-fluid23-def-b}
\end{align}%
\ese

\par\vspace{-\lastskip}

\item Evolution of the speed, lapse, and twists:
\bse\label{eq:T2-all-suite-def}
\begin{align} 
\opC & \coloneqq (\amdeux)_t
 - \frac{t}{2}(\Mbf^{00}- \Mbf^{11})(\Jbb, \Kpar, \Lbb)\,\Omega^2 \amdeux,
\label{eq:T2-8-def}
\\ 
\opO_0 & \coloneqq (\log\Omega)_t
 + \frac{1}{4t} - \frac{t}{2}\Mbf^{11}(\Jbb, \Kpar, \Lbb)\,\Omega^2,
\label{eq:evollambda-def}   
\\ 
\opK_{20} & \coloneqq |t|^{-3/2} \bigl(  |t|^{3/2} K_2 \bigr)_t
 - \bigl( J_1 J_2 -  P_0 K_2/2 \bigr)\Omega,
\label{eq:T2-9101112-evol-def-a}
\\ 
\opK_{30} & \coloneqq |t|^{-3/2} \bigl( |t|^{3/2} K_3 \bigr)_t
 - \bigl( J_1 J_3 + P_0 K_3/2 - Q_0 K_2 \bigr)\Omega. 
\label{eq:T2-9101112-evol-def-b} 
\end{align}%
\ese

\par\vspace{-\lastskip}

\item Constraints for the lapse and twists:
\bse
\label{eq:theconstraints-def}
\begin{align}
\opO_1 & \coloneqq (\log\Omega)_x
 + \frac{t}{2} \, \Mbf^{01}(\Jperp, \Lbb)\, \Omega^2 \amdeux,
\label{eq:T2-567-def}
\\ 
\opK_{21} & \coloneqq ( K_2 )_x
 - \bigl( J_0 J_2 - P_1 K_2/2 \bigr) \, \Omega \amdeux, 
\label{eq:221a}
\\ 
\opK_{31} & \coloneqq ( K_3 )_x
 - \bigl( J_0 J_3  - Q_1 K_2 + P_1 K_3/2 \bigr) \, \Omega \amdeux. 
\label{eq:221b}
\end{align}
\ese

\par\vspace{-\lastskip}

\eei
We formulate the field equations as follows. The derivation of this first-order system is given in \autoref{section-A}.

\begin{definition}
\label{def-quadr}
With the operators defined above, the \textbf{\JKL{} formulation of Einstein--Euler areal flows} is the system
\bel{equa-9385}
\begin{alignedat}{2}
\opL_{00}=\opL_{01}=\opL_{10}=\opL_{11} &=0,\qquad&
\opJ^0=\opJ^1=\opJ^2=\opJ^3 &=0,
\\
\opC=\opO_0=\opK_{20}=\opK_{30} &=0,\qquad&
\opO_1=\opK_{21}=\opK_{31} &=0.
\end{alignedat}
\ee
These are the evolution equations and constraints in~\eqref{eq:T2-1234-def}--\eqref{eq:theconstraints-def}. Their unknowns are $(\Jbb,\Kpar,\Lbb)$ and $(\Omega,\amdeux)$, subject to the algebraic inequalities
\bel{eq:causal}
\amdeux>0, \quad \Omega>0, \quad J_0\leq 0, \quad -\Jbb\cdot\Jbb=J_0^2-J_1^2-J_2^2-J_3^2\geq 0,
\ee
as well as the periodicity constraints
\bel{equa-pericons}
\int_{\Tbb^3}P_1\,\dVtrois= \int_{\Tbb^3}e^{-P}Q_1\,\dVtrois=0.
\ee
\end{definition}

Observe that the extended vacuum condition $\mu=0$ is equivalent to $\Jbb\cdot\Jbb=0$, which means that the matter momentum is null or vanishes. For a regular state described in terms of $(\mu,u)$ in~\eqref{Jdef}, we necessarily have $\Jbb=0$ at vacuum; the finite-energy state space has a richer closure that allows non-zero null vectors~$\Jbb$. Observe also that the constraints~\eqref{eq:theconstraints-def} imply, by integration, the three periodicity conditions
\bel{equa-peri-OmK}
\int \Mbf^{01}(\Jbb, \Lbb) \, \dVtrois
= \int e^{P/2} J_2 J_0 \dVtrois
= \int \bigl( Q e^{P/2} J_2 + e^{-P/2} J_3 \bigr) J_0 \dVtrois = 0. 
\ee
The particular case of isothermal flows is discussed in~\autoref{section-3-3}.


\subsection{Areal flows}
\label{section-3-2}

\subsubsection{Einstein--Euler areal flows}

We group the unknowns according to the norms used in the estimates:
\bel{eq:section-2-5-flow-variables}
\Phi\coloneqq(\Jperp,\Jpar,\Pbb,\Qbb),
\qquad
\Psi\coloneqq(\ell,\log\Omega,\Kpar),
\qquad
d_x\ell\eqqcolon\amdeux\,dx.
\ee
The variables in~\(\Phi\) are controlled by quadratic energies on spacelike and timelike hypersurfaces, whereas the variables in~\(\Psi\) are related to transport, constraint, and gauge equations and are absolutely continuous or of bounded variation. Here \(\ell\) is a primitive of the positive measure \(d_x\ell\) on the quotient circle.

\begin{definition}
\label{def-einstein-euler-areal-flow}
A \textbf{regular Einstein--Euler areal flow} is a pair \((\Phi,\Psi)\) satisfying pointwise all the evolution equations and constraints in \autoref{def-quadr}, together with the causality conditions~\eqref{eq:causal} and the periodicity constraints~\eqref{equa-pericons}. In particular, one has
\be
\opJ^0=\opJ^1=\opJ^2=\opJ^3=0.
\ee
\end{definition}


\subsubsection{Finite-energy regularity}

Let \(\Interval\Subset\RR\setminus\{0\}\) be a compact areal interval. As we will demonstrate in the present paper, the natural norm for the principal variables is
\be
\label{equa-flow-energy-norm}
\|\Phi\|_{\mathrm{energy}}
\coloneqq
\|\Phi\|_{L^\infty(\Interval;L^2(\Tbb^3,\dVtrois))}
+
\|\Phi\|_{L^\infty(\Sbb^1;L^2(\Interval\times\Tbb^2,\dVundeux))}.
\ee
where $\dVtrois$ and $\dVundeux$ denotes the volume elements induced by~$g$ on horizontal and vertical hypersurfaces of constant $x$ and constant $t$ respectively. Writing \(\BVac=W^{1,1}\), a natural norm for the variables $\Psi$ will be found to be
\be
\label{equa-flow-variation-norm}
\begin{aligned}
\|\Psi\|_{\mathrm{BV}}
\coloneqq{}&
\|(\log\Omega,\Kpar)\|_{L^\infty(\Interval;\BVac(\Tbb^3))}
+\|(\log\Omega,\Kpar)\|_{L^\infty(\Sbb^1;\BVac(\Interval\times\Tbb^2))}
\\
&+\|\ell\|_{L^\infty(\Interval;\BVac(0,1))}
+\|\ell\|_{L^\infty((0,1);\BVac(\Interval))}.
\end{aligned}
\ee
We proved in~\cite{TorusT2-1-2026} that, under the finite-energy regularity assumptions, all the operator expressions in~\eqref{equa-9385} are well defined in the sense of distributions. The notion of weak solutions to the Einstein constraint equations will also be developed in this paper; cf.~\autoref{section-9}. 


\subsection{Structural properties and symmetry reductions}
\label{section-3-3}

\subsubsection{General flows with \texorpdfstring{$\Tbb^2$}{T2} symmetry}

The first-order formulation is meaningful, in particular, in vacuum regions. The following properties are established in~\autoref{section-A}. Propositions~\ref{theo-struct} and~\ref{theo-struct-bis} below, taken together, prove \autoref{theo-main-results}. The first establishes hyperbolicity for timelike~$\Jbb$ and identifies the quadratic and null structures, while the second proves equivalence, reconstruction, and propagation of the constraints.

\begin{proposition}[Structure of the first-order formulation: hyperbolicity and quadratic nonlinearities]
\label{theo-struct}
Consider the Einstein--Euler system under $\Tbb^2$ symmetry on $\Tbb^3$ for a general equation of state  satisfying~\eqref{hyperbolic-eos-zero}. After expressing the spacetime metric in areal gauge~\eqref{metric:areal}, consider the metric unknowns~$\amdeux, \Omega, P, Q, G, H$, together with the fluid unknowns $\mu, u$ entering the stress-energy tensor~\eqref{eq:Einsframe}. Assume furthermore that the solutions under consideration are sufficiently regular.
\bei

\item \emph{Hyperbolicity.}

\bei

\item \emph{Under the non-vacuum condition $- \Jbb \cdot \Jbb>0$,} in terms of $(\amdeux, \Omega)$ and of the unknowns $(\Jbb, \Kpar, \Lbb)$ defined in~\eqref{eq:definevar-KL} and \eqref{Jdef}, subjected to the causality inequalities~\eqref{eq:causal}, the evolution system consisting of~\eqref{eq:T2-1234-def},~\eqref{eq:T2-Euler-perp-def}, and~\eqref{eq:T2-all-suite-def} is a first-order hyperbolic system of $12$ nonlinear balance laws, supplemented with the $3$ constraints~\eqref{eq:theconstraints-def}. In addition, the wave equation for the lapse~\eqref{eq:waveconffac}, below, follows from this set of equations; in vacuum, it is a consequence of the geometry equations alone.

\hfuzz=0.3pt
\item At every non-zero null state, the two acoustic speeds and the (double) matter speed coalesce. Thus the diagonalization established on the timelike cone \emph{does not extend} uniformly to its null boundary; at $\Jbb=0$, the fluid direction is undefined.

\item If $p'(0)=0$, the limiting equations displayed in~\eqref{equa-Euler-Minko-trivial} have an explicit nontrivial Jordan structure at every non-zero null state; this is a degenerate, weakly hyperbolic regime. 

\eei 

\item \emph{Nonlinearities.}

\bei

\item The double characteristic family associated with the speed ${-\amdeux^{-1}\frac{J_1}{J_0}}$ is \emph{linearly degenerate} (even on the vacuum).  The other two nonlinear wave speeds are distinct away from the vacuum. If $p(\mu)/\mu$ tends to zero with $\mu$ tends to zero, these two nonlinear wave speeds coincide on the vacuum on which hyperbolicity is lost.  If $p(\mu)/\mu$ tends to a constant in $(0,1)$ (isothermal limit), these two nonlinear wave speeds remain distinct even on the vacuum and strict hyperbolicity holds. 

\item The metric evolution equations~\eqref{eq:T2-1234-def} and~\eqref{eq:T2-all-suite-def} and the fluid evolution equations~\eqref{eq:T2-Euler-perp-def} contain only terms that are quadratic in $(\Jbb,\Kpar,\Lbb)$, except that the fluid quadratic forms are multiplied by the bounded state-dependent coefficient $q_\Jbb$. 
They are all exactly quadratic in the isothermal case, since $q_\Jbb$ is then a constant.
\item With the exception of the lapse equations in~\eqref{eq:T2-567-def} and~\eqref{eq:evollambda-def}, all quadratic source terms involving only the variables $\Jperp, \Lbb$ are null forms\footnote{This is not the case for nonlinearities involving the variables $\Jpar, \Kpar$.}.

\eei
\eei
\end{proposition} 

\begin{proof} 
The matrices of the fluid and geometry blocks are calculated in \autoref{section-A-3}. Their real eigenvalues and eigenvectors give strong hyperbolicity for timelike~$\Jbb$. At a non-zero null state ($\Jbb\neq 0$, $-\Jbb\cdot\Jbb=0$), direct substitution into the characteristic roots shows that all wave speeds coincide; when $p'(0)=0$, the limiting matrices displayed there have a nontrivial Jordan structure. Inspection of~\eqref{eq:T2-1234-def}--\eqref{eq:T2-all-suite-def}, together with the bounds for $q$ in~\eqref{equa-qminqmax-1}--\eqref{equa-qminqmax-2}, gives the classification of nonlinearities, as carried out in \autoref{section-A-3}. For instance, the terms $\Qbb\cdot\Qbb$, $\Pbb\cdot\Qbb$, and $\Pbb\wedge\Qbb$ in~\eqref{eq:T2-1234-def} are null forms. 
\end{proof}

\begin{proposition}[Structure of the first-order formulation: constraint propagation and equivalence]
\label{theo-struct-bis}
Under the hypotheses of~\autoref{theo-struct}, the following properties also hold.
\bei

\item \emph{Constraint propagation.}  Given a solution $(\Jbb, \Kpar, \Lbb,\amdeux, \Omega)$ to the equations~\eqref{eq:T2-1234-def},~\eqref{eq:T2-Euler-perp-def}, and~\eqref{eq:T2-all-suite-def}, and provided the constraints~\eqref{eq:theconstraints-def} on $\log\Omega, K_2, K_3$ are satisfied on an initial hypersurface of constant area~$t_0$, they propagate to all future times.

\item \emph{Equivalence property.}

\bei

\item Given such a solution $(\Jbb, \Kpar, \Lbb,\amdeux, \Omega)$, together with the value of the modular parameter~$\tau$ at one spacetime point and the spatial averages of $G,H$ at initial time, one recovers a unique Einstein--Euler solution.

\item The fluid variables $\mu$ and $u$ are determined from~\eqref{Jdef} in every non-vacuum region. In vacuum regions, one has $\mu=0$, while the velocity field $u$ is undefined. In addition, the wave equation~\eqref{eq:waveconffac}, stated below, holds for all areal times.

\eei

\eei
\end{proposition}

\begin{proof} 
The metric coefficients $P,Q$ are determined by integration of~\eqref{eq:definevar-01}, with the compatibility of their space and time derivatives being ensured by the curl equations and the period conditions.
The same relations~\eqref{eq:definevar-01} determine $G,H$ from their initial values at $t=t_0$, which can be freely gauge-fixed to be constant in~$x$ by a change of coordinates that shifts $y$ by the primitive of $G(t_0,\cdot)-\int_0^1 G(t_0,x)dx$ and likewise for~$z$. The relation~\eqref{equa-211b} reconstructs $\mu$ and $u$ on the timelike state cone. The remaining Ricci identity is the lapse equation~\eqref{eq:waveconffac}, which is a consequence of the other equations. Finally, the differential system for the three constraint residuals is homogeneous once the evolution equations hold, as proven in \autoref{section-A-4}, thus proving propagation from one areal slice.
\end{proof}


\subsubsection{Isothermal flows with \texorpdfstring{$\Tbb^2$}{T2} symmetry}

The isothermal equation of state leads to algebraic simplifications in the first-order formulation: the pressure ratio $q$ is constant and the reduced stress-energy components are quadratic. By definition, isothermal matter flows are characterized by the relation $p=k^2\mu$ for some constant $k\in(0,1)$, representing the sound speed, and~\eqref{equa-211b} then reduces to 
\be
- \Jbb \cdot \Jbb = 2(1+ k^2) \mu \quad \text{ (isothermal fluids),}
\ee
namely, the Lorentzian norm of the fluid momentum is proportional to the mass-energy density. Moreover, the pressure ratio $q$ defined by~\eqref{eq:defq} is constant and is given by
\bel{equa-case-iso}
q = \frac{1-k^2}{1+k^2} \eqqcolon 2 \, \kappa \quad \text{ (isothermal fluids)}. 
\ee
Here, $\kappa\in(0,1/2)$ is a convenient parameter associated with the isothermal equation of state, allowing us to distinguish this special case from a general pressure law. In particular, when the function $q$ is constant, the Euler equations become quadratic and, for instance, the energy $\Mbf^{00}$ becomes 
\bel{eq:energyGowdy} 
\Mbf^{00}(\Jbb, \Kpar, \Lbb) 
= \Ebf_0(\Jbb, \Kpar, \Lbb) + \kappa \, \bigl( - \Jbb \cdot \Jbb\bigr) \quad \text{ (isothermal fluids).}
\ee
The terms involving $\Jbb$ take the form $\Mbf^{00}(\Jbb, 0, 0) = (\frac{1}{2}+\kappa)J_0^2+(\frac{1}{2}- \kappa)(J_1^2+J_2^2+J_3^2)$. Since $\kappa<1/2$, this is a strictly convex quadratic form in $(J_0,J_1,J_2,J_3)$.

\paragraph{Structure of the first-order formulation for isothermal fluids.}

For isothermal fluids $p=k^2\mu$,~\autoref{theo-struct} can be strengthened as follows: the \JKL{} formulation of Einstein--Euler areal flows consists exclusively of quadratic terms in the variables $(\Jbb, \Kpar, \Lbb)$. Moreover, in this case, the energy density $(\Jbb, \Kpar, \Lbb) \mapsto \Mbf^{00}(\Jbb, \Kpar, \Lbb)$ is a \emph{non-negative}, \emph{strictly convex} quadratic form in all of its arguments.


\subsubsection{Isothermal flows with Gowdy symmetry}

When, in addition, $G=H=0$ so that $\Kpar=0$, and consequently $\Jpar=0$, the isothermal case becomes especially simple, and we state here the Einstein--Euler system in this setting. From~\eqref{eq:T2-1234-def}--\eqref{eq:theconstraints-def}, we then obtain the Einstein equations 
\be  
\left.
\hskip-.18cm
\aligned
\divuntrois \Pbb & = \Qbb \cdot \Qbb,
\, 
& \curluntrois (|t|^{-1} \Pbb) & = 0, 
\\
\divuntrois \Qbb & = - \Pbb \cdot \Qbb,
\, 
& \curluntrois (|t|^{-1} \Qbb) & = |t|^{-1} \Pbb \wedge \Qbb, 
\\
(\amdeux)_t  & = t \,  \kappa \, \bigl( - \Jperp \cdot \Jperp\bigr) \, \Omega^2 \amdeux,
\\
(\log\Omega)_t
&= - \frac{1}{4t} + \frac{t}{2}\Mbf^{11} \, \Omega^2,
\\
(\log\Omega)_x
&= - \frac{t}{2} \, \Mbf^{01} \, \Omega^2 \amdeux,
\endaligned
\right\} 
\text{(Gowdy, isothermal)}
\ee  
where $\kappa = q/2 = (1-k^2)/(2(1+k^2))$, while the Euler equations take the form 
\bel{eq:T2-567-two-G}
\left.
\aligned
\divuntrois \bigl( \Omega \, \Mbf^{0 \bullet} \bigr)
& = - \frac{1}{2\,t} \Msource,
\\
\qquad \divuntrois \bigl( \Omega \, \Mbf^{1 \bullet} \bigr) 
& = 0,
\endaligned
\right\}
\quad \text{(Gowdy, isothermal),}
\ee 
in which the tensor~\eqref{eq:T2-Mdef-0} and the source term~\eqref{eq:T2-Euler-perp2-def0} are given by 
\bel{eq:T2-567-two2}
\left.
\aligned
\Mbf^{00}
& = \Ebf_0(\Jperp, \Pbb, \Qbb) + \kappa (- \Jperp \cdot \Jperp),
\qquad
\\
\Mbf^{01} & = - \Ebf_1(\Jperp, \Pbb, \Qbb),
\\
\Mbf^{11}
& = \Ebf_0(\Jperp, \Pbb, \Qbb) - \kappa (- \Jperp \cdot \Jperp),
\qquad
\\
\Msource 
& = - \Pbb \cdot \Pbb -  \Qbb \cdot \Qbb + (1 - \kappa) \, (- \Jperp \cdot \Jperp),       
\endaligned
\right\}
\quad \text{(Gowdy, isothermal).}
\ee
Moreover, the source obeys the sharper inequality proven in~\autoref{section-A-5},
\be
\abs{\Msource} \leq 2 \, \Mbf^{00} \quad \text{(Gowdy, isothermal).}
\ee


\section{Convex currents and two mathematical entropy inequalities}
\label{section-4}

\subsection{Mass-energy and particle-number currents}
\label{section-4-1}

\subsubsection{Two balance laws before shocks}

This section presents a pair of identities that enter the maximum principles and the geometric energy estimates in \autoref{part-two}. These hold for regular solutions but become incompatible for finite-energy solutions, which leads to two different weak formulations proposed earlier in~\cite{TorusT2-1-2026}. We distinguish three features. First, every sufficiently regular Euler flow satisfies both the mass-energy and particle-number balance laws as \emph{equalities}. Second, the appropriate timelike component of each current is a convex mathematical entropy only after a suitable choice of main evolution variables (also named conservative variables) has been made. Third, at the weak regularity level, \emph{we must choose} which balance law remains an equality and which one becomes the \emph{reference mathematical entropy inequality}.
(See Dafermos' textbook~\cite{Dafermos-book} for this central notion of the theory of nonlinear hyperbolic systems.)

The first-order variables were defined in~\autoref{section-2-3}. 

\bei 

\item The first \emph{reference current} of interest in our theory is the mass-energy current already present in the \JKL{} system. By \eqref{eq:T2-Euler-perp-def-a}, its regular balance law is $\opJ^0=0$, that is, 
\bel{eq:reg-en-bal}
\divuntrois\bigl(\Omega\Mbf^{0\bullet}(\Jbb,\Kpar,\Lbb)\bigr)=-\frac{1}{2t}\Msource(\Jbb,\Kpar,\Lbb).
\ee

\item The second \emph{reference current} of interest here is the particle-number current~$\vNumb$ introduced momentarily (cf.~\eqref{equa--Nbb}, below), which obeys the homogeneous balance law
\bel{equa-entr}
\divuntrois\vNumb=0.
\ee
\eei
\noindent 
Both identities follow from the Euler system for regular solutions. It is not difficult to check that once jump discontinuities are admitted, one of these balance laws must be replaced by an inequality. The three remaining momentum equations $\opJ^1= \opJ^2= \opJ^3=0$ are kept as equalities in either weak formulation.


\subsubsection{Construction of the particle current}

We write $y \coloneqq -\Jbb\cdot\Jbb=2(\mu+p(\mu))\geq 0$, and recall $J_0\leq 0$. By the monotonicity in~\eqref{eq:section-2-density-change}, the density map is invertible; we denote its inverse by $\mu=\funmu(y)$ and abbreviate $p(y) \coloneqq p(\funmu(y))$ and likewise for other functions of density. All statements involving differentiability are first proven on the non-vacuum set $y>0$ and then extended to vacuum by the growth estimates below.  The normalization of the particle density is immaterial: multiplying it by a positive constant changes neither conservation, convexity, nor any mathematical entropy inequality.

Define the \emph{particle-number density} $\Numb= \Numb(\mu)$, uniquely up to a positive multiplicative constant, by
\bel{eq:section-4-particle-density}
\frac{d\Numb}{\Numb} = \frac{d\mu}{\mu+p(\mu)}, \qquad
\mu > 0, \qquad
\Numb(0) \coloneqq 0.
\ee
Since $p'\geq 0$, we have $\Numb'>0$ and
\bel{eq:section-4-particle-concavity}
\Numb''(\mu) = -\frac{\Numb(\mu) p'(\mu)}{(\mu+p(\mu))^2} \leq 0.
\ee
Thus $\Numb$ is increasing and concave. For a regular perfect-fluid solution, contraction of $\nabla_\alpha T^{\alpha\beta}=0$ with $u_\beta$, together with~\eqref{eq:section-4-particle-density}, gives $\nabla_\alpha(\Numb u^\alpha)=0$, namely~\eqref{equa-entr}.

The primitive velocity is not defined in vacuum, so the expression $\Numb u$ is unsuitable for the finite-energy theory. Using $\Jbb= \sqrt{2(\mu+p)}\,u$, introduce the enthalpy-normalized density
\bel{eq:section-4-hN-definition}
\hNumb(\mu) \coloneqq \frac{\Numb(\mu)}{\sqrt{2(\mu+p(\mu))}}
= \frac{\Numb(\mu)}{\sqrt{-\Jbb\cdot\Jbb}}.
\ee
The \emph{particle-particle current} can then be written entirely in the finite-energy variable $\Jbb$:
\bel{equa--Nbb}
\vNumb = \hNumb_\Jbb\Jbb, \qquad
\hNumb_\Jbb \coloneqq  \hNumb\bigl(\funmu(-\Jbb\cdot\Jbb)\bigr).
\ee
Here the one-form $\Jbb$ is identified with its metric-dual vector; in frame components, $\Numb^0=-\hNumb_\Jbb J_0$ and $\Numb^m=\hNumb_\Jbb J_m$ for $m=1,2,3$. Unlike $u$, the right-hand side extends continuously to the vacuum.  Differentiating~\eqref{eq:section-4-hN-definition} and using the pressure ratio $q=(\mu-p)/(\mu+p)$, we can rewrite \eqref{eq:section-4-particle-density} as
\bel{Sdef-scaled}
d\log\hNumb = \frac{1-p'(\mu)}{2(\mu+p(\mu))}\,d\mu
= \frac{1}{2}(\mu+p)^{-1} d\bigl((\mu+p)q\bigr).
\ee
Hence $0<p'<1$ implies that $\hNumb$ is increasing. This monotonicity is used in \autoref{lem-convexM} to recover the primitive fluid state from the main evolution variables.


\subsubsection{Growth at vacuum and at large density}

The differential identity
\bel{eq:section-4-logarithmic-growth}
\frac{d\log\Numb}{d\log\mu}=\frac{\mu}{\mu+p(\mu)}=\frac{1+q(\mu)}{2}
\ee
turns the pressure hypotheses~\eqref{hyperbolic-eos-zero} into growth estimates. With the constants $q_{\min}$ and $q_\infty^+$ defined there, integration on the near-vacuum and high-density regimes yields
\bel{near-vacuum-Numb}
\mu \lesssim \Numb(\mu) \lesssim \mu^{(1+q_{\min})/2}, \qquad
\mu^{1/2} \lesssim \hNumb(\mu) \lesssim \mu^{q_{\min}/2}, \qquad \mu\to 0,
\ee
\bel{eq:large-dens-growth}
\mu^{(1+q_{\min})/2}\lesssim\Numb(\mu)\lesssim\mu^{(1+q_\infty^+)/2}, \qquad
\mu^{q_{\min}/2}\lesssim\hNumb(\mu)\lesssim\mu^{q_\infty^+/2}, \qquad
\mu\to+\infty.
\ee
The constants implicit in these estimates depend on the normalization of $\Numb$ and on the pressure-law bounds, but not on the fluid state. In particular, $\hNumb(\mu)\to 0$ at vacuum and grows strictly slower than $\mu^{1/2}$ at high density because $q_\infty^+<1$. Consequently, the current~$\vNumb$ is sub-quadratic relative to the quadratic mass-energy variables.

That the map $\Jbb\mapsto\vNumb= \hNumb_\Jbb\Jbb$ is continuously differentiable on the closed causal state cone, that it vanishes along its boundary $-\Jbb\cdot\Jbb=0$, that $\Numb^0>0$ in the cone interior, and that the growth at infinity is bounded as
\bel{eq:section-4-particle-growth}
|\Numb^0|+\sum_{m=1}^3|\Numb^m|\lesssim1+|J_0|^{1+q_\infty^+},\qquad \frac{|\Numb^0|+\sum_{m=1}^3|\Numb^m|}{1+\Mbf^{00}(\Jbb,0,0)}\longrightarrow0\quad\text{as }|J_0|\to+\infty ,
\ee
so that $\vNumb$ is controlled in~$L^1$ by quadratic finite-energy bounds that will be available later on, even near vacuum. Indeed, continuity and vacuum behavior follow from~\eqref{near-vacuum-Numb}. On the future causal cone, $|J_m|\leq|J_0|$, while~\eqref{eq:large-dens-growth} and $-\Jbb\cdot\Jbb\leq J_0^2$ imply $\hNumb_\Jbb\lesssim1+|J_0|^{q_\infty^+}$. This gives the first estimate in~\eqref{eq:section-4-particle-growth}. Since $q_\infty^+<1$ and $\Mbf^{00}(\Jbb,0,0)\geq J_0^2/2$, the ratio tends to zero.


\subsubsection{Isothermal formula}

For $p=k^2\mu$, the pressure ratio $q=(1-k^2)/(1+k^2)$ is constant. Setting $\kappa=q/2\in(0,1/2)$, integration gives, up to one harmless positive normalization,
\bel{eq:section-4-isothermal-density}
\Numb(\mu)= \bigl(2(1+k^2)\bigr)^{\kappa+1/2}\mu^{\kappa+1/2},\qquad \hNumb(\mu)= \bigl(2(1+k^2)\bigr)^\kappa\mu^\kappa.
\ee
Since $-\Jbb\cdot\Jbb=2(1+k^2)\mu$, the particle current takes the homogeneous form
\bel{interesting-entropy}
\vNumb=(-\Jbb\cdot\Jbb)^\kappa\Jbb\qquad\text{(isothermal fluid).}
\ee
This formula illustrates the scaling difference between the quadratic mass-energy current and the sub-quadratic particle current.


\subsection{Future-convexity of reference currents}
\label{section-4-2}

\subsubsection{Main evolution variables}

For a balance-law system describing conservation of some currents (up to possible source terms), a mathematical entropy must be a convex function of the conservative variables, which are the \emph{time components} of conserved currents. This point is essential here because the two reference currents require two different choices of variables. The time components of the fluid stress tensor in the orthonormal frame are
\bel{eq:fl-str-comp}
T^{00} = \frac{1}{2}\Mbf^{00}(\Jbb,0,0), \quad
T^{01} = -\frac{1}{2}J_0J_1, \quad
T^{02} = -\frac{1}{2}J_0J_2, \quad
T^{03} = -\frac{1}{2}J_0J_3.
\ee
We also set $\Numb^0 \coloneqq -\hNumb_\Jbb J_0\geq 0$. The two collections of evolution variables are
\bel{eq:two-princ-vars}
Z_{\rm N} \coloneqq  \bigl(T^{00},T^{01},T^{02},T^{03}\bigr),\qquad
Z_{\rm M} \coloneqq  \bigl(\Numb^0,T^{01},T^{02},T^{03}\bigr),
\ee
where the subscript indicates whether the negative particle number or mass-energy is used as the reference entropy. The two coordinate systems have the same momentum entries but exchange their timelike entry.
This leads us to two ``dual'' lemmas, which state that the remaining current, used as a reference entropy, is convex as a function of timelike components of the other currents.


\subsubsection{First convexity lemma}

When the particle number is used as a reference current, or mathematical entropy, the following property is relevant.

\begin{lemma}[Convexity of the negative particle-number current]
\label{lem-convexN}
Under the pressure hypotheses~\eqref{hyperbolic-eos-zero}, every state in the future causal cone is uniquely recovered from $Z_{\rm N}$, and the map
\bel{eq:section-4-convex-particle-map}
Z_{\rm N}\longmapsto-\Numb^0= \hNumb_\Jbb J_0
\ee
is well-defined, non-positive, and convex. It vanishes continuously at vacuum.
\end{lemma}

\begin{proof}  We refer to~\cite{TorusT2-3-2026} for a proof, and provide only an outline of the argument. 
We set $E \coloneqq T^{00}$ and $|P|^2 \coloneqq  \sum_{m=1}^3(T^{0m})^2$. The density is the unique root in $[0,E]$ of $F(\mu) \coloneqq (E+p(\mu))(E-\mu)-|P|^2=0$, because $F'(\mu)=p'(\mu)(E-\mu)-(E+p(\mu))<0$ under $0<p'(\mu)<1$.  We then recover $\Jbb$ from $J_0^2=2(E+p)$ and $J_0J_m=-2T^{0m}$, with no singularity at vacuum ($\mu=0$, or equivalently $E=|P|$).  Thus $Z_{\rm N}$ uniquely determines the future fluid state in the finite-energy variables. 
After a tedious calculation of the relevant Hessian required for strict convexity in the interior of the future cone, the desired sign follows from $J_0\leq 0$ and $\hNumb\geq 0$, while the vacuum extension follows from~\eqref{eq:section-4-particle-growth}.
\end{proof}


\subsubsection{Second convexity lemma}

When the mass-energy is chosen as a reference current, or mathematical entropy, the following property is relevant.

\begin{lemma}[Convexity of the timelike mass-energy component]
\label{lem-convexM}
Under the pressure hypotheses~\eqref{hyperbolic-eos-zero}, every state in the future causal cone is uniquely recovered from~$Z_{\rm M}$, and the map
\bel{eq:section-4-convex-energy-map}
Z_{\rm M}\longmapsto\Mbf^{00}(\Jbb,0,0)
\ee
is well-defined, non-negative, and convex. The conclusion extends continuously to vacuum.
\end{lemma}

\begin{proof}  We refer to~\cite{TorusT2-3-2026} for a proof, and provide only an outline of the argument. 
For convenience, we write $y=-\Jbb\cdot\Jbb$ and denote components of~$Z_{\rm M}$ as $\rho\coloneqq\Numb^0$ and $P_m\coloneqq T^{0m}$, $m=1,2,3$.  At fixed $|P|^2=(P_1)^2+(P_2)^2+(P_3)^2\geq 0$, one can express $\rho\in[0,+\infty)$ as a $C^1$~function of $y\in[0,+\infty)$ with strictly positive derivative in $(0,+\infty)$. Its inverse determines \(y\) continuously up to the null boundary and smoothly when \(\rho>0\); the algebraic relations then determine the matter momentum~\(\Jbb\).
The representation
\bel{eq:conv-en-sup}
\Mbf^{00}(\Jbb,0,0)= \sup_{y>0} Q_y(\rho,P) , \qquad Q_y(\rho,P)\coloneqq\Bigl( \sqrt{y\hNumb(y)^{-2}\rho^2+4|P|^2}-2p(y) \Bigr)
\ee
is shown by proving that $\del_y Q_y(\rho,P)\gtrless 0$ for $y$ less/greater than the physical value of~$y$. For fixed $y$, $Q_y(\rho,P)$ is a Euclidean norm of a linear function of $(\rho,P)$, shifted by a constant, and is therefore convex. Its supremum is convex.
 When $\rho=0$, the supremum is the limit as $y\downarrow 0$ and gives the continuous extension to vacuum.
\end{proof}


\subsubsection{Convexity in the presence of geometry}

These two lemmas are expressed in terms of fluid variables, only. For the full Einstein--Euler system the geometry variables $(\Kpar,\Lbb)$ add non-negative quadratic terms to the fluid stress-tensor.  In changes of variables concerning~$\Jbb$ they act simply as constant shifts of the conservative variables, and do not affect convexity.  This is why the same entropy choice remains compatible with the coupled \JKL{} system.


\subsection{Mathematical entropy inequalities}
\label{section-4-3} 
 
As already mentioned in~\autoref{section-3-2}, our first-order operator notation is meaningful for weakly regular flows with finite energy in the sense defined by~\eqref{equa-flow-energy-norm}--\eqref{equa-flow-variation-norm}, namely
\bel{finite-energy-meaning}
\|\Phi\|_{\mathrm{energy}}+\|\Psi\|_{\mathrm{BV}}<+\infty.
\ee
The estimates in \autoref{part-two} will control this quantity in terms of the analogous initial quantity defined on a Cauchy hypersurface of constant areal time. 

At this stage, we must explain the role of the reference current in the full definition of weakly regular flows. We use the following standard terminology. The inequality $\opJ^0\leq 0$ means
\be
\langle\opJ^0,\varphi\rangle\leq 0
\ee
for every non-negative test function $\varphi\in C^\infty_c(\Mcal)$. The inequality $\divuntrois\vNumb\geq 0$ is understood analogously in the sense of distributions. Thus, the \emph{mathematical entropy production measures}, defined respectively by
\bel{eq:section-4-production-measures}
\mathfrak d_{\rm M} \coloneqq -\opJ^0,
\qquad
\mathfrak d_{\rm N} \coloneqq \divuntrois\vNumb,
\ee
are non-negative distributions. Since every non-negative distribution is a distribution of order zero, they are therefore non-negative Radon measures. The sign convention is chosen so that both measures represent a non-negative mathematical entropy production; cf.~the discussion of the non-relativistic limit in~\autoref{section-4-4}. More precisely, with the orientation convention entering the definition of~$\divuntrois$, the measures $\mathfrak d_{\rm M}$ and $\mathfrak d_{\rm N}$ describe, respectively, a \emph{decrease of the mass-energy} and an \emph{increase of the particle number}.

While a full definition of \emph{finite-energy Einstein--Euler areal flows} is postponed to \autoref{section-9}, we indicate here how the reference current enters the weak formulation. The three momentum equations remain equalities, while one of the following two mathematical entropy formulations is selected:
\be
\opJ^1=\opJ^2=\opJ^3=0,
\qquad
\begin{cases}
\opJ^0=0,\quad \divuntrois\vNumb\geq0
& \text{in the particle-production formulation},
\\
\divuntrois\vNumb=0,\quad \opJ^0\leq0
& \text{in the dissipated-energy formulation}.
\end{cases}
\ee
For regular flows, \emph{both inequalities reduce to equalities}, and the two formulations coincide. Indeed, both production measures vanish, so every regular flow satisfies both systems.

For weakly regular flows, which may contain shock waves, the two mathematical entropy formulations generally do not coincide: exchanging the equality and inequality \emph{changes the Rankine--Hugoniot locus}. If a discontinuity is represented by $x=\gamma(t)$ with speed $s=\gamma'(t)$, a balance law $\divuntrois X=S$ with a locally integrable source imposes
\be
[X^1-s\amdeux X^0]=0,
\ee
whereas the corresponding mathematical entropy inequality imposes the associated one-sided inequality on this jump. Since the locally integrable source terms in the areal equations have no singular part supported on the shock curve, the distinction between the two formulations comes entirely from the selected reference current.

Before returning to weakly regular flows in \autoref{section-9}, we are going derive structural properties for regular Einstein--Euler flows and for the two Einstein--Navier--Stokes systems introduced in the next section. Interestingly, the same finite-energy quantities will be used for all three systems. 


\subsection{The non-relativistic limit}
\label{section-4-4}

\subsubsection{Scaling regime and limiting system}

To justify our formalism for weak solutions of the relativistic isentropic compressible Euler equations, especially the choice of which of the fluid equations is weakened to an inequality, we make contact with the literature on nonrelativistic fluids.  We consider the plane-symmetric Euler equations in Minkowski space, which can be read off from the Euler--Einstein equations by suppressing the geometry (setting $\amdeux= \Omega=1$, $P_0=P_1=Q_0=Q_1=0$ and eliminating factors of $|t|$ in the operators). It is convenient to return to the notation $(\mu,u)$ instead of $J_m = \sqrt{2(\mu+p)} u_m$ for $m=0,1,2,3$. The Euler equations read
\bel{relativistic-Euler-isentropic}
\aligned
\del_t \Bigl( (\mu+p) u_0^2 - p \Bigr) - \del_x ( (\mu+p) u_0 u_1 ) & = 0, \\
\del_t ( (\mu+p) u_0 u_1 ) - \del_x \Bigl( (\mu+p) u_1^2 + p \Bigr) & = 0, \\
\del_t ( (\mu+p) u_0 u_m ) - \del_x ( (\mu+p) u_1 u_m ) & = 0, \quad m=2,3,
\endaligned
\ee
supplemented by the conservation of particle number~$\Numb(\mu)$,
\bel{relativistic-Euler-isentropic-N}
\del_t ( \Numb(\mu) u_0 ) - \del_x ( \Numb(\mu) u_1 ) = 0, \qquad \frac{\Numb'(\mu)}{\Numb(\mu)} = \frac{1}{\mu+p(\mu)}.
\ee

The non-relativistic limit corresponds to taking, for some small constant $\eps\in(0,1)$, the regime $|u_m|<\eps$ for $m=1,2,3$, hence $u_0= \sqrt{1+u_1^2+u_2^2+u_3^2}=1+\Obig(\eps^2)$, and an equation of state with $0<p'(\mu)<\eps^2$ hence $0\leq p(\mu)\leq \eps^2\mu$.  The \emph{formal} $\eps\to 0$ limit of the system~\eqref{relativistic-Euler-isentropic} is easily found to be the non-relativistic Euler equations
\bel{nonrel-Euler}
\aligned
\del_t \mu - \underline{\del}_x (\mu \underline{u}_1) & = 0, \\
\del_t (\mu \underline{u}_1) - \underline{\del}_x(\mu \underline{u}_1^2 + \underline{p}(\mu)) & = 0, \\
\del_t ( \mu \underline{u}_m ) - \underline{\del}_x ( \mu \underline{u}_1 \underline{u}_m ) & = 0, \quad m=2,3,
\endaligned
\ee
with $\underline{\del}_x= \eps\del_x$ and $\underline{p}(\mu)=p(\mu)/\eps^2$ and $\underline{u}=(u_1,u_2,u_3)/\eps$.


\subsubsection{Limiting entropy balance}

The equation for the particle number can be integrated from some reference density~$\muref$ as
\be
\frac{\Numb(\mu)}{\Numb(\muref)}
= \frac{\mu}{\muref} - \eps^2 \frac{\mu}{\muref} \int_{\muref}^\mu \frac{\underline{p}(\mu')d\mu'}{\mu'^2} + \Obig(\eps^4)
= \frac{\mu}{\muref} + \Obig(\eps^2),
\ee
where the constants implicit in $\Obig$ depend on the sup norm of $\log(\mu/\muref)$. Up to rescaling by $\Numb(\muref)/\muref$, the formal limit of~\eqref{relativistic-Euler-isentropic-N} is $\del_t\mu - \underline{\del}_x(\mu\underline{u}_1) = 0$, just as the $0$-th Euler equation. To account for all equations, we consider the difference between energy-momentum and particle number currents, whose components are
\be
\aligned
\Bigl( (\mu+p(\mu)) u_0^2 - p(\mu) \Bigr) - \frac{\muref}{\Numb(\muref)} \Numb(\mu) u_0
& = \eps^2 \Bigl( \mu |\underline{u}|^2 / 2 + \mu \int_{\muref}^\mu \frac{\underline{p}(\mu')d\mu'}{\mu'^2} \Bigr) + \Obig(\eps^4),
\\
\frac{1}{\eps} \Bigl( (\mu+p) u_0 u_1 - \frac{\muref}{\Numb(\muref)} \Numb(\mu) u_1 \Bigr)
& = \eps^2 \underline{u}_1 \Bigl( \mu |\underline{u}|^2 / 2
+ \underline{p}(\mu)
+ \mu \int_{\muref}^\mu \frac{\underline{p}(\mu')d\mu'}{\mu'^2} \Bigr)
+ \Obig(\eps^4).
\endaligned
\ee
The leading-order term in the $\eps\to 0$ limit reproduces the nonrelativistic energy and energy flux. It is then standard to consider weak solutions to the nonrelativistic Euler equations~\eqref{nonrel-Euler} that are solutions to the energy inequality
\be
\del_t \Bigl( \mu |\underline{u}|^2 / 2 + \mu \int_{\muref}^\mu \frac{\underline{p}(\mu')d\mu'}{\mu'^2} \Bigr)
- \underline{\del}_x \Bigl( \Bigl( \mu |\underline{u}|^2 / 2
+ \underline{p}(\mu)
+ \mu \int_{\muref}^\mu \frac{\underline{p}(\mu')d\mu'}{\mu'^2} \Bigr) \underline{u}_1 \Bigr)
\leq 0.
\ee
The standard nonrelativistic energy inequality can be understood as a formal limit of either one of our formulations.
\bei
\item In the particle-production formulation, the relativistic energy is conserved and the particle current has a positive production, so that their difference dissipates.
\item In the dissipated-energy formulation, the particle current is conserved, and upon subtracting it from the relativistic energy the resulting non-relativistic energy still dissipates.
\eei
We retain both choices in the relativistic theory; deciding which limiting admissibility condition is appropriate is a separate question.


\section{Einstein--Navier--Stokes systems}
\label{section-5} 

\subsection{Two hyperbolic systems and a parabolic model}
\label{section-5-1}

\subsubsection{Background on (non-)relativistic fluids}

We propose two hyperbolic relativistic Navier--Stokes systems. Each system is obtained by adjoining to the perfect-fluid variables a symmetric tensor field~$\zeta$, called the \emph{relaxing field}, and by deriving the constitutive tensors from a scalar potential~$\chi=\chi(\beta,\zeta)$. The relaxing field satisfies a tensorial system of hyperbolic balance laws with a source that drives~$\zeta$ toward the equilibrium manifold~$\{\zeta=0\}$. On this set, the constitutive stress tensor and the particle-number current coincide with their Euler counterparts. In the particle-production system, the total stress-energy tensor is divergence-free and the particle-number current has non-negative divergence. In the dissipated-energy system, the particle-number current is divergence-free and the source in the potential-generated stress tensor has the sign required by the mass-energy mathematical entropy inequality; an additional symmetric tensor cancels this source so that the total stress-energy tensor remains divergence-free. Thus both systems satisfy the contracted Bianchi identity, but they select different mathematical entropy inequalities. In passing, we also mention a projected shear--bulk model whose principal part is parabolic and which will not be studied in the present paper.

Our construction relies on several classical ideas from the analysis of compressible fluids and hyperbolic balance laws. Godunov~\cite{Godunov-1961} introduced variables associated with a convex supplementary conservation law in which the state and flux maps are defined from scalar potentials, while Friedrichs and Lax~\cite{FriedrichsLax-1971} established that a strictly convex extension yields a symmetric hyperbolic formulation. Ruggeri and Strumia~\cite{RuggeriStrumia-1981} developed the corresponding covariant structure for relativistic Euler flows, and Geroch and Lindblom~\cite{GerochLindblom-1990} introduced generating potentials for dissipative relativistic theories in divergence form. Chen, Levermore, and Liu~\cite{ChenLevermoreLiu} analyzed hyperbolic systems with stiff relaxation and identified structural conditions compatible with mathematical entropy dissipation. 

We build upon this circle of ideas with an additional requirement imposed by the Einstein equations: the potential, the relaxation source, and, in the dissipated-energy system, the additional symmetric tensor are chosen jointly so that the selected mathematical entropy production has a fixed sign while the total symmetric stress-energy tensor is divergence-free.
In particular, our theory guarantees that the total symmetric stress tensor remains divergence-free and can be coupled consistently to Einstein's equations, such that the contracted Bianchi equations hold.


\subsubsection{Proposed methodology}

We first formulate the two models on an arbitrary time-oriented four-dimensional Lorentzian spacetime, without imposing any symmetry. For an equation of state \(p=p(\mu)\), the perfect-fluid stress-energy tensor is\footnote{The labels \({\rm E}\) and \({\rm NS}\) are placed as superscripts or subscripts according to typographical convenience.}
\bel{eq:section-5-total-stress}
T^{\rm E}_{\alpha\beta} = (\mu+p)u_\alpha u_\beta + p g_{\alpha\beta}.
\ee
 Here \(u\) is a future-directed unit timelike vector and \(\mu>0\) is the mass-energy density for a classical fluid state. The constitutive tensors and the estimates introduced below extend continuously to the vacuum boundary \(\mu=0\), which is an important feature in our theory of weakly regular flows.

An Einstein--Navier--Stokes system consists of the fluid fields \((\mu,u)\), additional relaxing fields, a symmetric total stress-energy tensor \(T^{\rm NS}\) depending algebraically on these fields, and hyperbolic balance laws for the relaxing fields. The coupled Einstein and matter equations satisfy
\bel{eq:section-5-complete-system}
G_{\alpha\beta}=T^{\rm NS}_{\alpha\beta},
\qquad
\nabla^\alpha T^{\rm NS}_{\alpha\beta}=0.
\ee
Once Einstein's equations hold, the second identity follows from the contracted Bianchi identity. It is therefore the compatibility condition that the constitutive stress tensor and the equations for the relaxing fields must satisfy.

We impose four requirements on each of the two proposed systems.
\bei

\item On the equilibrium manifold, where the relaxing field vanishes, the total stress-energy tensor agrees algebraically with the Euler stress-energy tensor.

\item On the interior of the state space, the principal system is symmetric hyperbolic and causal with respect to the spacetime metric.

\item The system possesses a particle-number current whose contraction with every future timelike covector is strictly convex in the corresponding normal conserved variables.

\item The relaxation source has a fixed sign and yields one of the two mathematical entropy inequalities.
\eei

These properties are obtained from a scalar potential~\(\chi\) on the state space. Its derivatives generate a symmetric stress tensor~\(T_\chi\), the currents associated with the relaxing field, and the particle-number current~\(\vNumb_\chi\). The same potential and relaxing variables are used in both systems. The systems differ in the distribution of the relaxation source between the stress-energy balance law and the particle-number balance law. For both systems, the constitutive inequalities and comparison estimates used below hold on the full density range, \emph{extend continuously to vacuum}, and impose \emph{no upper bound on the finite rapidity.}

\subsubsection{Particle-production model}

The first system is defined on an arbitrary time-oriented Lorentzian spacetime, without symmetry restriction. Its matter variables are the potential variable~\(\beta\), determined by~\((\mu,u)\), and a symmetric relaxing field~\(\zeta\). A scalar potential~\(\chi(\beta,\zeta)\) generates the stress-energy tensor~\(T_\chi\), the relaxation current~\(A_\chi\), and the particle-number current~\(\vNumb_\chi\), while a relaxation map~\(\mathscr B(\beta,\zeta)\) determines their source terms. The stress-energy tensor~\(T_\chi\) is divergence-free, whereas~\(\vNumb_\chi\) has non-negative production. The system is defined in \autoref{def-prod-pot}, and a concrete choice of~\(\chi\) and~\(\mathscr B\) is given in \autoref{rem-example-chi}.

 
\subsubsection{Dissipated-energy model} 

The second system uses the same variables, potential~\(\chi(\beta,\zeta)\), and relaxation map~\(\mathscr B(\beta,\zeta)\), but is adapted to a \(\Tbb^2\)-symmetric spacetime endowed with its areal foliation. The future unit normal to this foliation is denoted by~\(n\). The source in the equation for~\(A_\chi\) is the same as in the particle-production system, whereas the divergence of~\(T_\chi\) is assigned a negative source in the direction~\(n\). This source is adjusted in such a way that the particle-number current~\(\vNumb_\chi\) is divergence-free. To recover the divergence-free total stress-energy tensor required by the contracted Bianchi identity, we introduce an additional symmetric tensor~\(T^{\rm aux}\), proportional to the projector onto the symmetry orbits. Torus invariance ensures that its divergence contains no derivative of the scalar coefficient multiplying this projector. Its divergence exactly cancels the source in the equation for~\(T_\chi\), and therefore the total stress-energy tensor \(T^{\rm NS}=T_\chi+T^{\rm aux}\) is divergence-free. This construction dissipates mass-energy while preserving particle number and is defined in \autoref{def-diss-en}.


\subsubsection{Projected shear--bulk model}

We finally recall a model arising in classical first-order relativistic fluid theories~\cite{Eckart1940,LandauLifshitz}. This model introduces no relaxing field. Its principal part is parabolic in the fluid variables~$(\mu,u)$ and therefore does not have finite propagation speed. Let
\be
h_{\alpha\beta}\coloneqq g_{\alpha\beta}+u_\alpha u_\beta,
\qquad
\vartheta\coloneqq\nabla_\gamma u^\gamma,
\ee
where \(h\) is the spatial projector orthogonal to~\(u\) and \(\vartheta\) is the fluid expansion. The trace-free spatial strain is
\bel{eq:section-5-reference-strain}
\sigma_{\alpha\beta}
\coloneqq
h_\alpha{}^\gamma h_\beta{}^\delta\nabla_{(\gamma}u_{\delta)}
-\frac13\vartheta h_{\alpha\beta}.
\ee
Given two non-negative viscosity coefficients \(\eta_{\rm sh}=\eta_{\rm sh}(\mu)\) and \(\eta_{\rm bk}=\eta_{\rm bk}(\mu)\), the total stress-energy tensor is
\bel{eq:section-5-reference-stress}
T^{\rm NS,para}_{\alpha\beta}
\coloneqq
T^{\rm E}_{\alpha\beta}
-2\eta_{\rm sh}(\mu)\sigma_{\alpha\beta}
-\eta_{\rm bk}(\mu)\vartheta h_{\alpha\beta},
\qquad
\eta_{\rm sh},\eta_{\rm bk}\geq0.
\ee
The two additional terms represent shear and bulk dissipation, respectively. This parabolic system provides a useful comparison with the two hyperbolic relaxation systems, and its detailed analysis is developed in~\cite{TorusT2-4-2026}.


\subsection{The (Godunov, Lax, Friedrichs) Euler potential}
\label{section-5-2}

\subsubsection{Rescaled momentum}

Let us first rewrite the Euler stress-energy tensor~\eqref{eq:section-5-total-stress} as the Hessian of a potential~$\chi$ in field space with respect to a suitably \emph{rescaled momentum} denoted by~$\beta$.  Specifically, we define
\bel{eq:pot-enthalpy}
\beta \coloneqq \frac{\Numb(\mu)}{\mu+p(\mu)} u = \frac{2\hNumb_\Jbb}{-\Jbb\cdot\Jbb} \Jbb ,
\qquad
s \coloneqq - \beta_\alpha \beta^\alpha = \Bigl(\frac{\Numb(\mu)}{\mu+p(\mu)}\Bigr)^2 .
\ee
From~\eqref{eq:section-4-particle-density}, we obtain
\(\del_\mu s=-2s\,\del_\mu p/(\mu+p)<0\) at positive density. Hence \(s\), and therefore~\(\beta\), determines \((\mu,u)\) uniquely away from vacuum. In contrast to~\(\Jbb\), the variable~\(\beta\) does not provide a regular parametrization of the vacuum: the ratio \(\Numb/(\mu+p)\) does not tend to zero and, in the asymptotically isothermal case, becomes unbounded. The variable~\(\beta\) is nevertheless convenient for exhibiting the symmetric hyperbolic structure of the Euler equations.


\subsubsection{Euler system from a potential}

\bse\label{te-from-potential}
We also introduce the \emph{Euler potential} $\chi_{\rm E}(\mu)$ by
\bel{eq:pot-equil}
\chi_{\rm E}'(\mu) = - \frac{1}{2} p(\mu) s'(\mu) 
\ee
and for definiteness the normalization condition $\chi_{\rm E}(0)=0$, which is allowed because the derivative is integrable near vacuum under our pressure assumptions~\eqref{hyperbolic-eos-zero}. Then the Euler stress-energy tensor takes the form
\bel{eq:pot-Euler-id} 
T_{\rm E}^{\alpha\beta} = \frac{\del^2\chi_{\rm E}}{\del\beta_\alpha\del\beta_\beta} , 
\ee
and the following two identities also hold: 
\bel{eq:pot-Euler-id-2}
p\beta^\alpha = \frac{\del\chi_{\rm E}}{\del\beta_\alpha},  
\qquad 
\Numb u^\alpha = \frac{\del\chi_{\rm E}}{\del\beta_\alpha} - \beta_\gamma T_{\rm E}^{\alpha\gamma}.
\ee
Thus the perfect-fluid stress and particle current are defined from a single scalar potential for every pressure law.  While this construction in terms of the $\beta$ variable is  first derived away from vacuum, the resulting tensorial formulas extend by continuity to the vacuum boundary, when expressed in terms of the original momentum variable~$\Jbb$.
\ese

Indeed, regarding \(\chi_{\rm E}\) as a function of~\(s=-\beta_\alpha\beta^\alpha\), we have \(d\chi_{\rm E}/ds=-p/2\), and therefore
\be
\frac{\del\chi_{\rm E}}{\del\beta_\alpha}=p\beta^\alpha,
\qquad
\frac{\del^2\chi_{\rm E}}{\del\beta_\alpha\del\beta_\beta}
=p g^{\alpha\beta}+\frac{\mu+p}{s}\beta^\alpha\beta^\beta
=T_{\rm E}^{\alpha\beta}.
\ee
The second identity in~\eqref{eq:pot-Euler-id-2} follows by contraction with~\(\beta\).

Let \(K_{\rm E}^{\alpha\beta\gamma}\coloneqq \del^3\chi_{\rm E}/(\del\beta_\alpha\del\beta_\beta\del\beta_\gamma)\).
As we discuss more generally in \autoref{section-5-3}, hyperbolicity (and causality with respect to the metric) reduces to showing positivity of the two-tensor $-n_\alpha K_{\rm E}^{\alpha\beta\gamma}$ for every future timelike vector~$n$.
Decompose \(n=\lambda(u+\omega)\), where \(\lambda>0\), \(g(u,\omega)=0\), and \(|\omega|_g<1\), and decompose a variation as \(\dot\beta=\dot a u^\flat+\dot z\), with \(g(u,\dot z)=0\).
Let \(k(\mu)\coloneqq(\del_\mu p(\mu))^{1/2}\) be the sound speed.
A direct calculation gives
\be
-n_\alpha K_{\rm E}^{\alpha\beta\gamma}\dot\beta_\beta\dot\beta_\gamma
= \lambda\frac{\mu+p}{\sqrt{s}} \left(|\dot z+\dot a\omega|_g^2 +\bigl(k^{-2}-|\omega|_g^2\bigr) \dot a^2 \right)>0.
\ee
Thus \(0<k<1\) implies that the Euler system is symmetric hyperbolic and causal with respect to~\(g\).

For the isothermal law \(p=k^2\mu\) we have a constant sound speed \(k\in(0,1)\), and, up to the positive normalization of~\(\Numb\),
\be
s(\mu)=C_s\mu^{-2k^2/(1+k^2)},
\qquad
\chi_{\rm E}(\mu)
=\frac{k^4}{1-k^2}\,\mu s(\mu),
\qquad C_s>0.
\ee
In particular, the potential is explicit and the hyperbolicity inequality above holds throughout the positive-density state space.


\subsection{Navier--Stokes potentials}
\label{section-5-3}

\subsubsection{Potential-generated hyperbolic operators}

Both hyperbolic models involve a symmetric relaxing field~\(\zeta\) taking values in a smooth vector subbundle
\be
\mathscr E\subseteq\Sym^2T^*\Mcal.
\ee
The bundle~\(\mathscr E\) is referred to as the \emph{relaxation bundle}. For example, it may be restricted to the metric-trace line \(\zeta_{\alpha\beta}=\xi g_{\alpha\beta}\), parametrized by a scalar field~\(\xi\). We proceed to describe the models in terms of the variables $(\beta,\zeta)$ in this section; the relation to $\mu,u$ or~$\Jbb$ variables used in later sections will be taken to be~\eqref{eq:pot-enthalpy} independently of the relaxing field~$\zeta$.  In particular, we emphasize that the maps always involve the pressure law~$p$ and particle number~$\Numb$ of the Euler system at $\zeta=0$. 

The main constitutive object is a smooth scalar function \(\chi=\chi(\beta,\zeta)\), called a \emph{Navier--Stokes potential}, which is decomposed as
\be
\chi(\beta,\zeta)
=\chi_{\rm E}(\beta)+\chi_{\rm relax}(\beta,\zeta),
\qquad
\chi_{\rm relax}(\beta,0)=0.
\ee
Taking derivatives with respect to~\(\beta\) and in the fiber directions of~\(\mathscr E\), we define the potential-generated stress-energy tensor and relaxation current by
\bel{eq:pot-currents}
T_\chi^{\alpha\beta}
\coloneqq
\frac{\del^2\chi}{\del\beta_\alpha\del\beta_\beta},
\qquad
A_\chi^{\alpha\beta\gamma}
\coloneqq
\frac{\del^2\chi}{\del\beta_\alpha\del\zeta_{\beta\gamma}}.
\ee
The tensor~\(T_\chi\) is symmetric, while~\(A_\chi\) is symmetric in its last two indices. Moreover, \(T_\chi(\beta,0)=T_{\rm E}(\beta)\). The associated homogeneous principal system is
\bel{eq:pot-general-sources-zero}
\nabla_\alpha T_\chi^{\alpha\beta}=0,
\qquad
\nabla_\alpha A_\chi^{\alpha\beta\gamma}=0.
\ee
Its unknowns are~\((\beta,\zeta)\). Since the principal matrices are third derivatives of the scalar potential~\(\chi\), they are symmetric; causal hyperbolicity reduces to a positivity condition on their contraction with future timelike covectors, as specified below.

 
\begin{definition}
\label{def-pot-hyperbolic}
Let \(\mathscr E\subseteq\Sym^2T^*\Mcal\) be a relaxation bundle and let \(\mathscr F\subset\mathscr E\) be an open fiber domain containing its zero section. A \textbf{Navier--Stokes potential} is a smooth scalar function
\be
\chi=\chi(\beta,\zeta),
\qquad
\beta\ \text{future timelike},
\qquad
\zeta\in\mathscr F,
\ee
satisfying the following conditions.
\bei

\item \emph{Euler equilibrium.} For every future timelike~\(\beta\), one has
\be
\chi(\beta,0)=\chi_{\rm E}(\beta).
\ee

\item \emph{Dominant energy condition.} The potential-generated stress-energy tensor satisfies
\be 
T_\chi^{\alpha\beta}(\beta, \zeta) X_\alpha Y_\beta
=
\frac{\del^2\chi}{\del\beta_\alpha\del\beta_\beta}(\beta, \zeta)
X_\alpha Y_\beta
\geq0
\ee
for all future-directed causal vectors~\(X,Y\), all future timelike~$\beta$ and all $\zeta\in\mathscr F$. 

\item \emph{Causal hyperbolicity.}  With \(\Xi=(\beta,\zeta)\), define the symmetric principal matrices
\bel{eq:pot-principal}
K^{\alpha AB}(\Xi)
\coloneqq
\frac{\del^3\chi}
{\del\beta_\alpha\del\Xi_A\del\Xi_B}(\Xi).
\ee
For every future-directed timelike vector~\(n\), the quadratic form
\be
-n_\alpha K^{\alpha AB}(\Xi)Z_AZ_B
\ee
is positive for every non-zero variation~\(Z\) of\/~\(\Xi\) and every state \((\beta,\zeta)\) in the constitutive domain.
\eei
\end{definition}

Positivity for one future timelike vector~\(n\) suffices for symmetric hyperbolicity. Requiring it for every such~\(n\) ensures that the characteristic propagation cone is contained in the Lorentz causal cone. Strict hyperbolicity is imposed only for future timelike~\(\beta\); the continuous extension of the generated tensors and comparison estimates to the causal boundary in~$\Jbb$ variables is verified separately for the potentials used below.

In view of the required causal hyperbolicity condition, a natural choice, explicited in an example in \eqref{rem-example-chi}, is to take $\chi$ to be quadratic in the relaxing field~\(\zeta\). In that case, the current~\(A_\chi\) is affine in~\(\zeta\), and its divergence equation provides the evolution equations for the relaxing field, coupled to the conservation equation for~\(T_\chi\).


\subsubsection{Navier--Stokes particle-number current}

Associated with the potential~\(\chi\), we introduce the particle-number current
\bel{eq:pot-particle-current}
\vNumb_\chi^\alpha
\coloneqq
\frac{\del\chi}{\del\beta_\alpha}
-\beta_\beta T_\chi^{\alpha\beta}
-\zeta_{\beta\gamma}A_\chi^{\alpha\beta\gamma}.
\ee
Due to~\eqref{eq:pot-Euler-id-2}, on the equilibrium manifold~\(\zeta=0\), this current agrees with the Euler particle-number current:
\be
\vNumb_\chi^\alpha(\beta,0)
=\Numb(\mu)u^\alpha
=(\vNumb_\Jbb)^\alpha .
\ee
Its divergence is determined algebraically by the source terms in the equations for~\(T_\chi\) and~\(A_\chi\).

\begin{proposition}
\label{prop-pot-relax}
Let~\(\chi\) be a Navier--Stokes potential and let \(I^\beta\) and \(I^{\beta\gamma}\) be algebraic source terms. Then the system
\bel{eq:pot-general-sources}
\nabla_\alpha T_\chi^{\alpha\beta}=I^\beta,
\qquad
\nabla_\alpha A_\chi^{\alpha\beta\gamma}=I^{\beta\gamma}
\ee
is symmetric hyperbolic in the interior of the state space, and the particle-number current satisfies
\bel{eq:pot-chain-identity}
\nabla_\alpha\vNumb_\chi^\alpha
=-\beta_\beta I^\beta
-\zeta_{\beta\gamma}I^{\beta\gamma}.
\ee
Moreover, let \(n\) be future timelike and set
\be
U^{\alpha A}
\coloneqq
\frac{\del^2\chi}{\del\beta_\alpha\del\Xi_A},
\qquad
Y^A\coloneqq-n_\alpha U^{\alpha A},
\qquad
\eta_n\coloneqq n_\alpha\vNumb_\chi^\alpha,
\qquad
\Xi=(\beta,\zeta).
\ee
Then \(\eta_n\) is strictly convex as a function of the normal conserved variables~\(Y\).
\end{proposition}


\begin{proof}
\bse
The currents \(U^{\alpha A}\) collect \(T_\chi^{\alpha\beta}\) and \(A_\chi^{\alpha\beta\gamma}\). Their Jacobian matrices are
\be
\frac{\del U^{\alpha A}}{\del\Xi_B}
=K^{\alpha AB}
=\frac{\del^3\chi}
{\del\beta_\alpha\del\Xi_A\del\Xi_B}.
\ee
They are symmetric by commutation of derivatives. Since \(-n_\alpha K^{\alpha AB}\) is positive definite by \autoref{def-pot-hyperbolic}, the map \(\Xi\mapsto Y\) is locally invertible and the system is symmetric hyperbolic.

From~\eqref{eq:pot-particle-current}, we obtain the differential identity
\(d\vNumb_\chi^\alpha=-\Xi_A\,dU^{\alpha A}\). Consequently, we have 
\bel{eq:pot-convexity}
D_Y\eta_n=\Xi,
\qquad
D_Y^2\eta_n
=\bigl(-n_\alpha K^\alpha\bigr)^{-1}>0.
\ee
This proves the strict convexity assertion. The chain rule gives a (space-time) divergence
\bel{eq:pot-chain}
\nabla_\alpha\frac{\del\chi}{\del\beta_\alpha}
=
T_\chi^{\alpha\beta}\nabla_\alpha\beta_\beta
+
A_\chi^{\alpha\beta\gamma}\nabla_\alpha\zeta_{\beta\gamma}.
\ee
We differentiate~\eqref{eq:pot-particle-current}, use~\eqref{eq:pot-chain}, and cancel the terms containing derivatives of~\(\beta\) and~\(\zeta\). We find
\bel{eq:pot-chain-prod}
\nabla_\alpha\vNumb_\chi^\alpha
=
-\beta_\beta\nabla_\alpha T_\chi^{\alpha\beta}
-\zeta_{\beta\gamma}\nabla_\alpha A_\chi^{\alpha\beta\gamma}.
\ee
Substitution of~\eqref{eq:pot-general-sources} proves~\eqref{eq:pot-chain-identity}.
\ese
\end{proof}


\subsection{Proposed Navier--Stokes models}
\label{section-5-4}

\subsubsection{Relaxation sources}

We now specify source terms that drive the relaxing field~$\zeta$ toward the equilibrium manifold~$\{\zeta=0\}$. According to \autoref{prop-pot-relax}, the sources in the equations for $T_\chi$ and $A_\chi$ determine the divergence of the particle-number current~$\vNumb_\chi$. We first prescribe the source in the equation for~$A_\chi$ through a map~$\mathscr B$ and define the corresponding non-negative production~$\mathfrak d_{\mathscr B}$. This production is then inserted either in the particle-number balance law or in the stress-energy balance law, so that the other current is conserved. An explicit choice of~$\mathscr B$ is given in~\eqref{B-barrier}.

\begin{definition}
\label{def-pot-source}
Let $\chi$ be a Navier--Stokes potential on an open fiber domain  \(\mathscr F\subset\mathscr E\) (containing its zero section) of a relaxation bundle \(\mathscr E\subseteq\Sym^2T^*\Mcal\). Then, a \textbf{relaxation rate map} consists of a smooth fiberwise map $\mathscr B=\mathscr B(\beta,\zeta)\in\mathscr E^*$ defined for all future timelike $\beta$ and all $\zeta \in \mathscr F$, satisfying the following conditions.

\bei 

\item \emph{Stable equilibrium property:}
\bel{eq:pot-relax-map}
\mathscr B(\beta,0)=0,
\qquad
\zeta_{\alpha\beta}\mathscr B^{\alpha\beta}(\beta,\zeta)\geq0,
\qquad
\zeta_{\alpha\beta}\mathscr B^{\alpha\beta}(\beta,\zeta)=0
\Longleftrightarrow
\zeta=0.
\ee

\item \emph{Positivity property:}  Its linearization along the equilibrium manifold $\mathbb M$ given by 
\bel{eq:pot-relax-linearization}
\mathbb M^{\alpha\beta\gamma\delta}(\beta)
\coloneqq
\frac{\del\mathscr B^{\alpha\beta}}{\del\zeta_{\gamma\delta}}(\beta,0)
\ee
must be symmetric under interchange of the two index pairs and positive definite on~$\mathscr E$, namely
\be
\eta_{\alpha\beta}\mathbb M^{\alpha\beta\gamma\delta}(\beta)\eta_{\gamma\delta}>0
\qquad
\text{for every nonzero }\eta\in\mathscr E.
\ee
\eei
Furthermore, the associated \textbf{production term} is
\bel{production-term-value}
\mathfrak d_{\mathscr B}
\coloneqq
\zeta_{\alpha\beta}\mathscr B^{\alpha\beta}
\geq0.
\ee
\end{definition}

\begin{definition}
\label{def-NS-triple}
A \textbf{Navier--Stokes triple} $(\mathscr F,\chi,\mathscr B)$ consists of an open fiber domain  \(\mathscr F\subset\mathscr E\) (containing its zero section) of a relaxation bundle \(\mathscr E\subseteq\Sym^2T^*\Mcal\), 
a Navier--Stokes potential~$\chi$ (in the sense of \autoref{def-pot-hyperbolic}), and a relaxation rate map~$\mathscr B$ (in the sense of \autoref{def-pot-source}).
\end{definition}

\subsubsection{Particle-production model}

The first model is defined on an arbitrary time-oriented Lorentzian spacetime and does not require any symmetry. Its total stress tensor is conserved, while the particle-number current has a non-negative production.

\begin{definition}
\label{def-prod-pot}
The \textbf{particle-production Einstein--Navier--Stokes model} associated with a Navier--Stokes triple $(\mathscr F,\chi,\mathscr B)$ is the Einstein--matter system
\bel{eq:pot-production-system}
G^{\alpha\beta}=T_\chi^{\alpha\beta},
\qquad
\nabla_\alpha T_\chi^{\alpha\beta}=0,
\qquad
\nabla_\alpha A_\chi^{\alpha\beta\gamma}
=-\mathscr B^{\beta\gamma}.
\ee
The unknowns are the spacetime metric and the fields $(\beta,\zeta)$, representing the rescaled matter momentum and the relaxing field, respectively. The tensors $T_\chi$ and $A_\chi$ are defined from the potential according to~\eqref{eq:pot-currents}.
\end{definition}

The chain identity in \autoref{prop-pot-relax} immediately gives the particle-number balance law
\bel{prod-pot-dNchi}
\nabla_\alpha\vNumb_\chi^\alpha
=\zeta_{\beta\gamma}\mathscr B^{\beta\gamma}
=\mathfrak d_{\mathscr B}
\geq0.
\ee
Thus, the Einstein equation is compatible with the contracted Bianchi identity, while the particle-number current has a non-negative production.


\subsubsection{Dissipated-energy model}

The second model is adapted to a prescribed spacelike foliation. Its particle-number current is conserved, while the potential-generated stress tensor loses energy in the direction of the future unit normal~$n$. An additional symmetric stress tensor compensates this loss and makes the total stress tensor divergence-free.

We specialize to $\Tbb^2$ symmetry and to the orthonormal frame~\eqref{eq:ourframe}, with future-directed unit normal $n=e_0=\Omega^{-1}\del_t$ to the areal foliation. We denote by~$\pi^\parallel$ the orthogonal projector onto the tangent spaces of the symmetry orbits:
\be
\pi^\parallel=e^2\otimes e^2+e^3\otimes e^3.
\ee
Its divergence is
\be
\nabla^\alpha\pi^\parallel_{\alpha\beta}
=t^{-1}\Omega^{-1}n_\beta.
\ee
We therefore introduce the (signed) rescaled tensor
\be
\underline\pi\coloneqq t\Omega\pi^\parallel.
\ee
For every $\Tbb^2$-invariant scalar field~$\sigma$, its gradient is orthogonal to the symmetry orbits, and consequently
\bel{eq:pot-orbit-projector-div}
\nabla^\alpha\bigl(\sigma\underline\pi_{\alpha\beta}\bigr)
=\sigma n_\beta.
\ee
In particular, no derivative of~$\sigma$ occurs in this divergence identity.

\begin{definition}
\label{def-diss-en}
Assume that all fields are $\Tbb^2$-invariant and work in the areal gauge. The \textbf{dissipated-energy Einstein--Navier--Stokes model} associated with a Navier--Stokes triple $(\mathscr F,\chi,\mathscr B)$ is
\bel{eq:pot-dissipated-system}
\begin{aligned}
G^{\alpha\beta}
&=T_\chi^{\alpha\beta}+T_{\rm aux}^{\alpha\beta},
&
\nabla_\alpha T_\chi^{\alpha\beta}
&=-\frac{\mathfrak d_{\mathscr B}}{\nu}n^\beta,
&
\nabla_\alpha A_\chi^{\alpha\beta\gamma}
&=-\mathscr B^{\beta\gamma},
\\
T_{\rm aux}^{\alpha\beta}
&=\frac{\mathfrak d_{\mathscr B}}{\nu}\underline\pi^{\alpha\beta},
&
\nu&\coloneqq-n_\alpha\beta^\alpha>0.
\end{aligned}
\ee
The unknowns are the spacetime metric and the fields $(\beta,\zeta)$. The tensors $T_\chi$ and $A_\chi$ are defined from~$\chi$ according to~\eqref{eq:pot-currents}, while $T_{\rm aux}$ transfers to the orbital directions the energy dissipated by~$T_\chi$.
\end{definition}

\begin{corollary}
\label{prop-diss-flow}
In the setting of \autoref{def-diss-en}, the additional stress tensor satisfies
\bel{nabTauxdB}
\nabla_\alpha T_{\rm aux}^{\alpha\beta}
=\frac{\mathfrak d_{\mathscr B}}{\nu}n^\beta.
\ee
The total stress tensor $T_{\rm NS}\coloneqq T_\chi+T_{\rm aux}$ thus satisfies
\bel{eq:pot-dissipated-system-3}
G^{\alpha\beta}=T_{\rm NS}^{\alpha\beta},
\qquad
\nabla_\alpha T_{\rm NS}^{\alpha\beta}=0.
\ee
The particle-number current~$\vNumb_\chi$ defined in~\eqref{eq:pot-particle-current} is also conserved:
\be
\nabla_\alpha\vNumb_\chi^\alpha=0.
\ee
\end{corollary}

\begin{proof}
\bse
Since $\mathfrak d_{\mathscr B}/\nu$ is $\Tbb^2$-invariant, we apply~\eqref{eq:pot-orbit-projector-div} with $\sigma=\mathfrak d_{\mathscr B}/\nu$ and obtain~\eqref{nabTauxdB}.
Together with the stress source in~\eqref{eq:pot-dissipated-system}, this identity gives
\be
\nabla_\alpha T_{\rm NS}^{\alpha\beta}
=\nabla_\alpha T_\chi^{\alpha\beta}
+\nabla_\alpha T_{\rm aux}^{\alpha\beta}
=0.
\ee
We next use the chain identity~\eqref{eq:pot-chain-identity}, the relation $\beta_\alpha n^\alpha=-\nu$, and the definition~\eqref{production-term-value}. We find
\be
\nabla_\alpha\vNumb_\chi^\alpha
=-\beta_\beta\left(-\frac{\mathfrak d_{\mathscr B}}{\nu}n^\beta\right)
+\zeta_{\beta\gamma}\mathscr B^{\beta\gamma}
=-\mathfrak d_{\mathscr B}+\mathfrak d_{\mathscr B}
=0.
\ee
Thus, the total stress tensor is divergence-free and the particle-number current is conserved.
\ese
\end{proof}

\begin{remark}
\label{rem-pot-two-src}
The particle-production and dissipated-energy systems have the same principal matrices~$K^\alpha$, since they are defined from the same potential and differ only in their source terms. On the equilibrium manifold~$\{\zeta=0\}$, we have $\mathscr B=0$, $\mathfrak d_{\mathscr B}=0$, and $T_{\rm aux}=0$, while $T_\chi$ and $\vNumb_\chi$ coincide with the Euler stress tensor and particle-number current. The equilibrium system is therefore the Einstein--Euler system. The zero-relaxation limit is studied in~\cite{TorusT2-4-2026}.
\end{remark}


\subsubsection{Finite-time exclusion of the constitutive boundary}

Symmetric hyperbolicity is available only while the relaxing field takes its values in the \emph{open constitutive domain}~\(\mathscr F\), which may be bounded in the fiber directions. It does not, by itself, prevent a regular solution from approaching the boundary~\(\del\mathscr F\). We therefore supplement the algebraic conditions on the potential and the relaxation rate map with a \emph{dynamical condition} ensuring that the boundary of the constitutive domain cannot be reached in finite time. This condition is related to the classical notion of a positively invariant domain for nonlinear hyperbolic systems. (See, for instance, Dafermos' textbook~\cite{Dafermos-book}). Here we use a slightly stronger version, since we also exclude convergence toward the boundary at a finite terminal time.

\begin{definition}
\label{def-pot-admissible-triple}
A \textbf{Navier--Stokes triple} \((\mathscr F,\chi,\mathscr B)\) is called \textbf{admissible} provided the following \textbf{invariant-domain property} holds for each of the two associated Einstein--Navier--Stokes models 
\eqref{eq:pot-production-system} and \eqref{eq:pot-dissipated-system}, as follows. 

Let \(\{\Sigma_t\}_{t_0\leq t<t_1}\) be a spacelike foliation of a compact spacetime slab, and let \((g,\beta,\zeta)\) be a regular solution of the corresponding hyperbolic evolution-constraint system\footnote{A gauge must be chosen to deal with the spacetime metric and formulate a well-posed initial value problem; this is a standard topic, not reviewed here \cite{Choquet-book}.} on \([t_0,t_1)\). Assume that
\be
\zeta(t_0,\Sigma_{t_0})\subset\mathcal K_0
\qquad\text{for some closed fiber subset }\mathcal K_0\Subset\mathscr F.
\ee
Then one has
\be
\zeta(t,\Sigma_t)\subset\mathscr F
\qquad\text{for every }t_0\leq t<t_1.
\ee
Moreover, if the fields admit a continuous limit in the ambient bundles as \(t\uparrow t_1\), while the nonsingular coefficients in the equation for~\(\zeta\) remain bounded, then the terminal relaxing field satisfies
\be
\zeta(t_1,\Sigma_{t_1})\subset\mathscr F,
\ee
rather than merely taking its values in~\(\overline{\mathscr F}\). Equivalently, the constitutive boundary~\(\del\mathscr F\) cannot be the first cause of breakdown of a regular solution.
\end{definition}

The stable-equilibrium condition
\be
\zeta_{\alpha\beta}\mathscr B^{\alpha\beta}(\beta,\zeta)>0
\qquad\text{for }\zeta\neq0
\ee
does not imply this property by itself. It controls the \emph{direction} of the relaxation near the equilibrium manifold, but lower-order terms in the equation for~\(\zeta\) may still drive the solution toward~\(\del\mathscr F\). A separate boundary estimate is therefore required.
The invariant-domain property can be established by means of a scalar function that diverges at the constitutive boundary.
When the bundle $\mathscr E$ is equipped with a positive-definite bundle metric~$g_{\mathscr E}$ (e.g., for the trace bundle $\mathscr E=\operatorname{span}(g)$), one can select a relaxation rate map based on a barrier function: for some $\delta>0$,
\bel{B-barrier}
\mathscr B = \frac{\delta^2}{\delta^2 - g_{\mathscr E}(\zeta, \zeta)} g_{\mathscr E}(\zeta, \cdot) , \qquad
\mathscr F = \bigl\{\zeta \bigm| g_{\mathscr E}(\zeta,\zeta) < \delta^2 \bigr\} .
\ee
The divergence of $\mathscr B$ at the constitutive boundary ensures that it prevents the relaxation field from exiting~$\mathscr F$.
In the following sections, we always work with an admissible Navier--Stokes triple. 


\section{The \JKL{} formulation of Einstein--Navier--Stokes areal flows}
\label{section-6}

\subsection{Navier--Stokes corrections to \JKL{} operators}
\label{section-6-1}

\subsubsection{Stress-tensor corrections}

We now express the Einstein--Navier--Stokes equations of \autoref{section-5} with $\Tbb^2$ symmetry in the language of \autoref{section-3}.  After presenting the operators, we express in \autoref{section-6-2} the particle-production and dissipated-energy models and justify basic features such as constraint propagation. In \autoref{section-6-3} we determine a class of bundles~$\mathscr E$ (for the relaxation variables~$\zeta$) and scalar potentials~$\chi$ that are adapted to the $\Tbb^2$-symmetric setting, in the sense that the relaxation terms preserve transport equations for $\Jpar$ and~$\Kpar$ that are crucial for maximum principles in \autoref{section-7}.  This finalizes our choice of Einstein--Navier--Stokes system studied in the rest of our work.

Navier--Stokes systems were described in the previous section in terms of variables $(\beta,\zeta)$.  The fluid momentum~$\Jbb$ used in our first-order formulation is defined from~$\beta$ by inverting~\eqref{eq:pot-enthalpy} for non-vacuum states:
\be
\Jbb = \frac{2^{1/2}(\mu+p(\mu))^{3/2}}{\Numb(\mu)} \beta , \qquad
\mu \text{ solution of} \quad \frac{\Numb(\mu)}{\mu+p(\mu)} = (- \beta_\alpha \beta^\alpha)^{1/2} .
\ee
The change of variables involves the particle number and pressure law of the equilibrium Euler system at $\zeta=0$, and has no $\zeta$~dependence.

Compared to the Einstein--Euler system, the Euler stress tensor~$T_{\rm E}$ is replaced by the Navier--Stokes tensor~$T_{\rm NS}$ in source terms of the Einstein equations, and in the conservation equation of the stress-tensor (contracted Bianchi identities).  The system also includes the sourced evolution equations for the relaxing field,
\be
\nabla_\alpha A_\chi^{\alpha\beta\gamma}
= - \mathscr B^{\beta\gamma} .
\ee
As described in \autoref{section-5}, the Navier--Stokes tensor is obtained from some potential $\chi=\chi_{\rm E}+\chi_{\rm relax}$:
\be
T_{\rm NS} = T_{\rm E} + T_{\chi_{\rm relax}} + T_{\rm aux} , \qquad
T_\chi^{\alpha\beta} = T_{\rm E}^{\alpha\beta} + T_{\chi_{\rm relax}}^{\alpha\beta} = \frac{\del^2 \chi}{\del\beta_\alpha\del\beta_\beta} .
\ee
Here, $T_{\rm aux}=0$ for the particle-production model, while this auxiliary stress-tensor is non-trivial in the dissipated-energy model.

We introduce a notation for components of the relaxation term $T_{\chi_{\rm relax}} = T_\chi - T_{\rm E}$,
\bel{eq:section-7-full-reduced-stress}
\mathcal S_{mn} \coloneqq 2T_{\chi_{\rm relax}}(e_m,e_n) = 2 \frac{\del^2 \chi_{\rm relax}}{\del\beta^m \del\beta^n}, \qquad
\mathcal S \coloneqq \eta^{mn}\mathcal S_{mn}, \qquad m,n\in\{0,1,2,3\} ,
\ee
and accordingly $\mathcal S^{mn} = \eta^{mm'} \eta^{nn'} \mathcal S_{m'n'}$,
where $\beta^m$ are components of $\beta=\frac{\Numb(\mu)}{\mu+p(\mu)} u$ in the frame~$e$.
In the dissipated-energy model, the auxiliary stress tensor defined in~\eqref{eq:pot-dissipated-system} has only two equal orbital components. We set
\bel{Sigma-explicit}
\Sigma \coloneqq \frac{2}{t} T_{\rm aux}(e_2,e_2)=\frac{2\Omega}{\beta^0}\mathfrak d_{\mathscr B}\geq 0, \qquad
T_{\rm aux}(e_m, e_n) = \begin{cases}
  \frac{1}{2} t \Sigma & \text{if } m=n\in\{2,3\} , \\
  0 & \text{otherwise}.
\end{cases}
\ee
We set $\Sigma=0$ and $T_{\rm aux}=0$ in the particle-production model.  Thus $T_{\rm aux}$ appears as a pressure term projected onto symmetry orbits; as a result it only affects one of the Einstein equations below. Consequently, the total stress components satisfy
\bel{TNS-JJ-explicit}
\aligned
2 T_{\rm NS}(e_m,e_n) & = J_m J_n + \frac{1-q_\Jbb}{2} (-\Jbb\cdot\Jbb) \eta_{mn} + \mathcal S_{mn} + \delta_{m=n\in\{2,3\}} t \Sigma ,
\\
2 \eta^{mn} T_{\rm NS}(e_m,e_n) & = (1-2q_\Jbb) (-\Jbb\cdot\Jbb) + \mathcal S + 2 t \Sigma .
\endaligned
\ee
The operators below include~$\mathcal S$, namely the potential-generated correction $2(T_\chi-T_{\rm E})$.  The contributions of~$T_{\rm aux}$, or equivalently~$\Sigma$, are treated as sources instead.  This convention separates the common principal part of the two systems from the additional orbital term in the dissipated-energy system.


\subsubsection{Geometric operator corrections}

We then obtain the Einstein--Navier--Stokes equations by replacing all quadratic terms in~$J$ in the Einstein--Euler equations of \autoref{section-3-1} by their Navier--Stokes analogue according to~\eqref{TNS-JJ-explicit}.
The four wave-map operators are
\bse\label{eq:expl-NS-op}
\be
\aligned
\opL_{00}^\chi & \coloneqq \opL_{00}+\frac{1}{2}(\mathcal S_{22}-\mathcal S_{33}), \qquad &
\opL_{01}^\chi & \coloneqq \opL_{01}, \\
\opL_{10}^\chi & \coloneqq \opL_{10}+\mathcal S_{23}, \qquad &
\opL_{11}^\chi & \coloneqq \opL_{11},
\endaligned
\ee
The evolution operator for the speed, lapse, and twists are
\be
\aligned
\opC^\chi & \coloneqq \opC - \frac{t}{2}\Omega^2\amdeux(\mathcal S_{00}-\mathcal S_{11}), \qquad &
\opO_0^\chi & \coloneqq \opO_0-\frac{t}{2}\Omega^2\mathcal S_{11},
\\
\opK_{20}^\chi & \coloneqq \opK_{20}-\Omega\mathcal S_{12}, \qquad &
\opK_{30}^\chi & \coloneqq \opK_{30}-\Omega\mathcal S_{13},
\endaligned
\ee
and the constraint operators for the lapse and twists are
\bel{opchi-constr}
\opO_1^\chi \coloneqq \opO_1-\frac{t}{2}\Omega^2\amdeux\mathcal S_{01}, \qquad
\opK_{21}^\chi \coloneqq \opK_{21}-\Omega\amdeux\mathcal S_{02}, \qquad
\opK_{31}^\chi \coloneqq \opK_{31}-\Omega\amdeux\mathcal S_{03}.
\ee
\ese
These $11$~operators correspond to two compatibility operators and $9$~components of the Einstein equations.  As in the Einstein--Euler system in \autoref{section-3}, the last Einstein equation will be seen as being equivalent to the time component of the stress-tensor conservation law.


\subsubsection{Fluid operator corrections}

In the Einstein--Euler system, fluid variables evolve according to the operators $\opJ^0,\opJ^1,\opJ^2,\opJ^3$.  These are neither quite the contracted Bianchi identity $\nabla^\alpha G_{\alpha\beta}=0$ nor the conservation $\nabla^\alpha T_{\alpha\beta}=0$ as they include derivatives of both fluid and geometric variables.  The spatial operators $\opJ^1,\opJ^2,\opJ^3$ can be understood in two useful ways:
\bei

\item by decomposing $\nabla^\alpha T_{\alpha\beta}=0$ into frame components and rewriting some stress-tensor terms (without derivatives) in terms of the geometry using $T_{\alpha\beta}=G_{\alpha\beta}$;

\item as compatibility conditions for the evolution and constraint equations of $\log\Omega,K_2,K_3$.

\eei
\noindent Navier--Stokes corrections are thus included as in~\eqref{eq:expl-NS-op}.

The operator $\opJ^0$ is obtained either from the component $e_0^\beta \nabla^\alpha T_{\alpha\beta}=0$ or from the last Einstein equation. In~\eqref{eq:T2-Euler-perp2-def0} we defined a source term for the first Euler equation
\bel{Msource-Euler-rewrite}
\aligned
\Msource
& = 5 \, \Ebf_0(\Kpar)
- \Pbb \cdot \Pbb -  \Qbb \cdot \Qbb
+ \frac{3}{2} \bigl(J_2^2 + J_3^2 + (1-q_{\Jbb})(-\Jbb\cdot\Jbb)\bigr)
- \frac{1}{2} (1-2q_\Jbb) (-\Jbb\cdot\Jbb)
\\
& = 5 \, \Ebf_0(\Kpar)
- \Pbb \cdot \Pbb -  \Qbb \cdot \Qbb
+ 3 T_{\rm E}(e_2,e_2) + 3 T_{\rm E}(e_3,e_3)
- \eta^{mn} T_{\rm E}(e_m,e_n) .
\endaligned
\ee
The Navier--Stokes analogue of~$\opJ^0$ is obtained by replacing $T_{\rm E}$ by~$T_\chi$ in this source as well as in the divergence parts of the operators. Altogether, the operators are
\bel{eq:expl-more-NS-op}
\aligned
\opJ^0_\chi & \coloneqq \opJ^0 + \divuntrois \bigl( \Omega \, \mathcal S^{0\bullet} \bigr)
+ \frac{1}{4t} \bigl( 3 \mathcal S_{22} + 3 \mathcal S_{33} - \mathcal S \bigr) ,
\\ 
\opJ^1_\chi & \coloneqq \opJ^1 + \divuntrois \bigl( \Omega \, \mathcal S^{1\bullet} \bigr),
\\ 
\opJ^2_\chi & \coloneqq \opJ^2 + \divuntrois \bigl(|t|^{1/2} \mathcal S^{2\bullet}\bigr) ,
\\ 
\opJ^3_\chi & \coloneqq \opJ^3 + \divuntrois \bigl(|t|^{1/2} \mathcal S^{3\bullet}\bigr) ,
\endaligned
\ee
where we recall that indices of~$\mathcal S$ are raised using~$\eta^{mn}$.
The contributions $\mathcal S_{mn}$ are algebraic in the fluid momentum~$\Jbb$ and relaxation variables~$\zeta$.  In contrast to~\eqref{eq:expl-NS-op} which involve these components without derivatives, \eqref{eq:expl-more-NS-op} includes their first-order derivatives, as befitting an evolution equation.

In the dissipated-energy model, the total stress tensor is $T_{\rm NS} = T_\chi + T_{\rm aux}$.  Accordingly, the source term~\eqref{Msource-Euler-rewrite} receives an additional contribution
\be
3T_{\rm aux}(e_2,e_2) + 3T_{\rm aux}(e_3,e_3) - \eta^{mn} T_{\rm aux}(e_m,e_n) = 2t \Sigma .
\ee
This results in correcting $\opJ^0$ to $\opJ^0_\chi + \Sigma$ in the dissipated-energy model.


\subsubsection{A geometry-matter tensor}

Recall from \autoref{section-3-1} that the evolution and constraint of the speed and lapse, and the first two Euler equations are conveniently expressed in terms of a combination $\Mbf^{\bullet\bullet}$ of the fluid and geometry variables, as well as a source term $\Msource$.
By expressing their fluid part in terms of $T_{\rm E}$ and replacing it by $T_\chi$ we are led to defining
\bel{Mbfchi-expr}
\aligned
\Mbf_\chi^{00}(\Jbb,\Kpar,\Lbb,\zeta) & \coloneqq 2 T_\chi^{00} + \Ebf_0(\Kpar, \Lbb) = \Mbf^{00}(\Jbb,\Kpar,\Lbb) + \mathcal S^{00} ,
\\
\Mbf_\chi^{01}(\Jbb,\Lbb,\zeta) & \coloneqq 2 T_\chi^{01} - \Ebf_1(\Pbb) - \Ebf_1(\Qbb) = \Mbf^{01}(\Jbb,\Lbb) + \mathcal S^{01} ,
\\
\Mbf_\chi^{11}(\Jbb,\Kpar,\Lbb,\zeta) & \coloneqq 2 T_\chi^{11} + \Ebf_0(\Lbb) - \Ebf_0(\Kpar) = \Mbf^{11}(\Jbb,\Kpar,\Lbb) + \mathcal S^{11} ,
\\
\Msource^\chi(\Jbb, \Kpar, \Lbb,\zeta)
& \coloneqq 3 T_\chi^{22} + 3 T_\chi^{33} - \eta_{mn} T_\chi^{mn}
+ 5 \, \Ebf_0(\Kpar) - \Pbb \cdot \Pbb -  \Qbb \cdot \Qbb
\\
& \, = \Msource(\Jbb, \Kpar, \Lbb) + \frac{1}{2} \bigl( 3 \mathcal S_{22} + 3 \mathcal S_{33} - \mathcal S \bigr) ,
\endaligned
\ee
where $T_\chi^{mn}$ denote frame components of~$T_\chi$.  For the Euler stress tensor these reduce to $\Mbf^{\bullet\bullet}$ and $\Msource$.
This notation allows us to write
\bel{using-Mbfchi-expr}
\aligned
\opC^\chi & = (\amdeux)_t - \frac{t}{2}(\Mbf_\chi^{00}- \Mbf_\chi^{11})\,\Omega^2 \amdeux ,
\\
\opO_0^\chi & = (\log\Omega)_t + \frac{1}{4t} - \frac{t}{2}\Mbf_\chi^{11}\,\Omega^2 ,
\\
\opO_1^\chi & = (\log\Omega)_x + \frac{t}{2} \, \Mbf_\chi^{01}\, \Omega^2 \amdeux ,
\\
\opJ^0_\chi & = \divuntrois \bigl( \Omega \, \Mbf_\chi^{0 \bullet}(\Jbb, \Kpar, \Lbb) \bigr)
+ \frac{1}{2t} \Msource^\chi(\Jbb, \Kpar, \Lbb) ,
\\ 
\opJ^1_\chi & = \divuntrois \bigl( \Omega \, \Mbf_\chi^{1 \bullet}(\Jbb, \Kpar, \Lbb) \bigr) .
\endaligned
\ee
This form will help us exhibit positivity properties when deriving geometric estimates in \autoref{section-8-1}.


\subsection{Two models and their structure}
\label{section-6-2}

\subsubsection{Particle-production and dissipated-energy systems}

The models of interest are built from a fiber domain $\mathscr F$, a scalar potential $\chi(\beta,\zeta)$ and a relaxation rate map~$\mathscr B$ satisfying Definitions~\ref{def-pot-hyperbolic} and~\ref{def-pot-source}. Here we express the particle-production system of \autoref{def-prod-pot} and the dissipated-energy system of \autoref{def-diss-en} in $\Tbb^2$~symmetry using the operators~\eqref{eq:expl-NS-op} and~\eqref{eq:expl-more-NS-op} above.
The two systems share the three constraint equations for $\log\Omega,K_2,K_3$,
\bel{eqchi-constr}
\opO_1^\chi = \opK_{21}^\chi = \opK_{31}^\chi = 0 ,
\ee
and the following evolution equations
\bel{eqchi-evol}
\aligned
\opL_{00}^\chi = \opL_{01}^\chi = \opL_{10}^\chi = \opL_{11}^\chi & = 0, \qquad &
\opJ^1_\chi = \opJ^2_\chi = \opJ^3_\chi & = 0, \\
\opC^\chi = \opO_0^\chi = \opK_{20}^\chi = \opK_{30}^\chi & = 0 , \qquad &
\nabla_\alpha A_\chi^{\alpha\beta\gamma} & = -\mathscr B^{\beta\gamma} .
\endaligned
\ee
The divergence equation for~$A_\chi$ is left in covariant form here for brevity. It consists of $\dim\mathscr E$ evolution equations. In total,~\eqref{eqchi-evol} consists of $11+\dim\mathscr E$ evolution equations.

\begin{definition}
\label{def-prod-areal-flow}
A \textbf{particle-production Einstein--Navier--Stokes areal flow} consists of $(\Phi,\Psi,\zeta)$ with $\Phi=(\Jperp,\Jpar,\Pbb,\Qbb)$ and $\Psi=(\ell,\log\Omega,\Kpar)$, satisfying the constraints~\eqref{eqchi-constr}, the evolution equations~\eqref{eqchi-evol}, and the energy conservation equation
\bel{prod-energy-conservation}
\opJ^0_\chi = 0 ,
\ee
together with the same algebraic inequalities~\eqref{eq:causal} on $\amdeux,\Omega,\Jbb$ and periodicity conditions~\eqref{equa-pericons} on $P_1,Q_1$ as in \autoref{def-quadr}.

A \textbf{dissipated-energy Einstein--Navier--Stokes areal flow} consists of $(\Phi,\Psi,\zeta)$ satisfying the inequalities~\eqref{eq:causal}, periodicity~\eqref{equa-pericons}, constraints~\eqref{eqchi-constr}, evolution equations~\eqref{eqchi-evol}, and the dissipated-energy equation
\bel{energy-dissip-Sigma}
\opJ^0_\chi = - \Sigma , \qquad
\Sigma \coloneqq
\frac{\Omega(-\Jbb\cdot\Jbb)}{\Numb^0} \, \mathfrak d_{\mathscr B} ,
\ee
where $\mathfrak d_{\mathscr B} = \zeta_{\alpha\beta} \mathscr B^{\alpha\beta} \geq 0$ is the production term defined in \autoref{def-pot-source}.
\end{definition}

\begin{remark}
The coefficient $(-\Jbb\cdot\Jbb)/N^0=2(\mu+p(\mu))/(N(\mu)u^0)$ is non-negative and bounded above by $1+\mu^{1/2}$ up to a constant, thanks to $u^0\geq 1$ and the asymptotics \eqref{near-vacuum-Numb}--\eqref{eq:large-dens-growth}.
\end{remark}


\subsubsection{Structural properties of the models}

We then have an analogue of Propositions~\ref{theo-struct} and~\ref{theo-struct-bis} for the particle-production and dissipated-energy Navier--Stokes models.
It describe the system as an evolution-constraint system, equivalent to the geometric equations of Definitions~\ref{def-prod-pot} and~\ref{def-diss-en}, and provides conservation of the particle current.

\begin{proposition}[Structural properties of the Einstein--Navier--Stokes models]
\label{prop-pot-JKL}
Consider the system of equations in \autoref{def-prod-areal-flow} under the pressure assumptions~\eqref{hyperbolic-eos-zero}.
Then the following properties hold for sufficiently regular solutions~$(\Phi,\Psi,\zeta)$.
\bei

\item \emph{Hyperbolicity.} \emph{Under the non-vacuum condition $- \Jbb \cdot \Jbb>0$,} the evolution system consisting of~\eqref{eqchi-evol} and either~\eqref{prod-energy-conservation} or~\eqref{energy-dissip-Sigma} is a first-order hyperbolic system of $12+\dim\mathscr E$ nonlinear balance laws with characteristic cones in the causal cone.

\item \emph{Constraint propagation.}  Given a solution of this hyperbolic system, if the constraints~\eqref{eqchi-constr} on $\log\Omega, K_2, K_3$ are satisfied at initial time they propagate to all future times.

\item \emph{Equivalence property.}
Given such a solution $(\Phi,\Psi,\zeta)$, together with the value of the modular parameter~$\tau$ at one spacetime point and the spatial averages of $G,H$ at initial time, one recovers a unique solution $(g,\beta,\zeta)$ of the Einstein--Navier--Stokes equations, specifically a particle-production relaxation flow (\autoref{def-prod-pot}) or dissipated-energy relaxation flow (\autoref{def-diss-en}).
In particular, the wave equation for the lapse~\eqref{eq:waveconffac-NS} below holds.

\item \emph{Particle-number current.}
The convex current derived from the potential~$\chi$ satisfies
\bel{eq:pot-areal-production}
\aligned
\divuntrois\vNumb_\chi
& = \mathfrak d_{\mathscr B}
\quad && \text{(particle-production areal flow),}
\\
\divuntrois\vNumb_\chi & = 0
\quad && \text{(dissipated-energy areal flow),}
\endaligned
\ee
where $\mathfrak d_{\mathscr B} = \zeta_{\alpha\beta}\mathscr B^{\alpha\beta}(\beta,\zeta) \geq 0$.
\eei
\end{proposition}

\begin{remark}
Particle-production areal flows have $\opJ^0_\chi=0$ and $\divuntrois\vNumb_\chi\geq 0$ and dissipated-energy flows have $\opJ^0_\chi\leq 0$ and $\divuntrois\vNumb_\chi=0$.
Formally, in the zero-relaxation limit $\zeta\to 0$ these equalities and inequalities reduce to those for (particle-production or dissipated-energy) weakly regular Einstein--Euler flows.
\end{remark}


\begin{proof}
\bse
For the fluid evolution equations, the conditions on~$\chi$ in \autoref{def-pot-hyperbolic} ensure hyperbolicity with a characteristic cone in the causal cone.  The geometric evolution equations have a principal part that coincides with the one in the \JKL{} formulation of the Einstein--Euler system, whose hyperbolicity is proven in \autoref{theo-struct}.

The three constraint operators are propagated for the same reason as for the \JKL{} formulation.  The equations $\opJ^1_\chi = \opJ^2_\chi = \opJ^3_\chi = 0$ were obtained as compatibility conditions between the constraint and evolution equations of $\log\Omega,K_2,K_3$.  This gives
\be
\aligned
(2 \opO_1^\chi )_t
& = |t| \amdeux \Omega^2 \opJ^1_\chi + (2 \opO_0^\chi)_x ,
\\
\bigl( |t|^{3/2} \opK_{21}^\chi \bigr)_t
& = |t| \amdeux \Omega^2 \opJ^2_\chi + \bigl( |t|^{3/2} \opK_{20}^\chi \bigr)_x ,
\\
\bigl( |t|^{3/2} \opK_{31}^\chi \bigr)_t
& = |t| \amdeux \Omega^2 \opJ^3_\chi + \bigl( |t|^{3/2} \opK_{30}^\chi \bigr)_x ,
\endaligned
\ee
and these right-hand sides vanish for solutions of the evolution equations.

The reconstruction of the Lorentzian metric~$g$ from $(\Phi,\Psi)$ is identical to~\autoref{theo-struct-bis}. The equations for $(\Phi,\Psi,\zeta)$ then give the Einstein and fluid equations except for the remaining lapse wave equation. As derived in \autoref{section-A-2} for the Einstein--Euler system, this equation follows from the evolution equations and constraints; the matter tensor enters only through the source combination on the right-hand side. We obtain
\bel{eq:waveconffac-NS}
\aligned
& \divuntrois \Bigl( t^{-1}  \Omega^{-1} (\log(\Omega \amdeux) )_t, \ t^{-1}  \Omega^{-1} \amdeux^{-1} (\log\Omega)_x \Bigr)
+ \frac{1}{4t^3} \Omega^{-2}
\\
& \quad = \frac{1}{4t} \, \Bigl( \Mwave(\Jbb, \Kpar, \Lbb) + \mathcal S_{00} - \mathcal S_{11} + \mathcal S_{22} + \mathcal S_{33} \Bigr)
+  \frac{1}{2} \Sigma .
\endaligned
\ee
Finally, the balance laws for $\Nbb_\chi$ are as determined in \autoref{section-5-3}.
\ese
\end{proof}


\subsection{Relaxation models adapted to torus symmetry}
\label{section-6-3}

\subsubsection{A bundle for relaxation variables}

In a $\Tbb^2$-symmetric spacetime, we consider the projectors $\pi^\perp$ and $\pi^{\parallel}$ orthogonal and parallel to the symmetry orbits, which can be expressed in the orthonormal coframe as
\be
g = \pi^\perp + \pi^{\parallel} , \qquad
\pi^\perp = - e^0 \otimes e^0 + e^1 \otimes e^1 , \qquad
\pi^{\parallel} = e^2 \otimes e^2 + e^3 \otimes e^3 .
\ee
We select a rank-$2$ sub-bundle of $\Sym^2 T^*\Mcal$, whose section~$\zeta$ we parametrize using two scalar variables $\xi_\perp,\xi_\parallel$:
\bel{rank-2-bundle}
\mathscr E_\pi = \operatorname{span} \{ \pi^\perp , \pi^{\parallel} \} ,
\qquad
\zeta = \xi_\perp \pi^\perp + \xi_\parallel \pi^{\parallel} .
\ee
Since $\pi^\perp_{\alpha\beta}\pi^{\perp\,\alpha\beta} = \pi^\parallel_{\alpha\beta}\pi^{\parallel\,\alpha\beta} = 2$ and $\pi^\perp_{\alpha\beta} \pi^{\parallel\,\alpha\beta}=0$, the basis of $\mathscr E_\pi^*$ dual to $(\pi^\perp,\pi^\parallel)$ is $(\pi^\perp/2,\pi^\parallel/2)$. This normalization is used in the formula for~\(A_\chi\) below.

\begin{remark}\label{rem-example-chi}
The quadratic form on~$\mathscr E_\pi$ induced by~$g$ defines a positive-definite bundle metric
\be
g_{\mathscr E_\pi}(\zeta,\zeta) = g^{\alpha\gamma} g^{\beta\delta} \zeta_{\alpha\beta} \zeta_{\gamma\delta} = 2 (\xi_\perp)^2 + 2(\xi_\parallel)^2 .
\ee
In \cite{TorusT2-4-2026} we prove that the quadratic potential
\be
\chi = \chi_{\rm E}(\mu) + a(\mu) \theta_{b,\kappa}(\zeta) , \qquad
\theta_{b,\kappa}(\zeta) = 2 \kappa(\xi_\perp + \xi_\parallel) + b (\xi_\perp)^2 + b (\xi_\parallel)^2 ,
\ee
with a suitable function $a(\mu)$ chosen to coincide with $\chi_{\rm E}(\mu)$ at large densities and decay faster as $\mu\to 0$, together with the relaxation rate map~\eqref{B-barrier} for a suitable $\delta>0$, gives an admissible torus-adapted Navier--Stokes triple.
The same model, restricted to $\xi_\parallel=\xi_\perp$, namely to the trace bundle $\mathscr E=\operatorname{span}\{g\}$, also gives a torus-adapted Navier--Stokes triple.
\end{remark}


\subsubsection{Models that respect $\Tbb^2$ symmetry}

\begin{definition}
\label{def-NS-torus-adapted}
Consider a Navier--Stokes potential~$\chi$ on a domain $\mathscr F\subset\mathscr E$ of a bundle as defined in \autoref{def-pot-hyperbolic}.  It is said to be \textbf{torus-adapted} if $\mathscr E \subset \mathscr E_\pi$ is a sub-bundle of the rank-$2$ bundle~\eqref{rank-2-bundle} and the potential depends covariantly on the vector~$\beta$ and on $\zeta\in\mathscr E$, namely can be written as a function $\chi=\chi(s_\perp,s_\parallel,\xi_\perp,\xi_\parallel)$ of four scalar variables, where
\be
s_\perp = - \beta_\alpha \pi^{\perp\,\alpha\beta} \beta_\beta , \qquad
s_\parallel = \beta_\alpha \pi^{\parallel\,\alpha\beta} \beta_\beta , \qquad
s_\perp \geq s_\parallel \geq 0 .
\ee
By extension, a Navier--Stokes triple $(\mathscr F, \chi, \mathscr B)$ is said to be torus-adapted if~$\chi$ is torus-adapted.
\end{definition}

\begin{proposition}[Particle and mixed-stress factorization]
\label{prop-NS-torus-adapted}
Consider a torus-adapted Navier--Stokes potential~$\chi$.
The orthogonal components of the particle-number current, and the mixed parallel-orthogonal stress-tensor components, take the following form, for some scalar functions~$\Numb_\chi,U_\chi$:
\be
\vNumb_\chi^a = \Numb_\chi(s_\perp, s_\parallel, \xi_\perp, \xi_\parallel) u^a ,
\qquad
T_\chi^{am} = \frac{1}{2} U_\chi(s_\perp, s_\parallel, \xi_\perp, \xi_\parallel) u^a u^m , \qquad
a=0,1 , \quad m=2,3 .
\ee
\end{proposition}

\begin{proof}
\bse
Upon decomposing $\beta=\beta^\perp+\beta^\parallel$ with $\beta^\perp_\alpha=\pi^\perp_{\alpha\beta}\beta^\beta$ and $\beta^\parallel_\alpha=\pi^\parallel_{\alpha\beta}\beta^\beta$,
one gets
\bel{factor-dchidbeta}
\frac{\del\chi}{\del\beta_\alpha}
= - 2 \frac{\del\chi}{\del s_\perp} \beta^{\perp\,\alpha} + 2 \frac{\del\chi}{\del s_\parallel} \beta^{\parallel\,\alpha} .
\ee
Since $s_\perp$ and $s_\parallel$ are quadratic in~$\beta$, whereas $\xi_\perp$ and $\xi_\parallel$ are linear in~$\zeta$, we introduce the weighted scaling operator
\be
\mathscr D_\pi
\coloneqq2s_\perp\del_{s_\perp}+2s_\parallel\del_{s_\parallel}
+\xi_\perp\del_{\xi_\perp}+\xi_\parallel\del_{\xi_\parallel}.
\ee
 The definition~\eqref{eq:pot-particle-current}, together with~\eqref{factor-dchidbeta}, then gives the exact decomposition
\be
\aligned
\vNumb_\chi^\alpha
& = \biggl( 1 - \beta_\beta \frac{\del}{\del\beta_\beta} - \zeta_{\beta\gamma} \frac{\del}{\del\zeta_{\beta\gamma}} \biggr) \frac{\del\chi}{\del\beta_\alpha}
\\
& = \biggl( - 1 + \beta_\beta \frac{\del}{\del\beta_\beta} + \zeta_{\beta\gamma} \frac{\del}{\del\zeta_{\beta\gamma}} \biggr) \biggl(
 2 \frac{\del\chi}{\del s_\perp} \beta^{\perp\,\alpha} - 2 \frac{\del\chi}{\del s_\parallel} \beta^{\parallel\,\alpha} \biggr)
\\
& = 2 \mathscr D_\pi \biggl(\frac{\del\chi}{\del s_\perp}\biggr) \beta^{\perp\,\alpha}
- 2 \mathscr D_\pi \biggl(\frac{\del\chi}{\del s_\parallel}\biggr) \beta^{\parallel\,\alpha} .
\endaligned
\ee
On the other hand, the stress-tensor is a second-order derivative, explicitly calculated to be
\bel{explicit-Tchiab}
\aligned
T_\chi^{\alpha\beta} = \frac{\del\chi}{\del\beta_\alpha\del\beta_\beta}
& = 4 \frac{\del^2\chi}{(\del s_\perp)^2} \beta^{\perp\,\alpha} \beta^{\perp\,\beta}
- 4 \frac{\del^2\chi}{\del s_\perp\del s_\parallel} ( \beta^{\perp\,\alpha} \beta^{\parallel\,\beta} + \beta^{\perp\,\beta} \beta^{\parallel\,\alpha})
\\
& \quad + 4 \frac{\del^2\chi}{(\del s_\parallel)^2} \beta^{\parallel\,\alpha} \beta^{\parallel\,\beta} - 2 \frac{\del\chi}{\del s_\perp} \pi^{\perp\,\alpha\beta} + 2 \frac{\del\chi}{\del s_\parallel} \pi^{\parallel\,\alpha\beta} .
\endaligned
\ee
One has $\beta^{\perp a} = s^{1/2}u^a$ and $\beta^{\parallel m} = s^{1/2} u^m$ with $s = -\beta\cdot\beta = s_\perp - s_\parallel$.
Therefore the coefficients in the proposition are
\be
\Numb_\chi
= 2 s^{1/2} \mathscr D_\pi\!\left(\frac{\del\chi}{\del s_\perp}\right),
\qquad
U_\chi
= -8s\frac{\del^2\chi}{\del s_\perp\del s_\parallel}.
\qedhere
\ee
\ese
\end{proof}

\begin{proposition}[Compatibility with the Gowdy subclass]
\label{prop-Gow-compat}
Fix a torus-adapted Navier--Stokes triple $(\mathscr F, \chi, \mathscr B)$.  Gowdy symmetry is propagated in the sense that any particle-production or dissipated-energy Navier--Stokes areal flow $(\Phi,\Psi,\zeta)$ with $\Kpar=\Jpar=0$ on one areal slice has $\Kpar=\Jpar=0$ identically. In this symmetry class, the wave-map operators $\opL_{\bullet\bullet}^\chi$ coincide with those of the vacuum Einstein equations, the operators $\opJ^2,\opJ^3,\opK_{\bullet\bullet}$ vanish identically, and the operators $\opC^\chi,\opO_0^\chi,\opO_1^\chi,\opJ^0_\chi,\opJ^1_\chi$ contain the potential-generated corrections in~\eqref{eq:expl-NS-op}--\eqref{eq:expl-more-NS-op}; in the dissipated-energy system, $\opJ^0_\chi$ has the source term $\Sigma = 2\Omega (\beta^0)^{-1}\mathfrak d_{\mathscr B}$.
\end{proposition}

\begin{proof}
The vanishing of $\Kpar,\Jpar$ at all times is an immediate corollary of the maximum principle \autoref{them:7-4} stated and proven in the next section, which bounds sup norms of suitably rescaled variables $\Gammazero,\Gammaunpar$ by those at initial time, which are here assumed to vanish.
Thus, Gowdy symmetry is preserved.  In this class, the stress-tensor~\eqref{explicit-Tchiab} reduces to
\be
T_\chi^{\alpha\beta}
= 4 \frac{\del^2\chi}{(\del s_\perp)^2} \beta^{\perp\,\alpha} \beta^{\perp\,\beta}
- 2 \frac{\del\chi}{\del s_\perp} \pi^{\perp\,\alpha\beta} + 2 \frac{\del\chi}{\del s_\parallel} \pi^{\parallel\,\alpha\beta} .
\ee
In particular, $\mathcal S_{22}=\mathcal S_{33}$ and $\mathcal S_{23}=0$, whence the lack of sources in~$\opL_{\bullet\bullet}^\chi$.  Substitution in~\eqref{eq:expl-NS-op}--\eqref{eq:expl-more-NS-op} proves the stated reduction.
\end{proof}


\needspace{20\baselineskip}
\part{Cauchy stability of finite-energy Einstein areal flows}
\label{part-two} 

\section{Maximum and total-variation principles}
\label{section-7}

\subsection{Einstein--Euler areal flows}
\label{section-7-1}

\subsubsection{From structural identities to estimates}

We now establish, from the structural identities of \autoref{part-one}, \emph{quantitative regularity properties} and a priori estimates. Our results encompass first regular solutions in three settings: Einstein--Euler areal flows, dissipated-energy and particle-production Einstein--Navier--Stokes areal flows.  All integral and pointwise estimates for the Navier--Stokes systems are uniform with respect to the relaxation rate field. They lead to the stability properties in \autoref{theo-intro-stability}.

Each estimate combines structural properties of the Einstein equations in $\Tbb^2$ symmetry and features of the matter stress-tensor.  For the Navier--Stokes systems, these require the Navier--Stokes potential~$\chi$ to be causal and dominant (\autoref{def-pot-hyperbolic}) and torus-adapted (\autoref{def-NS-torus-adapted}).
\bei
\item In this section, we obtain maximum principles for the twist variables and for the fluid momentum parallel to the symmetry orbits from transport equations.  The key idea is that the particle number and parallel momentum obey conservation equations that have the same velocity (the fluid velocity), hence their ratio is transported at that speed. (See~\autoref{proposition-JJJ}.)  For Navier--Stokes models, this relies on $\chi$~being torus-adapted (cf.~Proposition~\ref{proposition-JJJ-NS} and~\ref{prop-pot-maximum}). We also derive a total-variation principle for the Einstein--Euler system (cf.~\autoref{proposition-JJJ-TV}) and one of the Einstein--Navier--Stokes models (cf.~\autoref{proposition-JJJ-NS-TV}). Finally, in \autoref{them:7-4}, we summarize our results in the present section and discuss the extension to vacuum and weakly regular solutions. 

\item The following sections establish geometric monotonicity properties and weighted energy inequalities for these flows.  This requires the fluid stress tensor~$T_\chi$ to obey the dominant energy condition, and the mass-energy to be either conserved (which occurs in particle-production models) or dissipated.
\eei

Some details of the conclusions depend on the model. To obtain maximum principles in this section, we consider first (regular) Einstein--Euler areal flows. Appropriately scaled twists and parallel momenta are exactly transported, hence their range at some areal time~$t$ coincides with that at initial time~$t_0$. For dissipated-energy Navier--Stoke areal flows, the same conclusions hold. For particle-production Navier--Stoke areal flows, the particle current is produced rather than conserved, but the two momentum equations remain exactly conserved; the transport equations therefore have a damping term of favorable sign.  As a result, star-shaped regions around the origin are invariant domains for this system.  We also present a maximum-principle result based on entropy inequalities that is more robust but yields only that convex regions containing the origin are invariant domains.


\subsubsection{Parallel momentum per particle}

The particle number component $\Numb^0=-\hNumb(\mu)J_0\geq 0$ plays an important role here.  From the statements above~\eqref{eq:section-4-particle-growth}, it vanishes if and only if $\mu=0$, namely in vacuum.
We work first in a non-vacuum regime in which $J_0<0$ and $\Numb^0>0$. The resulting identities extend to vacuum whenever the normalized variables introduced below possess continuous bounded representatives.

The parallel fluid evolution operators $\opJ^2,\opJ^3$ include first derivatives of the fluid variables as well as derivatives of twists that can be rewritten using the constraints $\opK_{\bullet\bullet}$.  Doing so retrieves the usual form of (these components of) the Euler equations $e_m^\beta\nabla^\alpha T^{\rm E}_{\alpha\beta} = 0$ for $m=2,3$.  Explicitly, these are
\bel{J2Jperp-cons}
\aligned
\divuntrois(J_2\Jperp)
& = - \frac{1}{2} J_2(\Pbb\cdot\Jperp) + \frac{1}{2t\Omega} J_2 J_0 ,
\\
\divuntrois(J_3\Jperp)
& = \frac{1}{2} J_3(\Pbb\cdot\Jperp)-J_2\,(\Qbb\cdot\Jperp) + \frac{1}{2t\Omega} J_3 J_0 .
\endaligned
\ee
On the other hand, recall from \autoref{section-4-1} that the particle current is $\vNumb= \hNumb_\Jbb\Jbb$.  For regular flows it has vanishing divergence
\bse\label{hNumbJperp-cons-zero}
\bel{hNumbJperp-cons}
\divuntrois(\vNumb^\perp) = 0 .
\ee
Here we can omit the parallel components of~$\vNumb$, since $\divuntrois$ only depends on orthogonal components. To exhibit the structure more explicitly, we introduce variables $(\rho,\velo)$ that account for the geometric weights in the operator $\divuntrois$, and rewrite \eqref{hNumbJperp-cons} as
\bel{eq:section-8-Euler-speed}
\rho \coloneqq |t|\amdeux\Omega\Numb^0>0, \qquad
\velo \coloneqq -\amdeux^{-1}\frac{J_1}{J_0} , \qquad
\del_t \rho + \del_x (\rho \velo) = 0 .
\ee
\ese
Due to $J_0<0$ and $\amdeux>0$ the velocity~$\velo$ has the same sign as~$J_1$. The causality inequalities imply $|J_1/J_0| \leq 1$, hence $|\velo| \leq \amdeux^{-1}$.

The conservation laws~\eqref{J2Jperp-cons} have the same velocity~$\velo$ as~\eqref{hNumbJperp-cons-zero}. This motivates us to introduce the parallel momentum per particle in the orthonormal frame
\bel{equa-jhat}
\Jhatpar= \widehat J_2+i\widehat J_3 \coloneqq  \hNumb_\Jbb^{-1}\Jpar,
\ee
which is subject to a sourced transport system
\bel{eq:Jhat-transp}
\aligned
\del_t\widehat J_2+\velo\del_x\widehat J_2 & = \Omega\frac{\Pbb\cdot\Jperp}{2J_0} \widehat J_2 - \frac{1}{2t} \widehat J_2, \\
\del_t\widehat J_3+\velo\del_x\widehat J_3 & = \Omega\frac{\Qbb\cdot\Jperp}{J_0}\widehat J_2 - \Omega\frac{\Pbb\cdot\Jperp}{2J_0} \widehat J_3 - \frac{1}{2t}\widehat J_3.
\endaligned
\ee
This system can be combined with particle conservation. For every $H\in C^1(\RR^2)$, it gives the sourced first-order balance law
\bel{eq:Jhat-H}
\aligned
\divuntrois\bigl(H(\widehat J_2,\widehat J_3)\vNumb\bigr)
& = \frac{1}{2} \bigl(-\widehat J_2\del_{\widehat J_2}H + \widehat J_3\del_{\widehat J_3}H\bigr) (\Pbb\cdot\vNumb)
- \widehat J_2 \del_{\widehat J_3} H\,(\Qbb\cdot\vNumb) \\
& \quad + \frac{1}{2t\Omega}\bigl(\widehat J_2 \del_{\widehat J_2} H + \widehat J_3 \del_{\widehat J_3} H \bigr)\Numb_0.
\endaligned
\ee
Its right-hand side contains no derivative of a main unknown, but it has no sign regardless of the (non-constant) choice of~$H$. The origin of these sources is that the normalization~\eqref{equa-jhat} removes the density carried by the matter momentum but does not yet remove the rotation and dilation of the orthonormal frame $e_2,e_3$ along the quotient.


\subsubsection{Killing momentum per particle}

To eliminate all source terms, we turn to components in the basis $(\del_y,\del_z)$ of Killing vectors. This requires reconstructing the metric components $P,Q$ from the first-order variables, up to their harmless additive normalization at one point, from
\bel{reconstruct-P-Q}
\aligned P_t & = P_0\Omega, & \qquad
P_x & =P_1\Omega\amdeux,
\\
Q_t & =e^{-P}Q_0\Omega, & Q_x & = e^{-P}Q_1\Omega\amdeux.
\endaligned
\ee
These four relations are compatible thanks to the curl equations $\opL_{01}=\opL_{11}=0$. We introduce the Killing-frame matrix
\bel{eq:Kill-mat}
\mathsf M_{\rm K}(t,P,Q) \coloneqq |t|^{1/2} \begin{pmatrix}e^{P/2}&0\\ Qe^{P/2}&e^{-P/2}\end{pmatrix},\qquad
\mathsf M_{\rm K}^{-1} = |t|^{-1/2} \begin{pmatrix}e^{-P/2}&0\\-Qe^{P/2}&e^{P/2}\end{pmatrix}.
\ee
The components of $\Jhatpar$ in the fixed Killing basis $(\del_y,\del_z)$, rather than in the moving orthonormal frame, are therefore $\Gammaunpar=\mathsf M_{\rm K}(\widehat J_2,\widehat J_3)^T$, that is,
\bel{equawidetildeJ}
\Gammaun_2 \coloneqq |t|^{1/2}e^{P/2}\widehat J_2,\qquad
\Gammaun_3 \coloneqq |t|^{1/2}\bigl(Qe^{P/2}\widehat J_2+e^{-P/2}\widehat J_3\bigr).
\ee
The two-vector $\Gammaunpar=(\Gammaun_2,\Gammaun_3)$ is the \emph{Killing (parallel) momentum per particle}.

Combining the two parallel Euler equations with the reconstruction identities~\eqref{reconstruct-P-Q} gives the \emph{diagonal transport system}
\bel{equa-jhat-transp-cons}
\del_t\Gammaun_m + \velo\del_x\Gammaun_m = 0, \qquad m=2,3,
\ee
in which the source terms containing $P_0,P_1,Q_0,Q_1$ cancel.
If $\velo$ is locally Lipschitz, its characteristic flow $X(t;s,x)$ is defined by
\bel{eq:section-8-characteristics}
\frac{d}{dt}X(t;s,x)=\velo\bigl(t,X(t;s,x)\bigr),\qquad X(s;s,x)=x,
\ee
and~\eqref{equa-jhat-transp-cons} is equivalent to $\Gammaun_m(t,X(t;s,x))= \Gammaun_m(s,x)$. Thus the full range of the vector $\Gammaunpar$ is conserved by the regular Euler flow. As in~\eqref{eq:Jhat-H}, we also deduce a conservative form which we state below in \autoref{lem-transp-renorm} after describing a similar normalization of twists.


\subsubsection{Killing twists}

The same Killing-frame normalization applies to the twists. Define
\bel{equa-Ktilde}
\Gammazero_2 \coloneqq |t|^{3/2}e^{P/2}K_2,\qquad \Gammazero_3 \coloneqq |t|^{3/2}\bigl(Qe^{P/2}K_2+e^{-P/2}K_3\bigr).
\ee
Equivalently, $\Gammazero=|t|\mathsf M_{\rm K}(K_2,K_3)^T$ using the same triangular change from the moving orbit frame to the fixed Killing frame as for parallel momenta. In vacuum the quantities~$\Gammazero_m$ reduce to the classical twist constants. In the presence of matter they need not be constant in spacetime, but the Einstein evolution and constraint equations give
\bel{eq:T2-9101112-tilde}
\aligned
\del_t\Gammazero_2&=|t|^{3/2}\Omega\, e^{P/2}J_2J_1,\qquad& \del_t\Gammazero_3&=|t|^{3/2}\Omega\bigl(Qe^{P/2}J_2+e^{-P/2}J_3\bigr)J_1,
\\
\del_x\Gammazero_2&=|t|^{3/2}\Omega\,e^{P/2}J_2J_0\amdeux,\qquad& \del_x\Gammazero_3&=|t|^{3/2}\Omega\bigl(Qe^{P/2}J_2+e^{-P/2}J_3\bigr)J_0\amdeux.
\endaligned
\ee
Consequently, the twists obey the same transport equation as the parallel momentum, namely
\bel{equa-KDL3}
\del_t\Gammazero_m+\velo\del_x\Gammazero_m=0,\qquad m=2,3 ,
\ee
describing transport along the characteristics~\eqref{eq:section-8-characteristics}.
The relation between twists and momentum is particularly transparent in the weighted variables:
\bel{eq:tw-mom-rel}
\del_x\Gammazero_m = |t|\Omega\amdeux\Numb_0\Gammaun_m = -\rho \Gammaun_m,\qquad m=2,3.
\ee
Thus the twist is a potential for the Killing momentum density, whereas $\Gammaunpar$ is the corresponding momentum per particle.


\subsubsection{The $H$-divergence law.}

By combining the transport equations of $\Gammaun_m,\Gammazero_m$ with the conservation of~$\vNumb$ we get a family of conservation laws.

\begin{lemma}[Transport and renormalization of the parallel momentum]
\label{lem-transp-renorm}
Let $(\Phi,\Psi)$ be a regular Einstein--Euler areal flow on $[t_0,t_1]\times\Tbb^3$.
Assume that $\mu>0$ (non-vacuum case).  Then the variables $\Gammaun_m$ defined by~\eqref{equawidetildeJ} satisfy~\eqref{equa-jhat-transp-cons}. Moreover, for every locally Lipschitz function $H$ on~$\RR^2$, one has the \textbf{$H$-divergence law}
\bel{eq:transp-renorm}
\divuntrois\bigl(H(\Gammaunpar)\vNumb\bigr)=0.
\ee
Without assumption on the density~$\mu$, the law~\eqref{eq:transp-renorm} holds provided~$H$ is a convex function.
Moreover, the same balance law~\eqref{eq:transp-renorm} holds for $H(\Gammazero)$ any locally Lipschitz function of $\Gammazero$, defined by~\eqref{equa-Ktilde}, without assumption on the density~$\mu$.
\end{lemma}

\begin{proof}
For $\Gammazero$ the algebraic manipulations above are valid because the Killing-frame matrix is bounded and invertible.  For $\Gammaun$ they are only valid away from vacuum as the definition of $\Gammaun$ involves division by~$\hNumb$, which vanishes at vacuum.
The vacuum case is reached by writing $\rho H$ as a function of the conservative variables $\rho,\rho M$ and applying the chain rule.
\end{proof}

\begin{remark}
The $H$-divergence law~\eqref{eq:transp-renorm} can also be written as $\del_t(\rho H) + \del_x(\rho\velo H)=0$. The law includes for the particle-number conservation for $H=1$, the two parallel momentum equations when $H(\Gammaunpar)=\Gammaun_m$, and every nonlinear renormalization of these equations.  For particle-production Navier--Stokes areal flows, we will find that these equations have a non-trivial source term, which has a sign under specific restrictions on the choice of~$H$.
\end{remark}


\subsubsection{Maximum principle.}
To state the following principle, for general density, to avoid dividing by zero one considers
\bel{hNumbGammaun}
\hNumb\Gammaun_2 = |t|^{1/2} e^{P/2} J_2 , \qquad
\hNumb\Gammaun_3 = |t|^{1/2} (Q e^{P/2} J_2 + e^{-P/2} J_3) ,
\ee
and keep the notation $\hNumb\Gammaun_2$ for brevity, understood as the right-hand sides here.

\begin{proposition}[Maximum principle for Einstein--Euler flows]
\label{proposition-JJJ}
Let $(\Phi,\Psi)$ be a sufficiently regular Einstein--Euler areal flow on $[t_0,t_1]\times\Tbb^3$ with $\mu>0$ throughout.
Then, for every $t\in[t_0,t_1]$ and $m=2,3$,
\bel{equa-widetildemax}
\aligned
\inf_{\Sbb^1}\Gammaun_m(t_0,\cdot)&\leq\Gammaun_m(t,x)\leq\sup_{\Sbb^1}\Gammaun_m(t_0,\cdot),\\
\inf_{\Sbb^1}\Gammazero_m(t_0,\cdot)&\leq\Gammazero_m(t,x)\leq\sup_{\Sbb^1}\Gammazero_m(t_0,\cdot).
\endaligned
\ee
Then the range of $\Gammaunpar(t,\cdot)$, respectively $\Gammazero(t,\cdot)$, coincides with the range of the corresponding initial map, namely the range is an invariant set.
The statement about $\Gammazero$ extends to vacuum.
The bound on~$\Gammaun$ extends to vacuum as the statement that if $C_1\hNumb\leq\hNumb\Gammaun_m\leq C_2\hNumb$ at initial time $t=t_0$ for some constants $C_1,C_2\in\RR$ then these inequalities hold for all times.
As a result, the convex hull of the range of $\Gammaunpar(t,\cdot)$ coincides with that at initial time.
\end{proposition}

The proposition controls the normalized geometric quantities directly. The original orthonormal components are recovered by the exact identities
\bel{eq:section-8-reconstruction}
\aligned
\widehat J_2 & = |t|^{-1/2}e^{-P/2}\Gammaun_2, & \qquad
\widehat J_3 & =|t|^{-1/2}e^{P/2}(\Gammaun_3-Q\Gammaun_2),
\\
K_2 & =|t|^{-3/2}e^{-P/2}\Gammazero_2, &
K_3 & =|t|^{-3/2}e^{P/2}(\Gammazero_3-Q\Gammazero_2).
\endaligned
\ee
Hence, pointwise bounds on $P,Q$ obtained later on in \autoref{lem:orb-rec} convert the sharp maximum principle into bounds for $\Jhatpar$ and~$\Kpar$. The distinction is useful: the maximum principle for Killing components is independent of any estimate for $P,Q$, while reconstruction is not.

\subsubsection{Proof of \autoref{proposition-JJJ}}

Step 1: Assume first that either the density $\mu>0$ is positive or that $\Gammaunpar$ admits a bounded continuous extension through vacuum and $\velo$ admits a locally Lipschitz extension there.  
The characteristic flow~\eqref{eq:section-8-characteristics} is a diffeomorphism of the spatial circle at every time. Equations~\eqref{equa-jhat-transp-cons} and~\eqref{equa-KDL3} show that both vector-valued quantities are constant along each characteristic. Their ranges are therefore transported without enlargement, which proves the componentwise inequalities and the invariant-set statement.

\vskip.3cm

Step 2:  Case including the vacuum.   Assume that $\hNumb\Gammaun_m\leq C\hNumb$ at initial time, hence the positive part of their difference vanishes, so that
\be
\int_{\Tbb^3} (\hNumb\Gammaun_m-C\hNumb)^+ \dVtrois = 0 , \qquad t=t_0 .
\ee
The $H$-divergence law, applied to $H=(\Gammaun_m-C)^+$, gives the conservation in time of this integral, thus it vanishes at all times.  From positivity of the integrand one deduces that it vanishes as desired.


\subsubsection{Total variation identity}

The results above have analogues in all three settings of interest to us. We now turn to features that are only partially available for the particle-production Navier--Stokes system. The transport equation on $\Gammaun_m$ implies a transport equation for its weighted spatial derivative
\bel{eq:Kill-der}
\Lambda \coloneqq (\Lambda_2,\Lambda_3) , \qquad
\Lambda_m \coloneqq \rho^{-1} \del_x\Gammaun_m = \frac{1}{|t|\amdeux\Omega\Numb^0} \del_x \Gammaun_m , \qquad m=2,3 .
\ee
Differentiating~\eqref{equa-jhat-transp-cons} gives $\del_t(\del_x\Gammaun_m)+\del_x(\velo\del_x\Gammaun_m)=0$. Comparison with $\del_t\rho+\del_x(\velo\rho)=0$ then yields
\bel{eq:Kill-der-transp}
\del_t\Lambda_m+\velo\del_x\Lambda_m=0,\qquad m=2,3.
\ee
Consequently, for every locally Lipschitz function $H$ on~$\RR^4$, we obtain the first-order $H$-divergence law
\bel{eq:Kill-der-H}
\divuntrois\bigl(H(\Lambda,\Gammaunpar)\vNumb\bigr)=0.
\ee
An interesting choice, for the purposes of obtaining regularity bounds on the parallel momentum, is $H(\Lambda,\Gammaunpar)=|\Lambda_m|$, which yields
\bel{equa-Jx}
\del_t|\del_x\Gammaun_m|+\del_x\bigl(\velo|\del_x\Gammaun_m|\bigr)=0,\qquad m=2,3.
\ee
Since the spatial variable is periodic, integration gives the exact total-variation identity, stated now. 

\begin{proposition}[Total variation principle for Einstein--Euler flows]
\label{proposition-JJJ-TV}
Under the conditions in \autoref{proposition-JJJ}, the following total variation bound holds for sufficiently regular solutions: 
\bel{equa-Jx-Var}
\Var\bigl(\Gammaun_m(t,\cdot);\Sbb^1\bigr)= \Var\bigl(\Gammaun_m(s,\cdot);\Sbb^1\bigr),\qquad s,t\in\Interval.
\ee
\end{proposition}

Let us point out that the above property is not used in the present paper, but plays a role in our companion papers. 
The same conservative equation also gives quantitative continuity in time. Indeed, for $t_0\leq s<t\leq t_1$, and using $|\amdeux\velo|\leq 1$,
\bel{eq:section-8-time-continuity}
\aligned
\|\Gammaun_m(t,\cdot)-\Gammaun_m(s,\cdot)\|_{L^1(\Sbb^1,\amdeux dx)} 
& \leq \| \del_t \Gammaun_m\|_{L^1(\Sbb^1\times[s,t],\amdeux dx\,dt)}
= \| \velo \del_x \Gammaun_m\|_{L^1(\Sbb^1\times[s,t],\amdeux dx\,dt)}
\\
&  \leq \int_s^t \int_{\Sbb^1} |\del_x \Gammaun_m| dx\,dr
= (t-s) \Var\bigl(\Gammaun_m(t,\cdot);\Sbb^1\bigr) .
\endaligned
\ee
Hence the map $t\mapsto\Gammaun_m(t)$ is Lipschitz-continuous with values in $L^1(\Sbb^1,\amdeux dx)$, and its Lipschitz modulus is controlled by the initial total variation.


\subsection{Dissipated-energy Einstein--Navier--Stokes areal flows}
\label{section-7-2}

\subsubsection{Killing twists in torus-adapted Navier--Stokes models}

We now include relaxation terms in the discussion and reproduce analogous exact transport equations and maximum principles, culminating in \autoref{proposition-JJJ-NS}. For now, we consider dissipated-energy models, since these models admit a conserved particle-number current. Throughout, we assume that the Navier--Stokes potential~$\chi$ defining the model is torus-adapted in the sense of \autoref{def-NS-torus-adapted}.
By \autoref{prop-NS-torus-adapted}, this implies that the components of the particle-number current orthogonal to symmetry orbits are proportional to the fluid velocity, and likewise for the mixed orthogonal--parallel block of components of the stress-tensor:
\bel{diss-ENS-factorization}
\vNumb_{\chi\,a} = \hNumb_\chi J_a , \qquad
T_{\chi}(e_a,e_m) = \frac{1}{2} \widehat U_\chi J_a J_m , \qquad
a=0,1, \quad m=2,3
\ee
for some scalar functions $\hNumb_\chi$ and~$\widehat{U}_\chi$ of the momentum and relaxation variables $(\Jbb,\zeta)$.
This factorization ensures that the balance equations for the currents~$\vNumb_\chi$ and $T_\chi(e_m,\cdot)$ have the same speed~$\velo$ as the fluid.
With notation analogous to~\eqref{eq:section-8-Euler-speed}, the particle-number conservation equation $\divuntrois\vNumb_\chi=0$ takes a simple form:
\bel{dtrhochi-diss}
\del_t \rho_\chi + \del_x (\rho_\chi \velo) = 0 , \qquad
\rho_\chi \coloneqq |t| \amdeux \Omega \Numb_\chi^0 > 0 , \qquad
\velo \coloneqq -\amdeux^{-1}\frac{J_1}{J_0} .
\ee
Note that we do not include an index~$\chi$ on the velocity variable, and likewise on the Killing twists below, as the expression of these objects in terms of \JKL{} variables is identical to the Einstein--Euler case.

The Killing twists $(\Gammazero_2,\Gammazero_3)$ are components of $|t|(K_2 e_2+K_3 e_3)$ in the Killing frame $(\del_y,\del_z)$.  Their expression~\eqref{equa-Ktilde} only involves geometric variables, hence remains unchanged here.  Their evolution and constraint equations~\eqref{eq:T2-9101112-tilde} arise from operators $\opK_{20},\opK_{21},\opK_{30},\opK_{31}$ that are corrected by relaxation terms as explained in \autoref{section-6-1}.
Specifically, each product $J_a J_m=2T_{\rm E}(e_a,e_m)$ with $a=0,1$ and $m=2,3$ is corrected to $2T_\chi(e_a,e_m)=\widehat U_\chi J_a J_m$, namely all sources in~\eqref{eq:T2-9101112-tilde} are rescaled by this scalar function.
Consequently, the twists obey the same transport equation~\eqref{equa-KDL3} at speed~$\velo$ as in the Euler case.

\subsubsection{Transport equations for the momentum}

As in the Euler system (in which $2T_{\rm E}(e_0,e_m) = J_0J_m$), we introduce the parallel momentum per particle in the orthonormal frame $e_2,e_3$ and the Killing frame $\del_y,\del_z$,
\be
\aligned
\widehat J_m^\chi & \coloneqq \frac{2T_\chi(e_0,e_m)}{\Numb_{\chi0}}
= \frac{\widehat U_\chi}{\hNumb_\chi} J_m , \qquad m=2,3,
\\
\Gammaun_2^\chi & \coloneqq \frac{2T_\chi(e_0,\del_y)}{\Numb_{\chi0}} = |t|^{1/2}e^{P/2}\widehat J_2^\chi,
\\
\Gammaun_3^\chi & \coloneqq \frac{2T_\chi(e_0,\del_z)}{\Numb_{\chi0}} = |t|^{1/2}\bigl(Qe^{P/2}\widehat J_2^\chi+e^{-P/2}\widehat J_3^\chi\bigr).
\endaligned
\ee
The factorization~\eqref{diss-ENS-factorization} implies that $e_0$ and $e_1$ stress-tensor components involve the same parallel momentum per particle factors:
\be
2T_\chi(e_a,e_m) = \Numb_{\chi a} \widehat J_m^\chi , \qquad
2T_\chi(e_a,\del_y) = \Numb_{\chi a} \Gammaun_2^\chi , \qquad
2T_\chi(e_a,\del_z) = \Numb_{\chi a} \Gammaun_3^\chi ,
\quad
a=0,1 .
\ee


The evolution equations are obtained by noting that contracting the stress-tensor with a Killing vector~$X$ gives a conserved current $X_\alpha T_\chi^{\alpha\bullet}$:
\be
\nabla_\beta\bigl( X_\alpha T_\chi^{\alpha\beta} \bigr)
= X_\alpha \nabla_\beta T_\chi^{\alpha\beta}
+ (\nabla_\beta X_\alpha) T_\chi^{\alpha\beta}
= 0 ,
\ee
with the first term vanishing due to stress-tensor conservation and the second due to the symmetry of~$T_\chi$ and the Killing equation $\nabla_\alpha X_\beta + \nabla_\beta X_\alpha = 0$. We deduce exact conservation of $T_\chi(\del_y,\bullet)$ and $T_\chi(\del_z,\bullet)$.  Since these are vectors, their divergence is given by the $\divuntrois$ operator of~\eqref{equa-def-div13} in terms of their frame components. This yields the exact conservation laws
\bel{conserve-Killing-part-mom-chi}
\del_t( \rho_\chi \Gammaun_m^\chi ) + \del_x( \rho_\chi\velo \Gammaun_m^\chi ) = 0 , \qquad m=2,3
\ee
and combining with the conservation of particle-number~\eqref{dtrhochi-diss} gives a transport equation
\be
\del_t \Gammaun_m^\chi + \velo \del_x \Gammaun_m^\chi = 0 , \qquad m=2,3.
\ee
Passing from these Killing components to the orthonormal frame components~$\widehat J_m^\chi$ is done using the same Killing-frame matrix~\eqref{eq:Kill-mat} as for the Einstein--Euler system.  We deduce that the parallel momentum per particle $(\widehat J_2^\chi, \widehat J_3^\chi)$ satisfies the same linear transport equations~\eqref{eq:Jhat-transp} as in the Einstein--Euler system.


\subsubsection{Maximum principle}

The Killing twists $\Gammazero_m$ and Killing momenta per particle $\Gammaun_m^\chi$ are transported at velocity~$\velo$ by the system considered here.  The rest of \autoref{section-7-1} thus goes through upon changing $\rho\to\rho_\chi$ and $\Gammaun_m\to\Gammaun_m^\chi$. In particular, the maximum principle on $\Gammazero,\Gammaunpar_\chi$ holds as stated below, the reconstruction of~$K_m$ and~$\widehat J_m^\chi$ from these Killing components is unchanged, as is the $H$-divergence law
\be
\divuntrois\bigl(H(\Gammazero,\Gammaunpar_\chi)\vNumb_\chi\bigr)=0 .
\ee
One has $\Gammaun_m^\chi=-\rho_\chi^{-1}\del_x\Gammazero_m$, and the weighted spatial derivatives $\Lambda_m^\chi=\rho_\chi^{-1}\del_x\Gammaun_m^\chi$ obey the same transport equations~\eqref{eq:Kill-der-transp}, which implies that the total variation of $\Gammaun_m^\chi(t,\cdot)$ is constant in time, and controls the Lipschitz modulus of $\Gammaun_m^\chi$ in~$L^1(\Sbb^1,\amdeux dx)$. Owing to its importance we repeat here the same maximum principle as \autoref{proposition-JJJ}, collecting the relevant assumptions.

\begin{proposition}[Maximum principle for dissipated-energy Einstein--Navier--Stokes flows]
\label{proposition-JJJ-NS}
Fix a Navier--Stokes triple $(\mathscr F,\chi,\mathscr B)$ that is admissible (\autoref{def-pot-admissible-triple}) and torus-adapted (\autoref{def-NS-torus-adapted}). Let $(\Phi,\Psi,\zeta)$ be a sufficiently regular dissipated-energy Einstein--Navier--Stokes areal flow on $[t_0,t_1]\times\Tbb^3$, and assume that the density $\mu>0$ is positive. Then the range of the function $\Sbb^1\ni x\mapsto\bigl(\Gammaunpar_\chi(t,x),\Gammazero(t,x)\bigr)$ is a subset of~$\RR^4$ that does not depend on areal time $t\in [t_0,t_1]$.  In particular, for $t\in[t_0,t_1]$ and $m=2,3$, one has 
\be
\label{eq:relax-uncond-max}
\aligned \inf_{\Sbb^1}\Gammaun_m^\chi(t_0,\cdot) & \leq\Gammaun_m^\chi(t,x)\leq\sup_{\Sbb^1}\Gammaun_m^\chi(t_0,\cdot), \\
\inf_{\Sbb^1}\Gammazero_m(t_0,\cdot) & \leq\Gammazero_m(t,x)\leq\sup_{\Sbb^1}\Gammazero_m(t_0,\cdot).\endaligned
\ee
The second line also holds at vacuum.
At vacuum, the first line is understood as the statement that if $C_1\rho_\chi\leq 2T_\chi^{0m} \leq C_2 \rho_\chi$ for some $C_1,C_2\in\RR$ at initial time then this inequality holds for all times.
The convex hull of the range of $\Gammaunpar_\chi$ is thus preserved in time, without density assumption.
\end{proposition}


\subsubsection{Total variation bound.}

Following the same arguments as for the Euler model, we also obtain the following property. Let us again point out that this property is not used in the present paper.  

\begin{proposition}[Total-variation principle for dissipated-energy Einstein--Navier--Stokes flows]
\label{proposition-JJJ-NS-TV}
Under the conditions in \autoref{proposition-JJJ-NS}, the following total variation bound holds for sufficiently regular solutions: 
\bel{equa-Jx-Var-NS}
\Var\bigl(\Gammaun_m^\chi(t,\cdot);\Sbb^1\bigr)= \Var\bigl(\Gammaun_m^\chi(s,\cdot);\Sbb^1\bigr),\qquad s,t\in\Interval.
\ee
\end{proposition}


\subsection{Particle-production Einstein--Navier--Stokes areal flows}
\label{section-7-3}

\subsubsection{Transport equations with sources}

As in the previous section we fix a Navier--Stokes triple $(\mathscr F,\chi,\mathscr B)$ that is admissible (\autoref{def-pot-admissible-triple}) and torus-adapted (\autoref{def-NS-torus-adapted}).  We consider now the particle-production model of \autoref{def-prod-areal-flow}.  By \eqref{eq:pot-areal-production},
\bel{pprod-83}
\divuntrois\vNumb_\chi = \mathfrak d_{\mathscr B} = \zeta_{\alpha\beta}\mathscr B^{\alpha\beta}(\beta,\zeta) \geq 0 .
\ee
The other equations used in \autoref{section-7-2} do not receive any source term. In particular, Killing twists $\Gammazero$ are exactly transported at speed~$\velo$, namely obey~\eqref{equa-KDL3}. To find the transport equation for the Killing momentum per particle~$\Gammaunpar$, we use the $(\rho_\chi,\velo)$ variables in~\eqref{dtrhochi-diss} and express both the particle production equation~\eqref{pprod-83} and the exact conservation law~\eqref{conserve-Killing-part-mom-chi} in these variables:
\bel{eq:pot-Kill-coordinate-laws}
\aligned
\del_t \rho_\chi + \del_x (\rho_\chi \velo) & = \mathfrak d_{\mathscr B}^{\rm ar} \coloneqq |t|\Omega^2\amdeux\mathfrak d_{\mathscr B} \geq 0.
\\
\del_t( \rho_\chi \Gammaun_m^\chi ) + \del_x( \rho_\chi\velo \Gammaun_m^\chi ) & = 0 , \qquad m=2,3 .
\endaligned
\ee
Consequently, for $\rho_\chi>0$ (namely $\mu>0$),
\bel{eq:pot-damped-transport}
\del_t\Gammaun_m^\chi + \velo\del_x\Gammaun_m^\chi=-\frac{\mathfrak d_{\mathscr B}^{\rm ar}}{\rho_\chi}\Gammaun_m^\chi,\qquad \del_t\Gammazero_m+\velo\del_x\Gammazero_m=0.
\ee
We point out that contrary to the previous two models, the equations \eqref{eq:pot-damped-transport} contain a non-zero right-hand side and, consequently, no (useful) total-variation bound is available. 


\subsubsection{Radial maximum principle.}

The source term in~\eqref{eq:pot-damped-transport} is a damping term. It scales $\Gammaun_2^\chi$ and $\Gammaun_3^\chi$ identically.  Consequently, along each characteristic the vector $\Gammaunpar_\chi$ keeps its direction and only its length decreases. We arrive at a maximum principle for the particle-production model.

\begin{proposition}[Maximum principle for particle-production Navier--Stokes flows]
\label{prop-pot-maximum}
Fix a Navier--Stokes triple $(\mathscr F,\chi,\mathscr B)$ that is admissible (\autoref{def-pot-admissible-triple}) and torus-adapted (\autoref{def-NS-torus-adapted}).
Let $(\Phi,\Psi,\zeta)$ be a sufficiently regular particle-production Einstein--Navier--Stokes areal flow $(\Phi,\Psi,\zeta)$ on $[t_0,t_1]\times\Tbb^3$.
If the density $\mu>0$ is positive throughout, in terms of the characteristic flow~\eqref{eq:section-8-characteristics} of $\velo$, one has
\bel{eq:pot-characteristic-contraction}
\aligned
\Gammaunpar_\chi\bigl(t,X(t;t_0,x)\bigr)&=\exp\left(-\int_{t_0}^t\frac{\mathfrak d_{\mathscr B}^{\rm ar}}{\rho_\chi}\bigl(r,X(r;t_0,x)\bigr)\,dr\right)\Gammaunpar_\chi(t_0,x),\\
\Gammazero\bigl(t,X(t;t_0,x)\bigr)&=\Gammazero(t_0,x),
\endaligned
\ee
hence the ranges of $\Gammaunpar_\chi(t,\cdot)$ and $\Gammazero(t,\cdot)$ obey
\bel{eq:pot-radial-hull}
\operatorname{Ran}\Gammaunpar_\chi(t,\cdot)\subset\left\{\lambda z\middle|\ 0\leq\lambda\leq1,\ z\in\operatorname{Ran}\Gammaunpar_\chi(t_0,\cdot)\right\},\qquad \operatorname{Ran}\Gammazero(t,\cdot)=\operatorname{Ran}\Gammazero(t_0,\cdot).
\ee
For flows including the vacuum $\mu=0$, a weaker statement holds: the convex hull of $\{0\}$ and the range of~$\Gammaunpar$ is contained in the same convex hull at initial time.
\end{proposition}

\begin{proof}
Integration along the characteristic flow proves~\eqref{eq:pot-characteristic-contraction}. Since $\mathfrak d_{\mathscr B}^{\rm ar}/\rho_\chi\geq 0$, its exponential factor lies in $(0,1]$, which proves~\eqref{eq:pot-radial-hull}.
To include the vacuum case, one simply uses the $H$-divergence law with $H=(\varpi_m\Gammaun_\chi^m-C)^+$ for some linear form~$\varpi$ and constant~$C>0$.
The additional source term compared to the dissipated-energy formulation has the appropriate sign by convexity of~$H$.
\end{proof}


\subsubsection{Conclusion for this section} 

We thus arrive at the following conclusion, which we state in a condensed form for ease of use in the rest of the paper.

\begin{theorem}[A maximum principle for the three Einstein--fluid models]\label{them:7-4}
For an Einstein--Euler areal flow, particle-production Einstein--Navier--Stokes areal flow, or dissipated-energy Einstein--Navier--Stokes areal flow,
if for some constant $C>0$ one has $2|T_\chi^{02}|+2|T_\chi^{03}|\leq C\Numb_\chi^0$ at initial time, then this inequality holds for all times.
\end{theorem}


\section{Geometric monotonicity properties}
\label{section-8}

\subsection{Positivity hypotheses and equations}
\label{section-8-1}

\subsubsection{Finite-energy estimates}

We now turn from the maximum principles of \autoref{section-7} to the geometric consequences of the equations.
We treat simultaneously the Einstein--Euler system and both Einstein--Navier--Stokes formulations.
The estimates derived here do not use the particle number current, hence the only difference between the particle-production and dissipated-energy settings is in the energy balance law, which we only use as an inequality.
The Einstein--Euler system is a special case of the Einstein--Navier--Stokes systems with the zero bundle, namely with the relaxation variable identically set to $\zeta=0$.

Assuming only the dominant energy condition on the matter stress-tensor, in \autoref{section-8-monotonicity}, we prove monotonicity properties for the volume $\mathcal V(t)$ of spacelike slices, the conformal length $\mathcal L(t)$, and an energy $\mathcal E_\chi(t)$ that combines the fluid mass-energy with a term associated to the twists~$\Kpar$ and first-order geometric variables~$(\Pbb,\Qbb)$,
\bel{V-L-E-expr}
\aligned
\mathcal V(t) & \coloneqq  \int_{\Tbb^3} \dVtrois , \qquad\qquad
\mathcal L(t) \coloneqq \int_{\Sbb^1} d_x\ell ,
\\
\mathcal E_\chi(t) & \coloneqq \int_{\Tbb^3} \bigl( 2 T_\chi^{00} + \Ebf_0(\Kpar,\Pbb,\Qbb) \Bigr) \Omega\,\dVtrois + \frac{3}{2|t|} \mathcal L(t) .
\endaligned
\ee
In the future-contracting regime $t<0$, we also obtain a \emph{collapse criterion} (\autoref{prop:collapse-criterion}) providing two lower bounds on the conformal length $\mathcal L(t)$ in terms of the volume~$\mathcal V(t)$, and alternatively in terms of the initial energy for sufficiently small time intervals.
In \autoref{section-8-bv}, we deduce total variation bounds on the lapse, twist, and conformal length measure.
We then impose on the matter stress-tensor an additional coercive stress condition (\autoref{def-Tchi-coercive}), which ensures that the energy~$\mathcal E_\chi(t)$ norm is equivalent to the spatial $L^2$~norm, and also allows us to obtain bounds on timelike integrals in \autoref{section-8-timelike} with respect to $\dVundeux=|t|\,\Omega \,dt\,dy\,dz$.
These control the finite-energy norms of \autoref{section-3-2} provided the momentum per particle $\Jhatpar_\chi$ (or equivalently~$\Gammaunpar_\chi$) is uniformly bounded.
Altogether, these results assemble into a control of areal flows in finite-energy and bounded variation norms introduced in \eqref{equa-flow-energy-norm}--\eqref{equa-flow-variation-norm}.
In \autoref{section-9} we deduce that the same estimates hold for limiting finite-energy Einstein--Euler areal flows defined there.

\begin{theorem}[Geometric estimates from initial data]
\label{theo:geometric-estimates}
\bse
Fix an admissible Navier--Stokes triple $(\mathscr F,\chi,\mathscr B)$ as defined in \autoref{def-pot-admissible-triple}.
Let $(\Phi,\Psi,\zeta)$ with $\Phi=(\Jperp,\Jpar,\Pbb,\Qbb)$ and $\Psi=(\ell,\log\Omega,\Kpar)$ be a sufficiently regular dissipated-energy or particle-production Einstein--Navier--Stokes areal flow $(\Phi,\Psi,\zeta)$ on $[t_0,t_1]\times\Tbb^3$.  Denote the values at time $t=t_0$ by
\be
(\mathring\Phi,\mathring\Psi,\mathring\zeta) \coloneqq (\Phi,\Psi,\zeta)(t_0,\cdot) \quad \text{on } \Tbb^3 .
\ee
Then the following estimates hold, where implied constants depend on the Navier--Stokes potential $\chi$ and on $t_0,t_1$, but not on the relaxation structure~$\mathscr B$ nor on the solution $(\Phi,\Psi,\zeta)$.
\bei
\item \emph{Spatial energy.} The fluid variables, and the first-order geometric variables $\Pbb,\Qbb$ have
  \bel{estimate-SE}
    \int_{\{t\}\times \Tbb^3} \Bigl( T_\chi^{00} + \Ebf_0(\Pbb,\Qbb) \Bigr) \dVtrois
    \lesssim \|\mathring\Phi\|_{\mathrm{energy}} , \qquad t\in[t_0,t_1] .
  \ee

\item \emph{Conformal length and volume.}
  \bel{estimate-LV}
  \aligned
  & \mathcal L(t) \leq \mathcal L(t_0) , \quad \mathcal V(t) \leq |t/t_0|^{3/4} \mathcal V(t_0) , \qquad t_0 \leq t \leq t_1 < 0 \quad \text{(future-contracting regime)},
  \\
  & \mathcal L(t) \geq \mathcal L(t_0) , \quad \mathcal V(t) \geq |t/t_0|^{3/4} \mathcal V(t_0) , \qquad 0 < t_0 \leq t \leq t_1 \quad \text{(future-expanding regime)}.
  \endaligned
  \ee
  Furthermore, the minimum conformal length $\mathcal L_{\min}=\min_{t\in[t_0,t_1]} \mathcal L(t)\geq 0$ is reached at $t=t_1$ in the future-contracting regime and at $t=t_0$ in the future-expanding regime.
\eei
In the future-contracting regime, subsequent estimates require the assumption that $\mathcal L_{\min}>0$ and depend on $1/\mathcal L_{\min}$, which is \emph{not bounded} by the initial data.  In the future-expanding regime $\mathcal L_{\min}>0$ is part of the initial data.
The constants $C_1,C_2,\dots$ below depend on $\mathcal L_{\min}$ and on the finite-energy and bounded-variation norms of initial data.
\bei
\item \emph{Both-sided bounds on conformal length and volume.}
  \bel{both-sided-LV}
  (C_1)^{-1} \leq \mathcal L(t) \leq C_1 , \qquad (C_2)^{-1} \leq \mathcal V(t) \leq C_2 .
  \ee

\item \emph{Bounded variation in space.} The lapse, twist and conformal length measure satisfy
  \bel{spatial-bv}
    \|\ell\|_{\BVac(0,1)} + \|(\log\Omega,\Kpar)\|_{\BVac(\Tbb^3)}
    \leq C_3 , \qquad t \in [t_0,t_1] .
  \ee

\item \emph{Temporal integrals.} The (not necessarily positive) spatial stress~$T_\chi^{11}$ admits integrals on constant-$x$ slices
\bel{temporal-stress}
\biggl| \int_{[t_0,t_1]\times\Tbb^2} \Bigl( T_\chi^{11} + E_0(\Pbb,\Qbb) \Bigr) \dVundeux \biggr|
\leq C_4 , \quad x\in\Sbb^1 .
\ee
\eei
Assume further that the potential~$\chi$ is torus-adapted (\autoref{def-NS-torus-adapted}) and the stress-tensor satisfies the \emph{coercive stress condition} (\autoref{def-Tchi-coercive}, below), and assume that the parallel momentum per particle is bounded at initial time in the sense that $2|T_\chi^{02}|+2|T_\chi^{03}|\leq C^\parallel\Numb_\chi^0$ for some constant~$C^\parallel>0$.
Then \eqref{temporal-stress} are completed to an $L^2$~control of the whole fluid momentum~$\Jbb$ along timelike slices, and the following estimates are satisfied.
\bei
\item \emph{Bounded variation in time.} The lapse, twist and conformal length measure satisfy
  \bel{bv-time-bound}
    \|\ell\|_{\BVac(\Interval)} + \|(\log\Omega,\Kpar)\|_{\BVac(\Interval\times\Tbb^2)}
    \leq C_5 , \qquad x \in \Sbb^1 .
  \ee
\eei
Altogether, under the assumed coercive stress condition, the finite-energy and bounded-variation norms \eqref{equa-flow-energy-norm}--\eqref{equa-flow-variation-norm} of the areal flow are bounded by a function of its initial data and, in the future-contracting regime, of~$\mathcal L_{\min}$.
\ese
\end{theorem}

\subsubsection{Dominant energy condition}

As a preliminary step it proves important to discuss the energy condition satisfied by the matter stress tensor.
We recall from \autoref{def-pot-hyperbolic} that the stress-tensor is derived as the Hessian of a potential~$\chi$, and that it is assumed to obey the \emph{dominant energy condition}, namely for all future-directed causal vectors~$X,Y$,
\be 
T_\chi^{\alpha\beta} X_\alpha Y_\beta \geq 0 , \qquad \text{future causal } X,Y.
\ee

\begin{lemma}[Consequences of the dominant energy condition]
\label{lem:conseq-DEC}
Given a Navier--Stokes potential~$\chi$ as in \autoref{def-pot-hyperbolic}, one has the following inequalities between components in the orthonormal coframe $e^m$: for $m,n=0,1,2,3$,
\bel{Tchi-DEC}
|T_\chi^{mn}| \leq T_\chi^{00} , \qquad
-4 T_\chi^{00} \leq \eta_{mn} T_\chi^{mn} \leq 2 T_\chi^{00} .
\ee
The Euler stress-tensor with pressure law satisfying~\eqref{hyperbolic-eos-zero} admits the stronger trace bound
\bel{TrTE-lower}
- T_{\rm E}^{00} \leq \eta_{mn} T_{\rm E}^{mn} \leq 2 T_{\rm E}^{00} .
\ee
\end{lemma}

\begin{proof}
For $m\in\{1,2,3\}$, we apply the dominant energy condition to $e^0$ and $e^0\pm e^m$ and obtain $|T_\chi^{0m}|\leq T_\chi^{00}$.  Applying it to the two null covectors $e^0+\alpha e^m$ and $e^0+\beta e^n$ with $\alpha,\beta=\pm 1$ gives $T_\chi^{00}+\alpha\beta T_\chi^{mn} \geq -\alpha T_\chi^{0n} - \beta T_\chi^{0m}$.  Summing the inequalities with $(\alpha,\beta)=(+,-),(-,+)$ gives $T_\chi^{00}\geq T_\chi^{mn}$ and summing those with $(\alpha,\beta)=(+,+),(-,-)$ gives $T_\chi^{00}\geq -T_\chi^{mn}$.  Hence $|T_\chi^{mn}|\leq T_\chi^{00}$. In particular,
\be
-3T_\chi^{00}\leq T_\chi^{11}+T_\chi^{22}+T_\chi^{33}\leq 3T_\chi^{00}.
\ee
This proves~\eqref{Tchi-DEC}. For the Euler tensor with $0\leq p(\mu)$, its trace equals $-\mu+3p(\mu)$, while the time component $T_{\rm E}^{00} = (\mu+p(\mu)) (u^0)^2 - p(\mu) \geq \mu$ due to $(u^0)^2\geq 1$, which proves the improved lower bound in~\eqref{TrTE-lower}.
\end{proof}

\begin{remark}
In dissipated-energy Einstein--Navier--Stokes areal flows, with
\be
T_{\rm aux}^{mn} = \frac{t \Sigma}{2} (\delta_{m=n=2} + \delta_{m=n=3}) , \qquad
\Sigma \geq 0 ,
\ee
the total stress-tensor $T_{\rm NS}=T_\chi+T_{\rm aux}$ obeys most of the inequalities~\eqref{Tchi-DEC}.
Specifically,
\bel{Tchi-DEC-NS-1}
|T_{\rm NS}^{22} - T_{\rm NS}^{33}| \leq 2 T_{\rm NS}^{00} , \qquad
|T_{\rm NS}^{23}| \leq T_{\rm NS}^{00} , \qquad
|T_{\rm NS}^{am}| \leq T_{\rm NS}^{00} , \quad a=0,1, \quad m=0,1,2,3 ,
\ee
while the remaining component and the trace satisfy one-sided bounds that depend on the sign of~$t$,
\bel{Tchi-DEC-NS-2}
\aligned
T_{\rm NS}^{22} & = T_{\rm NS}^{33} \leq T_{\rm NS}^{00} \qquad &
\eta_{mn} T_{\rm NS}^{mn} & \leq 2 T_{\rm NS}^{00} ,
\qquad && t < 0 \quad \text{(future-contracting regime)} ,
\\
T_{\rm NS}^{22} & = T_{\rm NS}^{33} \geq - T_{\rm NS}^{00} \qquad &
\eta_{mn} T_{\rm NS}^{mn} & \geq -4 T_{\rm NS}^{00} ,
\qquad && t < 0 \quad \text{(future-contracting regime)} .
\endaligned
\ee
This will be reflected in one-sided estimates on the energy, due to energy dissipation.
\end{remark}

\subsubsection{Coercive stress condition}

Our energy estimates and timelike integrals only provide the full set of spacelike and timelike finite-energy bounds if the stress-tensor components $T_\chi^{00}$ and $T_\chi^{11}$ are sufficiently coercive.
Due to $T_\chi^{00}\geq T_\chi^{11}$, we can focus on coercivity of the spatial component.
Even in the Euler case, the spatial component $T_{\rm E}^{11}$ explicited in~\eqref{eq:T2-Texpr} only controls $-\Jbb\cdot\Jbb$ and $(J_1)^2$, but cannot control the parallel components~$\Jpar$, regardless of pressure assumptions.

If the Navier--Stokes potential is torus-adapted, these components are controlled by the transport estimates of \autoref{section-7}, which provide sup-norm bounds on the parallel momentum per particle $\widehat J_{\chi\,m}=2T_\chi^{0m}/\Numb_\chi^0$, $m=2,3$, defined there.
We are led to requiring coercivity of $T_{\rm E}^{11}$ in regions of the state space \emph{where $\Jhatpar_\chi$ is bounded.}  This bound is stated in terms of stress-tensor and particle-number components to encompass the vacuum case.

\begin{definition}
\label{def-Tchi-coercive}
The stress-tensor~$T_\chi$ defined from a torus-adapted Navier--Stokes potential~$\chi$ obeys the \textbf{coercive stress condition} if for all constants $C^\parallel>0$, there exists constants $C_1,c_2>0$ such that, for all future causal vectors $\Jbb$ and all $\zeta\in\mathscr F$ satisfying $2|T_\chi^{02}|+2|T_\chi^{03}|\leq C^\parallel\Numb_\chi^0$ (bounded momentum per particle) one has
\be
T_\chi^{11} + C_1 \geq c_2 \bigl(T_\chi^{00} + (J_0)^2\bigr) .
\ee
\end{definition}

\begin{remark}
1. The coercive stress condition on~$\chi$ is an inequality between its derivatives in field space~$(\Jbb,\zeta)$, with no effect of the geometry.  It is expressed here for the specific orthonormal frame adapted to $\Tbb^2$ symmetry, but by covariance it holds for components in any orthonormal frame.

2. Note that $(J_0)^2$ controls all components of~$\Jbb$.  Neither a control of $T_\chi^{00}$ nor of~$(J_0)^2$ would be sufficient on its own, as away from the Euler equilibrium $\zeta=0$ the mass-energy $T_\chi^{00}$ could be either less coercive or include faster-growing terms.  The constant shift~$C_1$ is allowed here thanks to the geometric terms in~$\mathcal E_\chi$ defined in~\eqref{V-L-E-expr}.

3. For the Euler stress-tensor, the condition holds thanks to $0\leq p(\mu)\lesssim\mu$.  Indeed, \eqref{eq:section-4-particle-density} together with these bounds ensures that $\Numb(\mu)\lesssim\mu$ at infinity and the bound of $T_{\rm E}^{0m}=(\mu+p(\mu))u^0 u^m$ by $\Numb_{\rm E}^0=\Numb(\mu)u^0$ then implies $|\Jpar|\lesssim \mu^{1/2}$.
Given the explicit form~\eqref{eq:T2-Texpr}, $2T_{\rm E}^{11}=(J_1)^2 + \frac{1}{2}(1-q_\Jbb)(-\Jbb\cdot\Jbb) = (J_1)^2+\mu$, which thus controls both $(J_1)^2$ and $|\Jpar|^2$, as well as $(-\Jbb\cdot\Jbb)=2(\mu+p)$, hence their sum $(J_0)^2$.
Importantly, we cannot ask for such a coercivity for all~$\mu$ as $\Numb(\mu)$ is generally not bounded by~$\mu$ at small densities, as seen in~\eqref{near-vacuum-Numb}.
\end{remark}

\subsubsection{The system of equations}

The evolution and constraint equations of interest were determined in \autoref{section-6-1}, and involve the total stress-energy tensor $T_{\rm NS} = T_\chi + T_{\rm aux}$, with the latter term vanishing in the Euler and particle-production Navier--Stokes settings.
The auxiliary tensor $T_{\rm aux}$ is described by a single scalar,
\be
\Sigma \coloneqq \begin{cases}
  0 & \text{(Euler, or particle-production Navier--Stokes)}, \\
  2\Omega(\beta^0)^{-1}\mathfrak d_{\mathscr B}\geq 0 & \text{(dissipated-energy Navier--Stokes)} .
\end{cases}
\ee
We shall only use that $\Sigma$ is non-negative, as it holds in all three systems of interest.
We can thus describe uniformly the set of equations (and energy inequality) satisfied by all models as
\be
\aligned
\opL_{00}^\chi = \opL_{01}^\chi = \opL_{10}^\chi = \opL_{11}^\chi & = 0, \qquad &
\opJ^1_\chi = \opJ^2_\chi = \opJ^3_\chi & = 0, \qquad &
\opJ^0_\chi & \leq 0 ,
\\
\opC^\chi = \opO_0^\chi = \opK_{20}^\chi = \opK_{30}^\chi & = 0 , \qquad &
\opO_1^\chi = \opK_{21}^\chi = \opK_{31}^\chi & = 0 ,
\endaligned
\ee
with the operators in \eqref{eq:expl-NS-op} and~\eqref{eq:expl-more-NS-op}.

In \autoref{section-6-1}, we introduced notation $\frac{1}{2}\mathcal S_{mn}$ for the frame components of the relaxation stress-tensor $T_{\chi_{\rm relax}} = T_\chi - T_{\rm E}$, as it enabled us to describe the operators as corrections of the Einstein--Euler equations.
Accordingly, we have introduced in~\eqref{Mbfchi-expr} the geometry-matter tensor $\Mbf_\chi^{\bullet\bullet} = \Mbf^{\bullet\bullet} + \mathcal S^{\bullet\bullet}$ and a source which we rewrite here in brief as
\bel{Mchi-fluid-geom}
\aligned
\Mbf_\chi^{ab}(\Jbb,\Kpar,\Lbb,\zeta) & = \Mbf^{ab}(0,\Kpar,\Lbb) + 2 T_\chi^{ab}(\Jbb,\zeta) , \qquad a,b=0,1 ,
\\
\Msource^\chi(\Jbb,\Kpar,\Lbb,\zeta)
& = \Msource(0,\Kpar,\Lbb) + \bigl( 2 \eta_{mn} T_\chi^{mn} + 3 (T_\chi^{00} - T_\chi^{11}) \bigr)(\Jbb,\zeta) .
\endaligned
\ee
The main operators $\opC^\chi,\opO_0^\chi,\opO_1^\chi,\opJ^0_\chi,\opJ^1_\chi$ used in this section are written explicitly in~\eqref{using-Mbfchi-expr} using this notation.

\begin{lemma}[Control of the geometry-matter tensor]
\label{lem:conseq-DEC-M}
Given a Navier--Stokes potential~$\chi$ as in \autoref{def-pot-hyperbolic}, one has the following inequalities: for $a,b\in\{0,1\}$,
\bel{Tchi-DEC-M}
|\Mbf_\chi^{ab}| \leq \Mbf_\chi^{00} , \qquad
-2 \Mbf_\chi^{00} \leq \frac{1}{2} \Msource^\chi - \frac{3}{4}(\Mbf_\chi^{00} - \Mbf_\chi^{11}) \leq \Mbf_\chi^{00} .
\ee
For the Euler stress-tensor the last lower bound is improved to $-\Mbf_\chi^{00}$.
\end{lemma}

\begin{proof}
The inequalities are obtained separately for the fluid and geometry parts of~\eqref{Mchi-fluid-geom}.
From the expressions~\eqref{eq:T2-Mdef-0} and~\eqref{eq:T2-Euler-perp2-def0}, the geometric terms are
\be
\aligned
\Mbf^{00}(0,\Kpar,\Lbb) & = \Ebf_0(\Kpar) + \Ebf_0(\Pbb) + \Ebf_0(\Qbb) ,
\\
\Mbf^{01}(0,\Kpar,\Lbb) & = - \Ebf_1(\Pbb) - \Ebf_1(\Qbb),
\\
\Mbf^{11}(0,\Kpar,\Lbb) & = - \Ebf_0(\Kpar) + \Ebf_0(\Pbb) + \Ebf_0(\Qbb) ,
\\
\Bigl(\frac{1}{2} \Msource - \frac{3}{4}\Mbf^{00} + \frac{3}{4}\Mbf^{11}\Bigr)(0, \Kpar, \Lbb) & = \Ebf_0(\Kpar) - \frac{1}{2} \Pbb \cdot \Pbb - \frac{1}{2} \Qbb \cdot \Qbb ,
\endaligned
\ee
which are all bounded in absolute value by $\Mbf^{00}(0,\Kpar,\Lbb)$ thanks to $\Ebf_0(\Kpar)\geq 0$ and $\Ebf_0(\Xbb)\geq|\Ebf_1(\Xbb)|\geq 0$ and $|\Xbb\cdot\Xbb|=|X_0^2-X_1^2| \leq 2\Ebf_0(\Xbb) = X_0^2+X_1^2$ for $\Xbb=\Pbb,\Qbb$.
Together with \autoref{lem:conseq-DEC} this gives the desired inequalities.  In contrast to \autoref{lem:conseq-DEC} where in the Euler case the lower bound is improved by a factor of~$4$ compared to the Navier--Stokes case, here, it only gains a factor of~$2$ due to geometric terms.
\end{proof}

\subsection{Monotonicity properties}
\label{section-8-monotonicity}

\subsubsection{Future-contraction, future-expansion, and weighted energy decay}

The proofs of the following monotonicity property (\autoref{prop:monotonicity}), as well as a pair of lower bounds on the conformal length (\autoref{prop:collapse-criterion}) which serve as a criterion for the collapse of spacetime in the future-contracting case, occupy the rest of this subsection.
We also provide in \autoref{lem:orb-rec} some bounds on the metric variables $P,Q$ used in comparing the orthonormal frame $e_2,e_3$ with the Killing vectors in \autoref{section-7}.
We recall the energy~$\mathcal E_\chi$, conformal length $\mathcal L=\int d_x\ell$ and volume~$\mathcal V$ of time slices defined in~\eqref{V-L-E-expr}, as well as the area~$|t|$ of symmetry orbits, whose monotonicity is stated for completeness.

\begin{proposition}[Monotonicity for Einstein--matter areal flows]
\label{prop:monotonicity}
1. In the future-contracting regime $[t_0,t_1]\subset(-\infty,0)$, the conformal length~$\mathcal L$, the orbit area~$|t|$, the volume~$\mathcal V$ and its rescaling $|t|^{-3/4}\mathcal V$, and the energy $|t|\mathcal E_\chi$ are monotonically non-increasing in time.

2. In the future-expanding regime $[t_0,t_1]\subset(0,+\infty)$, the conformal length~$\mathcal L$, orbit area~$t$, volume~$\mathcal V$ and rescaled volume~$t^{-3/4}\mathcal V$ are monotonically non-decreasing, while the energy $t^{-2}\mathcal E_\chi$ is monotonically non-increasing.  In the Euler case, the exponent of~$t$ in the energy bound can be improved, namely $t^{-1}\mathcal E_{\rm E}$ is monotonically non-increasing.

As a result, the upper bound~\eqref{estimate-SE} on the energy~$\mathcal E_\chi(t)$ holds, and the upper/lower bounds~\eqref{estimate-LV} on $\mathcal L(t),\mathcal V(t)$ hold in the future-contracting/expanding regimes, respectively.
\end{proposition}

\begin{proposition}[Collapse criterion]
\label{prop:collapse-criterion}
\bse
In the future-contracting case $[t_0,t_1]\subset(-\infty,0)$, the conformal length admits two (non-negative) lower bounds.
\bei
\item In terms of the scaled volume $\mathcal V_1$ defined in~\eqref{eq:section-8-volume-functionals}, below,
\bel{non-collapse-1}
\mathcal L(t) \geq \frac{|t_0/t|^{3/2}}{\mathcal V_1(t_0)} \mathcal V(t)^2 , \qquad t_0 \leq t \leq t_1 < 0 .
\ee

\item In a restricted time interval with end-point $t_2 = - |t_0|^{1/3} \bigl( |t_0| - \frac{3}{2} \mathcal L(t_0) / \mathcal E_\chi(t_0) \bigr)^{2/3} < 0$,
\bel{non-collapse-2}
\mathcal L(t)
\geq \frac{2}{3} |t_0| \mathcal E_\chi(t_0)
- \Bigl( \frac{2}{3} |t_0| \mathcal E_\chi(t_0) - \mathcal L(t_0) \Bigr) \Bigl|\frac{t_0}{t}\Bigr|^{3/2} ,
\qquad t_0 \leq t \leq \min(t_1,t_2) < 0 .
\ee
\eei
\ese
\end{proposition}

\subsubsection{Monotonicity of conformal length and volume}

Consider the conformal length measure $d_x\ell= \amdeux\,dx$ and the volume form $\dVtrois=|t|\Omega\,d_x\ell\,dy\,dz$ induced by the metric on areal time slices.
The conformal length $\mathcal L=\int d_x\ell$ and volume~$\mathcal V$ of time slices defined in~\eqref{V-L-E-expr} are part of a one-parameter family of functions defined by
\bel{eq:section-8-volume-functionals}
\aligned
\mathcal V_m(t) & \coloneqq \int_{\Tbb^3} \bigl(|t|\Omega\bigr)^m\dVtrois, \qquad m\in\RR , \quad t\in[t_0,t_1] ,
\\
\mathcal L(t) & = \mathcal V_{-1}(t) , \qquad \mathcal V(t) = \mathcal V_0(t) .
\endaligned
\ee

The evolution equations~\eqref{using-Mbfchi-expr} for $\amdeux$ and $\Omega$ give, for every $m\in\RR$ the exact identity
\bel{eq:section-8-weighted-density}
\del_t\bigl(|t|^{(m-3)/4}\Omega^m\dVtrois\bigr)
= \frac{t\,\Omega^2}{2}\bigl(\Mbf_\chi^{00}+m\Mbf_\chi^{11}\bigr)|t|^{(m-3)/4}\Omega^m\dVtrois.
\ee
The power of~$|t|$ is chosen to cancel a $1/(4t)$ term in the evolution operator~$\opO_0^\chi$ in~\eqref{using-Mbfchi-expr}.
At this stage, we rely on the dominant energy condition: as stated in \autoref{lem:conseq-DEC-M}, $|\Mbf_\chi^{ab}| \leq \Mbf_\chi^{00}$ for $a,b\in\{0,1\}$.
The source factor $\Mbf_\chi^{00}+m\Mbf_\chi^{11}$ in~\eqref{eq:section-8-weighted-density} is thus non-negative for $m\in[-1,1]$.
Upon integrating on~$\Sbb^1$, this implies the monotonicity property
\bel{sgntdtVm}
\sgn(t) \, \del_t \Bigl( |t|^{-(3/4)(m+1)} \mathcal V_m(t) \Bigr)
= \frac{1}{2} |t|^{(m+1)/4} \int_{\Tbb^3} \bigl(\Mbf_\chi^{00}+m\Mbf_\chi^{11}\bigr)\Omega^{m+2}\dVtrois \geq 0 .
\ee
Three values are of particular interest.
For $m=-1$ and $m=0$ respectively, we get the monotonicity of the conformal length $\mathcal L=\mathcal V_{-1}$ and of $|t|^{-3/4}\mathcal V$ stated in \autoref{prop:monotonicity}, and the resulting bounds in \autoref{theo:geometric-estimates}.
In the future-contracting regime, we also use the $m=1$ case to get~\eqref{Vt2leqLV1}, below.

\subsubsection{Monotonicity of the energy}
The energy, which was defined in~\eqref{V-L-E-expr} explicitly as a sum of fluid and geometric components, admits two other interesting expressions,
\bel{alternative-Echi}
\mathcal E_\chi
= \int_{\Tbb^3} \Bigl( \Mbf_\chi^{00} + \frac{3}{2t^2\Omega^2} \Bigr) \Omega \dVtrois
= \frac{2}{t} \int_{\Tbb^3} \mathcal L_{e_0}(\dVtrois) .
\ee
The expression as a Lie derivative can be verified using the evolution equations of $\amdeux$ and~$\Omega$, with the operators explicited in~\eqref{using-Mbfchi-expr}.
We now show monotonicity in time of $|t|^m\mathcal E_\chi$ for suitable ranges of parameters $m\in\RR$ depending on the sign of~$t$.

The regular mass-energy inequality $\opJ^0_\chi\leq 0$ (which is an equality in the Euler setting or the particle-production setting), combined with the evolution equation $\opC^\chi=0$ of~$\amdeux$ to account for the derivative of the $3/(2t^2)$ term below, gives a balance inequality for the integrand of the energy~$\mathcal E_\chi$.  Namely, for every $\gamma\in\RR$,
\bel{ineq-wt-en}
\aligned
&|t|^{-\gamma}\del_t\!\left(|t|^{\gamma+1}\left(\Mbf_\chi^{00}\Omega^2+\frac{3}{2t^2}\right)d_x\ell\,dy\,dz\right)+|t|\del_x\!\left(\Mbf_\chi^{01}\Omega^2\right)dx\,dy\,dz\\
&\quad \leq \sign(t)\Bigl(
\gamma \Mbf_\chi^{00} - \frac{1}{2} \Msource^\chi + \frac{3}{4} (\Mbf_\chi^{00} - \Mbf_\chi^{11}) \Bigr)\Omega^2d_x\ell\,dy\,dz.
\endaligned
\ee
Upon integrating on the spatial circle, we deduce
\be
\del_t\bigl(|t|^{\gamma}\mathcal E_\chi\bigr)
\leq \sign(t) |t|^{\gamma-1} \int \Bigl(
\gamma \Mbf_\chi^{00} - \frac{1}{2} \Msource^\chi + \frac{3}{4} (\Mbf_\chi^{00} - \Mbf_\chi^{11}) \Bigr)\Omega \dVtrois .
\ee
In \autoref{lem:conseq-DEC-M} we bounded $\frac{1}{2} \Msource^\chi - \frac{3}{4}(\Mbf_\chi^{00} - \Mbf_\chi^{11})$ in the range $[-2\Mbf_\chi^{00},\Mbf_\chi^{00}]$ using the dominant energy condition, with the lower bound being improved to $-\Mbf_\chi^{00}$ in the Euler setting.
This yields the monotonicy results stated in \autoref{prop:monotonicity}.
\bei
\item In the future-contracting regime $t<0$, taking $\gamma=1$ gives that $|t|\mathcal E_\chi(t)$ is monotonically non-increasing.
\item In the future-expanding regime $t>0$, taking $\gamma=-2$ gives that $|t|^{-2}\mathcal E_\chi(t)$ is monotonically non-increasing.  In the Euler setting one can take $\gamma=-1$ and get the announced monotonicity of $|t|^{-1}\mathcal E_\chi(t)$.
\eei
This entails the following bounds, in which $|t/t_0|^2$ is replaced by $|t/t_0|$ in the Euler setting:
\bel{Echit-from-Echit0}
\mathcal E_\chi(t)
\leq \begin{cases}
  |t_0/t| \mathcal E_\chi(t_0) , & t_0 \leq t\leq t_1 < 0 \quad \text{(future-contracting regime)} , \\
  |t/t_0|^2 \mathcal E_\chi(t_0) , & 0 < t_0 \leq t\leq t_1 \quad \text{(future-expanding regime)} .
\end{cases}
\ee

\subsubsection{Proof of \autoref{prop:collapse-criterion}}

In the future-contracting regime $[t_0,t_1]\subset(-\infty,0)$, we prove two lower bounds on~$\mathcal L$.  Using~\eqref{sgntdtVm} for $m=1$ gives that $|t|^{-3/2}\mathcal V_1$ is non-increasing, hence $\mathcal V_1(t)\leq|t/t_0|^{3/2}\mathcal V_1(t_0)$, which is used as follows.
The Cauchy--Schwartz inequality implies that
\bel{Vt2leqLV1}
\mathcal V(t)^2 = \mathcal V_0(t)^2
\leq \mathcal V_1(t) \mathcal V_{-1}(t) = \mathcal V_1(t) \mathcal L(t)
\leq |t/t_0|^{3/2} \mathcal V_1(t_0) \mathcal L(t),
\ee
which together with the upper bound on~$\mathcal V_1(t)$ gives the lower bound~\eqref{non-collapse-1} on the conformal length.
This lower bound can alternatively be understood as an upper bound on the volume of timelike slices: if in some limit (with uniformly bounded initial data) $\mathcal L(t)\to 0$ at some time, then $\mathcal V(t)\to 0$ at that time.

A second bound is obtained by returning to the details of the derivative~\eqref{sgntdtVm} for $m=-1$, rather than only its sign.  By bounding $|\Mbf_\chi^{11}|\leq\Mbf_\chi^{00}$ we obtain
\be
- \del_t \mathcal L(t)
\leq \int_{\Tbb^3} \Mbf_\chi^{00} \Omega\,\dVtrois
= \mathcal E_\chi(t) - \frac{3}{2|t|} \mathcal L(t) , \qquad t_0 \leq t \leq t_1 < 0 .
\ee
Due to the bound on~$\mathcal E_\chi(t)$ in~\eqref{Echit-from-Echit0}, we deduce a negative lower bound on $\del_t \bigl(|t|^{3/2} \mathcal L(t)\bigr)$, whose integration on a time interval $[t_0,t]$ gives the inequality
\bel{lower-Lt-23}
\mathcal L(t)
\geq \frac{2}{3} |t_0| \mathcal E_\chi(t_0)
- \Bigl( \frac{2}{3} |t_0| \mathcal E_\chi(t_0) - \mathcal L(t_0) \Bigr) \Bigl|\frac{t_0}{t}\Bigr|^{3/2} .
\ee
The combination in parentheses is non-negative as
\be
\frac{2}{3} |t_0| \mathcal E_\chi(t_0) - \mathcal L(t_0)
= \biggl(\int_{\Tbb^3} \Mbf_\chi^{00} \Omega\,\dVtrois\biggr)(t_0) \geq 0 .
\ee
As a result, since $|t_0/t|^{3/2}\to+\infty$ as $t\uparrow 0$, the lower bound in \eqref{lower-Lt-23} is only non-negative for a limited range of times $t\in [t_0,\min(t_1,t_2)]$ with $t_2$ being the lower bound vanishes, explicited above~\eqref{non-collapse-2}.

\subsection{Total variation in space}
\label{section-8-bv}

\subsubsection{Integrating the constraints}

We consider the three constraint equations $\opO_1^\chi=\opK_{21}^\chi=\opK_{31}^\chi=0$ with the operators given in~\eqref{opchi-constr} as corrections of the Einstein--Euler constraint operators~\eqref{eq:theconstraints-def}.
Assembling these identities gives
\be
\aligned
\opO_1 & \coloneqq (\log\Omega)_x + \frac{t}{2} \, \bigl( 2 T_\chi^{01} - \Ebf_1(\Pbb) - \Ebf_1(\Qbb) \bigr) \, \Omega^2 \amdeux ,
\\ 
\opK_{21} & \coloneqq ( K_2 )_x
+ \bigl( 2 T_\chi^{02} + P_1 K_2/2 \bigr) \, \Omega \amdeux ,
\\ 
\opK_{31} & \coloneqq ( K_3 )_x
+ \bigl( 2 T_\chi^{03} + Q_1 K_2 - P_1 K_3/2 \bigr) \, \Omega \amdeux .
\endaligned
\ee
The factors in parentheses are bounded using \autoref{lem:conseq-DEC} thanks to the dominant energy condition for the fluid part:
\be
\aligned
|2 T_\chi^{01} - \Ebf_1(\Pbb) - \Ebf_1(\Qbb)|
& \leq 2T_\chi^{00} + \Ebf_0(\Pbb) + \Ebf_0(\Qbb) \leq \Mbf_\chi^{00} ,
\\
|2 T_\chi^{02} + P_1 K_2/2|
& \leq 2T_\chi^{00} + \frac{1}{4} (P_1)^2 + \frac{1}{4} (K_2)^2 \leq \Mbf_\chi^{00} ,
\\
|2 T_\chi^{03} + Q_1 K_2 - P_1 K_3/2|
& \leq 2T_\chi^{00} + \frac{1}{2} (Q_1)^2 + \frac{1}{4} (P_1)^2 + \frac{1}{2} (K_2)^2 + \frac{1}{4} (K_3)^2 \leq \Mbf_\chi^{00} .
\endaligned
\ee

We arrive at bounds on the total variation on a given time slice: accounting for the volume factor $\dVtrois=|t|\,\Omega\,\amdeux\,dx\,dy\,dz$,
\bel{varlogOmega-obtain}
\aligned
\Var(\log\Omega(t,\cdot);\Tbb^3)
= \int_{\Sbb^1} \bigl| (\log\Omega)_x \bigr| dx
& \leq \frac{1}{2} \int \Mbf_\chi^{00} \Omega \dVtrois \leq \frac{1}{2} \mathcal E_\chi(t) ,
\\
\Var(K_m(t,\cdot);\Tbb^3)
= \int_{\Sbb^1} \bigl| (K_m)_x \bigr| dx
& \leq \int \Mbf_\chi^{00} \dVtrois \leq \mathcal E_\chi(t) \sup_{\Tbb^3} \Omega^{-1} , \quad m=2,3 .
\endaligned
\ee
For the twist variables the scaling in~$\Omega$ is different, and in the last step we bounded the integral of $\Mbf_\chi^{00}\dVtrois$ by the sup-norm of $\Omega^{-1}$ times the integral of $\Mbf_\chi^{00} \Omega \dVtrois$.
The energy $\mathcal E_\chi(t)$ is bounded in terms of the initial data, but not yet~$\Omega^{-1}$.

\subsubsection{Pointwise bound on the lapse}

The bounds~\eqref{varlogOmega-obtain} must be complemented by sup-norm bounds.
At this stage we introduce the minimum conformal length
\bel{Lmin-expr}
\mathcal L_{\min} \coloneqq \min_{t\in[t_0,t_1]} \mathcal L(t) > 0 .
\ee
The following lemma provides the bound on the lapse needed for~\eqref{spatial-bv}, and the bounds on volume stated in~\eqref{both-sided-LV}, hence those on the conformal length.

\begin{lemma}[Volume and lapse bounds]
\label{lem-section-8-Euler-lapse-closure}
Suppose that $\mathcal L_{\min}>0$. Then there exists $C\geq1$, depending only on the time interval, $\mathcal L_{\min}^{-1}$, $\mathcal V(t_0)$, $\mathcal L(t_0)$, and $\mathcal E_\chi(t_0)$, such that
\bel{eq:Eul-vol-lap}
C^{-1}\leq |t|^{-3/4}\mathcal V(t)\leq C,
\qquad
C^{-1}\leq\Omega(t,x)\leq C ,
\ee
for every $(t,x)\in[t_0,t_1]\times\Sbb^1$.
\end{lemma}

\begin{proof}
\bse
We denote in this proof by $C_i$ for $i=0,1,\dots$ any constant that depends on the time interval, $\mathcal L_{\min}^{-1}$, $\mathcal V(t_0)$, $\mathcal L(t_0)$, and $\mathcal E_\chi(t_0)$.

First, we learn from~\eqref{varlogOmega-obtain} that values of $\log\Omega$ at different spatial points on the same time slice differ by at most $\frac{1}{2}\mathcal E_\chi(t)$, hence
\bel{Omega-over-Omega}
\exp(-\mathcal E_\chi(t)/2) \leq \frac{\Omega(t,x)}{\Omega(t,x')} \leq \exp(\mathcal E_\chi(t)/2) , \qquad x,x'\in\Sbb^1, \quad t\in [t_0,t_1] .
\ee
Then, since $\dVtrois=|t|\Omega\,d_x\ell\,dy\,dz$, the ratio $|t|^{-1}\mathcal V/\mathcal L$ is a weighted average of~$\Omega$, so any value of~$\Omega$ on the spatial slice differs from it by at most the exponential factors in~\eqref{Omega-over-Omega}.
This yields the two-sided estimate
\bel{eq:section-8-lapse-bound}
\frac{|t|^{-1}\mathcal V(t)}{\mathcal L(t)}e^{-\mathcal E_\chi(t)/2}\leq\Omega(t,x)\leq \frac{|t|^{-1}\mathcal V(t)}{\mathcal L(t)}e^{\mathcal E_\chi(t)/2}.
\ee

We set $Y(t)\coloneqq |t|^{-3/4}\mathcal V(t)$.
The energy estimate gives $\sup_t\mathcal E_\chi(t)\leq C_0$.  The conformal length is bounded above by the energy and below by $\mathcal L_{\min}$ by definition, hence $(C_1)^{-1} \leq \mathcal L(t) \leq C_1$.
From~\eqref{eq:section-8-lapse-bound} we thus have
\bel{eq:section-8-lapse-via-Y}
\|\Omega(t,\cdot)\|_{L^\infty}\leq |t|^{-1/4}\frac{Y(t)}{\mathcal L(t)}e^{\mathcal E_\chi(t)/2} \leq C_2 Y(t) .
\ee
The identity~\eqref{sgntdtVm}, for $m=0$, gives
\be
|Y'(t)|
= \frac{1}{2} |t|^{1/4} \int_{\Tbb^3} \Mbf_\chi^{00}\Omega^2\dVtrois
\leq \frac{1}{2} |t|^{1/4} \mathcal E_\chi(t) \|\Omega(t,\cdot)\|_{L^\infty}
\leq C_3 Y(t) .
\ee
Gronwall's inequality gives the two-sided bound for~$Y$, hence on~$\mathcal V$ and, by~\eqref{eq:section-8-lapse-bound}, on the lapse.
\ese
\end{proof}

\subsubsection{Twists and conformal length function}

The variation bound~\eqref{varlogOmega-obtain} and~\eqref{eq:section-8-lapse-bound} together give the control of $\log\Omega$ stated in~\eqref{spatial-bv}.
The bounds on the variation of twists in~\eqref{varlogOmega-obtain} then become bounded by functions of the initial data and of~$\mathcal L_{\min}$.  There remains to bound some integral of the twist: this is done by noting that
\be
\frac{\int_{\Tbb^3} E_0(\Kpar) \Omega \dVtrois}{\int_{\Tbb^3} \Omega \dVtrois} \leq \frac{\mathcal E_\chi(t)}{\mathcal V(t)} \sup \Omega^{-1}(t,\cdot) ,
\ee
and all factors on the right-hand side are controlled by the initial data and~$\mathcal L_{\min}$.  This controls a weighted average of $E_0(\Kpar)$.  Together with the variation bound we deduce the claimed pointwise bound on the twists.

Finally, the function $\ell$ on $[0,1]$ is a spatial primitive of~$\amdeux$, namely $d_x\ell=\amdeux dx$.  Since $\amdeux\geq 0$, we trivially get a total variation bound in terms of the conformal length functional~$\mathcal L$, itself bounded by the energy:
\be
\Var(\ell(t,\cdot);[0,1]) = \int_0^1 |\del_x\ell| dx = \int_0^1 \amdeux dx = \mathcal L \leq \frac{2|t|}{3} \mathcal E_\chi .
\ee
This completes the proof of~\eqref{spatial-bv}.

\subsubsection{Reconstruction of the orbit metric}

In \autoref{section-7} we considered parallel momenta per particle in two frames: the orthonormal frame $e_2,e_3$, and the Killing frame $\del_y,\del_z$.  These are related by matrices~\eqref{eq:Kill-mat} involving the metric components $P,Q$ that are seen in our formalism as reconstructed variables.  We provide here the relevant bounds.

\begin{lemma}[Reconstruction of the orbit metric]
\label{lem:orb-rec}
There exists $C>0$ depending only on the quantities in the preceding estimates and on $\|P(t_0,\cdot)\|_{L^\infty}+\|Q(t_0,\cdot)\|_{L^\infty}$ such that
\bel{eq:section-8-PQ-reconstruction}
\sup_{[t_0,t_1]\times\Sbb^1}(|P|+|Q|)+\sup_{[t_0,t_1]\times\Sbb^1}\left(\|\mathsf M_{\rm K}\|+\|\mathsf M_{\rm K}^{-1}\|\right)\leq C .
\ee
\end{lemma}

\begin{proof}
\bse
At a fixed time, the reconstruction equations~\eqref{reconstruct-P-Q} and the Cauchy--Schwarz inequality give
\bel{eq:P-sp-rec}
\Var(P(t,\cdot);\Sbb^1)\leq\int_{\Sbb^1}\Omega\amdeux|P_1|\,dx\leq\left(\int_{\Sbb^1}\Omega^2\amdeux P_1^2\,dx\right)^{1/2}\mathcal L(t)^{1/2} .
\ee
At every fixed spatial point, the timelike estimate gives
\bel{eq:P-ti-rec}
|P(t,x)-P(t_0,x)|\leq\left(\int_{t_0}^{t_1}|s|\Omega^2P_0^2(s,x)\,ds\right)^{1/2}\left(\int_{t_0}^{t_1}|s|^{-1}\,ds\right)^{1/2}.
\ee
The spacelike and timelike bounds therefore control $P$ pointwise. Since $Q_x=e^{-P}\Omega\amdeux Q_1$ and $Q_t=e^{-P}\Omega Q_0$, the same two estimates, with the now bounded factor $e^{-P}$, control $Q$. Finally, the explicit matrices in~\eqref{eq:Kill-mat} are bounded together with their inverses because the interval stays away from $t=0$. This proves~\eqref{eq:section-8-PQ-reconstruction}.
\ese
\end{proof}

\subsection{Timelike integrals}
\label{section-8-timelike}

\subsubsection{Timelike integral of the stress}

The first fluid equation $\opJ^1_\chi=0$ is the spatial component $e^1_\beta \nabla_\alpha T_\chi^{\alpha\beta}=0$ of the stress-tensor conservation law.  Combined with the earlier estimates, it provides a total variation bound on~$\Sbb^1$ for the integral of~$T_\chi^{11}$ (corrected by geometric terms) along constant-$x$ timelike slices, which is completed by controlling the space-time integral of~$T_\chi^{11}$ by that of~$T_\chi^{00}$, itself bounded by the spatial energy estimates.
Note that under the present assumptions, the spatial stress $T_\chi^{11}$ is not assumed to be positive, so the function~$\mathcal F_{[t_2,t_3]}$ defined next need not be positive.
The proof follows in the rest of the section.

\begin{proposition}[Timelike slice estimate]
\label{prop-section-8-timelike-slices}
For any interval $[t_2,t_3]$ contained in $\Interval=[t_0,t_1]$, consider the function on~$\Sbb^1$ defined by
\be
\mathcal F_{[t_2,t_3]}(x) \coloneqq \int_{[t_2,t_3]\times\{x\}\times\Tbb^2} (2T_\chi^{11} + \Ebf_0(\Pbb) + \Ebf_0(\Qbb)) \Omega \, \dVundeux , \qquad x\in\Sbb^1 ,
\ee
where $\dVundeux=|t|\,\Omega \,dt\,dy\,dz$ is the volume form induced by~$g$.
There exists a constant $C>0$ depending only on $t_0,t_1$ (but not $t_2,t_3$), the initial energy~$\mathcal E(t_0)$, volume~$\mathcal V(t_0)$, initial bounded-variation norms of $\ell,\log\Omega,\Kpar$, and minimal conformal length~$\mathcal L_{\min}$ such that
\be
\|\mathcal F_{[t_2,t_3]}\|_{\BV(\Sbb^1)}
= \Var(\mathcal F_{[t_2,t_3]};\Sbb^1) + \| \mathcal F_{[t_2,t_3]} \|_{L^\infty(\Sbb^1)} \leq C .
\ee
\end{proposition}

\subsubsection{Space-time integral of the fluid equation}

We integrate $\opJ^1_\chi=0$ explicited in~\eqref{using-Mbfchi-expr} on a spacetime slab $[t_2,t_3]\times [x_1,x_2]\times\Tbb^2$ with the volume form $\dVuntrois=|t|\Omega^2\,dt\,d_x\ell\,dy\,dz$.  This gives that the difference of boundary terms at $x_2,x_3$ is bounded by space-like integrals that are bounded by the energies~$\mathcal E$ at $t_2,t_3$ thanks to $|\Mbf_\chi^{01}|\leq\Mbf_\chi^{00}$:
\be
\bigl| \mathcal F_{[t_2,t_3]}^\Kpar(x_2) - \mathcal F_{[t_2,t_3]}^\Kpar(x_3) \bigr|
\leq \mathcal E(t_2) + \mathcal E(t_3) ,
\ee
where
\be
\aligned
\mathcal F_{[t_2,t_3]}^\Kpar(x)
\, & \! \coloneqq \int_{[t_2,t_3]\times\{x\}\times\Tbb^2} \Mbf_\chi^{11} \Omega \, \dVundeux \\
& = \int_{[t_2,t_3]\times\{x\}\times\Tbb^2} \bigl(2T_\chi^{11} + \Ebf_0(\Pbb) + \Ebf_0(\Qbb) - \Ebf_0(\Kpar)\bigr) \Omega \, \dVundeux \\
& = \mathcal F_{[t_2,t_3]}(x) - \int_{t_2}^{t_3} \bigl(|t|\Omega^2\Ebf_0(\Kpar)\bigr)(t,x) \, dt .
\endaligned
\ee
Thanks to our control of the bounded variation norms of $\Omega,\Kpar$, the last term has bounded variation as a function of $x\in\Sbb^1$.
We thus arrive at a bound on the total variation of~$\mathcal F_{[t_2,t_3]}$, which must be completed by an integral bound.

\subsubsection{Monotonicity of the conformal length density}

The spatial integral of $\mathcal F_{[t_2,t_3]}(x)$ is now compared to the time integral of the spatial energy~$\mathcal E(t)$, as both amount to integrals of stress-tensor components (augmented by geometric contributions).  The volume forms involved differ.  Specifically, the volume form $\dVuntrois=|t|\Omega^2\,dt\,d_x\ell\,dy\,dz$ and the volume form $\dVundeux=|t|\,\Omega \,dt\,dy\,dz$ of the induced metric on a constant~$x$ slice are related by
\be
\dVuntrois = \Omega \, \amdeux\,dx\, \dVundeux .
\ee
We are led to needing a lower bound $\amdeux_{\rm ref}$ on~$\amdeux$, which we obtain by monotonicity in time.

From the evolution equation $\opC^\chi=0$ of the conformal length density~$\amdeux$, and from the dominant energy condition $|\Mbf_\chi^{11}|\leq\Mbf_\chi^{00}$ in \autoref{lem:conseq-DEC-M}, we deduce that $\amdeux$ has the same monotonicity as its spatial integral~$\mathcal L$.
\bei
\item In the future-contracting regime, $\del_t\amdeux\leq 0$, so we set $\amdeux_{\rm ref}=\amdeux(t_1)$.
\item In the future-expanding regime, $\del_t\amdeux\geq 0$, so we set $\amdeux_{\rm ref}=\amdeux(t_0)$.
\eei
In both regimes, by monotonicity and by definition of $\mathcal L_{\min}$, we have
\bel{intamdeuxref}
\int_{\Sbb^1} \amdeux_{\rm ref}(x)\,dx = \mathcal L_{\min} , \qquad
\amdeux(t,x) \geq \amdeux_{\rm ref}(x) , \qquad t \in [t_0,t_1] .
\ee

By the dominant energy condition $|\Mbf_\chi^{11}|\leq\Mbf_\chi^{00}$ (\autoref{lem:conseq-DEC-M}), we obtain 
\bel{eq:avg-time-est}
\int_{\Sbb^1}\bigl|\mathcal F_{[t_2,t_3]}^{\Kpar}\bigr|\,\amdeux_{\rm ref}\,dx
\leq \int_{[t_0,t_1]\times\Sbb^1}|t|\Omega^2 |\Mbf_\chi^{11}|\,\amdeux\,dt\,dx
\leq\int_{t_0}^{t_1}\mathcal E_\chi(t)\,dt.
\ee
Dividing by the total mass $\mathcal L_{\min}$ of $\amdeux_{\rm ref}\,dx$, which is positive, gives a control of an average of $\mathcal F_{[t_2,t_3]}^{\Kpar}$.  The value at any point $x\in\Sbb^1$ then differs from this average by the total variation bound above.

\subsection{Consequences of a coercive stress}
\label{section-8-timebv}

\subsubsection{Coercive stress condition}

We have so far proven the spatial energy bound~\eqref{estimate-SE}, two sets of conformal length and volume bounds~\eqref{estimate-LV}--\eqref{both-sided-LV}, bounded variation in space~\eqref{spatial-bv}, and a control of temporal integrals~\eqref{temporal-stress}.
As stated in \autoref{theo:geometric-estimates}, it is time to make three further assumptions:
\bei
\item the potential~$\chi$ is torus adapted (\autoref{def-NS-torus-adapted}),
\item the parallel momentum per particle is uniformly bounded, namely $2|T_\chi^{02}|+2|T_\chi^{03}|\leq C^\parallel\Numb_\chi^0$ at initial time, for some constant~$C^\parallel>0$,
\item the coercive stress condition holds (\autoref{def-Tchi-coercive}).
\eei
As explained in \autoref{section-7}, the first condition leads to transport equations for the parallel momentum per particle defined (away from vacuum) as $\widehat J_\chi^m = 2 T_\chi^{0m}/\Numb_\chi^0$, which lead to a maximum principle when $\widehat J_\chi^m$ is uniformly bounded.  This second condition is stated in a way that makes sense at vacuum, and that is subject to the maximum principle of \autoref{them:7-4}.
As a result, this bound holds for all times:
\be
2|T_\chi^{02}|+2|T_\chi^{03}|\leq C^\parallel\Numb_\chi^0 , \qquad t\in[t_0,t_1] .
\ee
The coercive stress condition defined in \autoref{def-Tchi-coercive} then implies that, for some constants $C_1,c_2$ that only depend on the Navier--Stokes potential~$\chi$, one has the pointwise bound
\be
T_\chi^{11} + C_1 \geq c_2 \bigl(T_\chi^{00} + (J_0)^2\bigr) .
\ee

Integrating this inequality along a timelike slice $[t_0,t_1]\times\Tbb^2$ at fixed $x\in\Sbb^1$ with respect to the volume form $\dVundeux=|t|\,\Omega \,dt\,dy\,dz$ gives
\be
c_2 \int_{[t_0,t_1]\times\Tbb^2} \bigl(T_\chi^{00} + (J_0)^2\bigr) \dVundeux
\leq \biggl| \int_{[t_0,t_1]\times\Tbb^2} T_\chi^{11} \dVundeux \biggr|
+ C_1 \int_{[t_0,t_1]\times\Tbb^2} \dVundeux .
\ee
These terms are controlled respectively by the timelike integrals~\eqref{temporal-stress} and by pointwise bounds on $\Omega$ and boundedness of the interval $[t_0,t_1]$.
Thanks to the term $J_0^2$, this completes the set of finite-energy norms that appear in~$\|\Phi\|_{\rm energy}$ defined in~\eqref{equa-flow-energy-norm}.
The term $T_\chi^{00}$ is useful next.

\subsubsection{Bounded variation in time}

The bounded-variation norm $\|\Psi\|$ defined in~\eqref{equa-flow-variation-norm} consists of the spatial bounded variation norms controlled in \autoref{section-8-bv} and of timelike variation.  We now derive these in the same way as the spatial ones, by considering the evolution equations $\opC^\chi=\opO_0^\chi=\opK_{20}^\chi=\opK_{30}^\chi=0$ for $\amdeux,\log\Omega,\Kpar$.
\bei
\item Since $\amdeux$ is montonic in time, $\ell(t,x) = \int_0^x\amdeux(t,x')dx'$ also is, and its total variation in time is simply a difference of its end-point values, which are bounded by the $L^\infty$~norm of~$\ell$, bounded earlier.

\item For the lapse and twists, we find (using the dominant energy conditions $|T_\chi^{1m}|\leq T_\chi^{00}$ for $m=1,2,3$)
\be
\aligned
\int_{t_0}^{t_1} \bigl| (\log\Omega)_t \bigr| dt
& \leq \frac{1}{2} \int_{t_0}^{t_1} \biggl( 2 T_\chi^{00} + \Ebf_0(\Kpar,\Pbb,\Qbb) + \frac{1}{2t^2\Omega^2} \biggr) |t|\Omega^2\,dt ,
\\
\int_{t_0}^{t_1} \bigl| (K_2)_t \bigr| dt
& \leq \int_{t_0}^{t_1} \biggl( 2 T_\chi^{00} + \frac{1}{2} |P_0| |K_2| + \frac{3}{2|t|\Omega} |K_2| \biggr) \Omega\,dt ,
\\ 
\int_{t_0}^{t_1} \bigl| (K_3)_t \bigr| dt
& \leq \int_{t_0}^{t_1} \biggl( 2 T_\chi^{00} + \frac{1}{2} |P_0| |K_3| + |Q_0| |K_2| + \frac{3}{2|t|\Omega} |K_3| \biggr) \Omega\,dt,
\endaligned
\ee
which are controlled by the available timelike estimates on the fluid and geometry, together with pointwise bounds on~$\Omega$.
\eei
This concludes the proof of \autoref{theo:geometric-estimates}.


\section{Conclusions and perspectives}
\label{section-9} 

\subsection{In summary}
\label{section-10-1}

\subsubsection{Areal stability}

The \JKL{} formulation combines the first-order metric variables, twists, parallel matter momentum, and total stress tensor in a single evolution--constraint system. The Einstein identities derived in this paper yield monotonicity formulas for the volume and conformal length, together with spacelike and timelike integral estimates. At weak regularity, these estimates use only the basic balance laws and one selected mathematical entropy inequality. The maximum principles concern regular flows, and their closed support constraints pass to Young-measure limits of regular approximations. Together, these estimates propagate the finite-energy bounds on every compact interval of areal time on which the orbit area does not vanish and the conformal length does not collapse. The latter condition follows from the initial conformal length in the future-expanding regime and is an additional hypothesis only in the future-contracting regime. These conclusions concern the geometry of the matter spacetimes considered here and require no perturbative smallness assumption on the geometry or the matter fields.

\subsubsection{Vacuum and unbounded states}

The proposed formulation and the derived estimates allow vacuum, arbitrarily large mass-energy density, and arbitrary finite rapidity. Hyperbolicity of the classical evolution systems holds in the positive-density, future-timelike interior of the state space, whereas the constitutive tensors, the mathematical entropy inequalities, and the finite-energy estimates extend continuously to its vacuum boundary in terms of the variable~$\Jbb$. The generalized Young-measure formulation retains this full state range; together with the associated concentration measures, it also describes concentration phenomena arising from sequences with unbounded mass-energy density.

\subsubsection{Two selected mathematical entropy inequalities}

The dissipated-energy system preserves particle number and yields a non-negative loss of normal mass-energy; its orbital transfer tensor restores conservation of the total stress tensor appearing in the Einstein equations. The particle-production system preserves energy--momentum and yields non-negative particle production. The same scalar potential and the same barrier relaxation map define both systems, although the corresponding source is inserted into different balance laws. Consequently, their causal principal parts and their density-global constitutive ranges coincide. They nevertheless select two different mathematical entropy inequalities: any finite-energy zero-relaxation limit satisfies the inequality associated with the corresponding system.


\subsubsection{Equilibrium case}

For either system, vanishing production implies $\zeta=0$, hence $T_\chi=T_{\rm E}$ and $\vNumb_\chi=\Numb u$ by the normalization of the potential on the Euler equilibrium manifold. In the dissipated-energy system, $T_{\rm aux}$ then vanishes as well. Thus, the equality cases reduce to the Einstein--Euler system. Along a sequence for which the production does not vanish in the limit, the limiting production measure may remain non-zero. Independently, oscillation or concentration in the finite-energy variables may produce a generalized Young measure (accounting also for concentration) occurring in the resulting finite-energy solution. These measures must therefore be retained separately in the limiting formulation.


\subsection{Mathematical entropy production}
\label{section-9-1}

\subsubsection{Toward a zero-viscosity limit}

Remarkably, all of the Cauchy stability estimates involving \emph{only} the finite energy norms on an initial hypersurface; moreover, they are \emph{independent} of the relaxation rate field $\mathcal B$. Consequently, this open up a path toward tackling the zero-viscosity, or zero-relaxation, problem, consisting of proving the convergence of Einstein-Navier--Stokes flows toward Einstein-Euler flows. This is a challenging problem that will be fully solved only in the fourth article in this series~\cite{TorusT2-4-2026}. We must show that the relaxation variable $\zeta$ approaches zero, and, in more geometric terms, and we must deal with families of solutions approaching the Euler equilibrium manifold. 
 
Three important steps to prove that zero-viscosity limits of Einstein--Navier--Stokes flows are finite-energy Einstein--Euler flows are as follows. One must identify the mathematical entropy production selected by each system, which may be applied to quantify the convergence of the relaxing field toward equilibrium. One must formulate the finite-energy estimates for weakly regular Einstein-Euler areal flows. Finally, one must construct a generalized Young measure, which allows to write all of the limiting \JKL{} equations achieved when $\zeta$ approaches zero; this object includes the effect of oscillations and  concentration phenomena, further studied in~\cite{TorusT2-3-2026} (perfect-fluid matter) and~\cite{TorusT2-4-2026} (dissipative matter). 


\subsubsection{Family of Einstein--Navier--Stokes flow}  

The relaxation rate map introduced in \autoref{def-pot-source} remains denoted by~$\mathscr B$, but we introduce a rescaling in the relaxation equations by introducing a viscosity parameter $\eps \in (0,1]$, namely we now write
\bel{eq:section-9-relaxation-source}
\nabla_\alpha A_\chi^{\eps,\alpha\beta\gamma}(\beta^\eps,\zeta^\eps)
= - \frac{1}{\eps}\mathscr B^{\beta\gamma}(\beta^\eps,\zeta^\eps), 
\ee
and we also denote the family of solutions by $(\beta^\eps,\zeta^\eps)$. Throughout this section, we assume that the initial data set has \emph{uniformly bounded} finite energy bounds with respect to~$\eps$. By the main theorem of the present paper, this implies the \emph{uniform} bounds for the family of flow at any future areal time. 
It then follows from the quadratic and the bounded-variation estimates
that a countable sequence of $\eps \to 0$ can be extracted so that the Einstein--Navier--Stokes flow  converges in a suitably weak sense to a ``limiting flow''. However, the notion of convergence at this stage is  \emph{far from sufficient} to imply 
that this formal limit is indeed an Einstein--Euler flow. For instance, no equi-integrability at the quadratic scale is implied by our estimates and oscillation and concentration may occur in both the geometric and the matter variables. This major issue is left for \cite{TorusT2-3-2026,TorusT2-4-2026}, but at this stage we can still define this limiting flow.


\subsubsection{Mathematical entropy productions}

The two Einstein--Navier--Stokes systems are defined from a choice of Navier--Stokes potential~$\chi$ and a relaxation rate map~$\mathscr B$. Their principal differential operators coincide, but the non-negative production occurs in different balance laws. For either system, we set
\bel{eq:section-9-production-density}
\mathfrak d_{\mathscr B}^{\eps}
\coloneqq {1 \over \eps} \zeta_{\alpha\beta}^{\eps}
\mathscr B^{\alpha\beta}(\beta^\eps,\zeta^\eps)
\geq0.
\ee
Before we proceed, let us recall that, for the dissipated-energy system, \autoref{prop-diss-flow} and~\eqref{energy-dissip-Sigma} give
\bel{eq:en-op-visc}
\opJ_0^{\chi,\eps}
=-\Sigma^\eps\leq0,
\qquad
\divuntrois\vNumb_\chi^\eps=0,
\qquad
\Sigma^\eps
=\frac{2\Omega^\eps}{\nu^\eps}\mathfrak d_{\mathscr B}^\eps,
\qquad
\nu^\eps\coloneqq-n_\alpha\beta^{\eps,\alpha}>0.
\ee
The auxiliary tensor~$T_{\rm aux}^\eps$ has the opposite divergence. Consequently, the total stress tensor
\be
T_{\rm NS}^\eps
=T_\chi^\eps+T_{\rm aux}^\eps
\ee
is divergence-free, while the potential-defined mass-energy current satisfies a mathematical entropy inequality and the particle-number current remains conserved.

On the other hand, for the particle-production system, the four energy--momentum equations remain equalities, and~\eqref{eq:pot-areal-production} gives
\bel{eq:section-9-N-production}
\opJ_0^{\chi,\eps}=0,
\qquad
\divuntrois\vNumb_\chi^\eps
= \mathfrak d_{\mathscr B}^\eps\geq0.
\ee
Thus, the dissipated-energy system preserves particle number and dissipates normal mass-energy, whereas the particle-production system preserves energy--momentum and produces particle number. In either case, the mathematical entropy inequality follows from the constitutive equations; it is not an additional condition on regular flows.

\subsubsection{Vanishing of the relaxing field}

Let $\mathcal M=[t_0,t_1]\times\Tbb^3$ be a compact spacetime slab. The areal factors are bounded above and below on~$I$, and the lapse bounds in \autoref{section-8} are independent of~$\eps$. The dissipated-energy identity and the particle-number identity imply, respectively,
\bel{eq:section-9-production-bounds}
\sup_{0<\eps\leq1}\int_{\mathcal M}\frac{\mathfrak d_{\mathscr B}^{\eps}}{\nu^\eps}\,\dVuntrois_{\eps}
\leq C \, \mathcal I_{\rm M,0},
\qquad
\sup_{0<\eps\leq1}\int_{\mathcal M}\mathfrak d_{\mathscr B}^{\eps}\,\dVuntrois_{\eps}
\leq C \, \mathcal I_{\rm N,0},
\ee
where $\mathcal I_{\rm M,0},\mathcal I_{\rm N,0}$ are uniformly bounded by initial data and $\mathcal L_{\min}$.
The first integral is the normal mass-energy loss and also controls the coefficient of~$T_{\rm aux}^\eps$; the second is the particle-number production. The constants depend on the compact slab and on the corresponding initial finite-energy sizes, but not on~$\eps$ or~$\mathscr B^\eps$.
 
Thanks to the coercivity $\zeta^\eps_{\alpha\beta}(\mathscr B^\eps)^{\alpha\beta}\gtrsim \frac{1}{\eps} |\zeta^\eps|^2$, one deduces $\zeta^\eps\to0$ strongly in~$L^2(\mathcal M)$ for the particle-production model and strongly in~$L^2(\mathcal M,(\nu^\eps)^{-1}\dVuntrois_\eps)$ for the dissipated-energy model; the latter being a weighted estimate. In the generalized Young measure argument in~\cite{TorusT2-3-2026,TorusT2-4-2026}, every finite-state limit lies on the Euler equilibrium manifold, while a contribution at the quadratic concentration scale may remain.


\subsection{Limiting Euler flows}
\label{section-9-2}

\subsubsection{Formal limits}

Recall that~$\Phi$ denotes the variables controlled by the quadratic energy, while~$\Psi$ denotes the variables controlled by the absolute-continuity and bounded-variation estimates in~\eqref{equa-flow-energy-norm}--\eqref{equa-flow-variation-norm}. Informally, we may say that a \textbf{limiting} finite-energy Einstein--Euler areal flow is a formal limit of either a sequence of dissipated-energy areal flows or of particle-production areal flows with uniform initial finite-energy bounds: 
\be
\|\Phi(t_0)\|_{\mathrm{energy}}
+\|\Psi(t_0)\|_{\mathrm{BV}}<+\infty.
\ee
A family of Einstein--Navier--Stokes flows with uniform initial finite-energy bounds need not, however, converge to a single Einstein--Euler flow, since \emph{oscillation} and \emph{concentration} may persist at the quadratic scale. The suitable object is the generalized Young-measure which we will define in~\cite{TorusT2-3-2026,TorusT2-4-2026}.



\subsubsection{The two selected mathematical entropy inequalities}

In either weak formulation, the three momentum equations remain equalities:
\be
\opJ_1^{\chi,\eps}=\opJ_2^{\chi,\eps}=\opJ_3^{\chi,\eps}=0.
\ee
The remaining matter equation is coupled to one of the two mathematical entropy inequalities. A limiting flow satisfies the \emph{particle-production formulation} 
\bel{eq:part-num-form}
\opJ_0^{\chi,\eps}=0,
\qquad
\divuntrois\vNumb
=\mathfrak d_{\rm N}\geq0,
\ee
or  the \emph{dissipated-energy formulation} if
\bel{eq:mass-en-form}
\divuntrois\vNumb=0,
\qquad
\opJ_0^{\chi,\eps}=-\mathfrak d_{\rm M}\leq0.
\ee
The applicable production~$\mathfrak d_{\rm M}$ or~$\mathfrak d_{\rm N}$ is a non-negative Radon measure.
 For a regular Einstein--Euler flow, both productions vanish and the formulations coincide. A regular Einstein--Navier--Stokes flow may have non-zero constitutive production.

For weakly regular flows containing shock waves, the two formulations generally differ. Suppose, for instance, that a discontinuity is supported on a curve $x=\gamma(t)$ with speed $s=\gamma'(t)$, and adopt the convention that $[\,\cdot\,]$ denotes the right state minus the left state. A balance law $\divuntrois X=S$ with locally integrable source gives
\be
[X^1-s\amdeux X^0]=0,
\ee
whereas the inequality $\divuntrois X\geq0$ gives
\be
[X^1-s\amdeux X^0]\geq0.
\ee
The regular source terms in the areal equations have no singular part on the shock curve. Consequently, exchanging the equality and the inequality changes the Rankine--Hugoniot locus.


\subsection{Outlook}
\label{section-10-3}

\subsubsection{Oscillations}

Our stability estimates will also be crucial to define an associated Young measure, which will allow us to derive to the first-order formulation of Einstein's balance laws in a measure-valued form. Without further compactness information, the resulting Young measure need not be concentrated at a single point of the state space. The parametrized families of fluid quasi-currents and the associated mathematical entropy kernels are developed in~\cite{TorusT2-3-2026}. They impose additional relations on the support of the Young measure and provide criteria under which the nonlinear Euler stress is represented by a single finite-energy fluid state rather than by a non-Dirac probability measure. The genuine nonlinearity condition of the pressure law arises in this context in order to apply Tartar's compensated compactness arguments, but it is not required for the stability estimates and the Young measure representation established in the present paper.


\subsubsection{Concentration}

The generalized Young measure may also retain the effect of quadratic oscillation and concentration in the first-order geometric variables. Their contributions to the Einstein equations, leading to the notion of \emph{stress-energy correctors} in the sense introduced in~\cite{LeFlochLeFloch-1,LeFlochLeFloch-port}, are analyzed in the companion papers~\cite{TorusT2-1-2026,TorusT2-4-2026}. Therein, we provide conditions under which the averaged nonlinear terms can be identified strongly and the limiting flow reduces to an (uncorrected) Einstein--Euler areal flow.


\subsubsection{Other symmetry classes and foliations} 

The estimates proved here rely on the area of the torus orbits as a global time function and on the balance laws exhibited by the \JKL{} formulation. The present stability results identify the quadratic energy, bounded-variation, and mathematical entropy estimates. A natural problem is to determine which parts of this structure persist in other foliations by spacelike hypersurfaces or under other symmetry assumptions.  


\phantomsection
\subsection*{Acknowledgments}

We acknowledge financial support from the Simons Center for Geometry and Physics, Stony Brook University, where part of this research was carried out during a visit in 2019. Part of this work was also completed while the second author (PLF) was visiting the Mittag-Leffler Institute during the Semester Program ``General Relativity, Geometry, and Analysis: beyond the first 100 years after Einstein''. During the completion of this project, the authors were  partially supported by the research project ANR-23-CE40-0010-02: Einstein-PPF, funded by the Agence Nationale de la Recherche, and by the ERC-MSCA Staff Exchange Project 101131233: Einstein-Waves, funded by the European Research Council. 
 

\phantomsection
\pdfbookmark[0]{References and appendices}{Appendix-target}


\appendix

\section{Derivation of the \JKL{} formulation}

\label{section-A}

\subsection{Einstein evolution and constraints}
\label{section-A-1} 

\subsubsection{Einstein evolution equations}

We explain here which Ricci components produce the evolution equations, which produce the constraints, and how the Euler equations follow from the contracted Bianchi identity.  We thus proceed to compute suitably chosen linear combinations of the $10$ frame components~\eqref{eq:Einsframe} in terms of the geometry variables~\eqref{eq:definevar}. Together with two compatibility equations, we obtain a total of $12$ equations, as follows. 
\bei

\item \emph{Evolution of the first-order geometry.} 
Using our notation, after elementary calculations, the evolution equations for $\Lbb= \Pbb+i\,\Qbb$, with $\Pbb = (P_0, P_1)$ and $\Qbb = (Q_0, Q_1)$, are given by~\eqref{eq:T2-1234-def}, as stated previously, where the operators were defined in~\eqref{equa-div-curl}. The two divergence equations are the components $R_{22} - R_{33}$ and $2\, R_{23}$ of the Einstein equations, while the curl equations are compatibility conditions satisfied by the derivatives of $P$ and~$Q$ and are needed to reach a closed evolution system for the first-order unknowns.

\item \emph{Evolution of the conformal length density.}  
Notably, in our formulation, only the equations governing the quotient geometry depend explicitly on the equation of state through the definition~\eqref{eq:T2-Mdef-0} of the tensor $\Mbf$. The component $R_{22}+R_{33}$ of the Einstein equations is found to be as stated\footnote{Since $\del_t = \Omega e_0$ and $\del_x= \Omega\amdeux e_1 + G\del_y+H\del_z$, the differential equations in \eqref{eq:T2-all-suite-def}--\eqref{eq:theconstraints-def} could equivalently be expressed in terms of the frame derivatives $e_0$ and $e_1$. We prefer not to adopt this notation here.} in \eqref{eq:T2-8-def}. The function $\amdeux$ determines the conformal geometry on the quotient manifold $\Ncal = \Mcal/\Tbb^2$, and will play a central role when analyzing the global foliation. As proven in \autoref{section-8}, equation~\eqref{eq:T2-8-def} controls the volume of spacelike slices up to the multiplicative factor $|t|\Omega$.

\item \emph{Evolution of the lapse.} 
The metric coefficient $\Omega$ is the lapse function in the metric decomposition~\eqref{metric:areal}, and its time derivative is given by the Ricci component $R_{00}+R_{11}-R_{22}-R_{33}$ of the Einstein equations, as stated in \eqref{eq:evollambda-def}. 

\item \emph{Evolution of the twists.} 
The components $R_{12}$ and $R_{13}$ provide the time derivatives of the twist coefficients $K_2, K_3$, namely the differential equations \eqref{eq:T2-9101112-evol-def-a}--\eqref{eq:T2-9101112-evol-def-b}. 

\item \emph{Additional equation for the geometry.} 
Finally, the component $R_{00}-R_{11}+R_{22}+R_{33}$ provides another geometric equation (for the lapse), but its statement is postponed to~\eqref{eq:waveconffac}, below, as it is \emph{not} included in our first-order formulation.
\eei 


\subsubsection{Einstein constraint equations}

We derive separately the remaining equations for the geometry, which are automatically satisfied once they hold on any given hypersurface of constant areal time.
\bei 

\item \emph{Constraint on the lapse.} While time derivatives of $\amdeux$ and $\Omega$ obey the equations~\eqref{eq:T2-8-def} and~\eqref{eq:evollambda-def}, respectively, only the space derivative of $\Omega$ obeys an equation. The component $R_{01}$ of the Einstein equation relates this derivative to an energy flux term, as stated in \eqref{eq:T2-567-def}. 

\item \emph{Constraints on the twists.} The components $R_{02}$ and $R_{03}$ of the Einstein equations give the space derivatives of the twists as stated in \eqref{eq:221a}--\eqref{eq:221b}. 
\eei 


\subsection{Lapse and matter identities}
\label{section-A-2}

\subsubsection{Euler equations}

At this stage, we have (by including~\eqref{eq:waveconffac} stated below) twelve equations corresponding to the ten components of the Einstein equations, supplemented with the two compatibility equations for $\curl_\Ncal \Pbb$ and $\curl_\Ncal \Qbb$ appearing in~\eqref{eq:T2-1234-def}. In the regular non-vacuum regime, the contracted Bianchi identity yields the Euler equations from this Einstein system. However, in our first-order presentation, it is important to express these Euler equations explicitly, and to further rearrange the non-linearities. We distinguish between the transverse and parallel components. The propagation equations for the transverse momentum \eqref{eq:T2-Euler-perp-def-a}--\eqref{eq:T2-Euler-perp-def-b} are \emph{coupled} to the remaining Euler equations \eqref{eq:T2-fluid23-def-a}--\eqref{eq:T2-fluid23-def-b}.
\bei

\item \emph{Evolution of the transverse momentum.} We find \eqref{eq:T2-Euler-perp-def-a}--\eqref{eq:T2-Euler-perp-def-b} stated earlier. Here, $\Mbf^{m\bullet} = (\Mbf^{m0},\Mbf^{m1})$ (for $m=0,1$) is the vector with components $(\Mbf^{m0},\Mbf^{m1})$ defined in~\eqref{eq:T2-Mdef-0}, $\divuntrois$ is the divergence of this vector in the sense~\eqref{equa-def-div13}, and the source term is given in \eqref{eq:T2-Euler-perp2-def0} and obeys the inequality $\abs{\Msource}\leq 5 \, \Mbf^{00}$ (proven in~\autoref{section-A-5}, below). Interestingly, the first fluid equation coincides with the wave equation~\eqref{eq:waveconffac} below,
whereas the second equation is nothing but the compatibility equation $\bigl(\log \Omega \bigr)_{xt} - \bigl(\log \Omega \bigr)_{tx} = 0$, in which we replaced the first-order derivatives of the lapse by energy and energy fluxes (namely, using~\eqref{eq:evollambda-def} and~\eqref{eq:T2-567-def}). 

\item \emph{Evolution of the parallel momentum.} The remaining two fluid equations are deduced from the compatibility equation $\bigl( |t|^{3/2} K_2 \bigr)_{xt} - \bigl( |t|^{3/2} K_2 \bigr)_{tx} = 0$ and the analogue for $K_3$. Namely, from~\eqref{eq:221a}--\eqref{eq:221b} and~\eqref{eq:T2-9101112-evol-def-a}--\eqref{eq:T2-9101112-evol-def-b} we obtain  \eqref{eq:T2-fluid23-def-a}--\eqref{eq:T2-fluid23-def-b}.  
Since $P_0, P_1, K_2, K_3$ and the lapse and conformal length density already have their own evolution equations, it is natural to view~\eqref{eq:T2-fluid23-def-a}--\eqref{eq:T2-fluid23-def-b} as evolution equations for the parallel momentum $(J_2,J_3)$ rather than for the orthogonal momentum~$\Jperp$. 
\eei


\subsubsection{The remaining Einstein equation: wave equation for the lapse}

The identity displayed below is obtained from the Einstein equations by taking the Ricci combination $R_{00}-R_{11}+R_{22}+R_{33}$.  As we now explain, it can be derived from the other Einstein equations and the Euler equations.  Thus, it should not be regarded here as an additional equation to be imposed independently of the first-order system.
The identity reads\footnote{We state~\eqref{eq:waveconffac} in a form that remains meaningful in the sense of distributions when $(\log\Omega)_t$, $\amdeux_t/\amdeux$, and $\Omega_x/\amdeux$ are merely integrable functions or measures, as occurs for finite-energy solutions.}
\bel{eq:waveconffac}
\divuntrois \Bigl( t^{-1}  \Omega^{-1} (\log(\Omega \amdeux) )_t, \ t^{-1}  \Omega^{-1} \amdeux^{-1} (\log\Omega)_x \Bigr)
+ \frac{1}{4t^3} \Omega^{-2} - \frac{1}{4t} \, \Mwave(\Jbb, \Kpar, \Lbb) = 0 ,
\ee
with 
\bel{eq:waveconffac-2}
\aligned
\Mwave(\Jbb, \Kpar, \Lbb)
& \coloneqq 6 \, \Ebf_0(\Kpar) + 2 \, \Ebf_0(\Jpar) - \Pbb \cdot \Pbb - \Qbb \cdot \Qbb - \Jperp \cdot \Jperp. 
\endaligned
\ee 
Here, the divergence term is the divergence of the one-form
\be
t^{-1} \bigl(d\log\Omega + e_0(\log \amdeux)e^0\bigr), 
\ee
whose components in the Vielbein basis $e_0,e_1$ are $t^{-1}  \Omega^{-2} \amdeux^{-1} \bigl( (\Omega \amdeux )_t, \Omega_x \bigr)$.

To verify both the sign and the geometric factors, let $\mathcal W$ denote the left-hand side of~\eqref{eq:waveconffac}. We expand the divergence with~\eqref{equa-divN}--\eqref{equa-def-div13} and substitute the two first-order lapse equations~\eqref{eq:evollambda-def} and~\eqref{eq:T2-567-def}. The remaining terms regroup as
\be
\mathcal W=-\frac{1}{2}\left(\divuntrois\bigl(\Omega\Mbf^{0\bullet}\bigr)+\frac{1}{2t}\Msource\right)=-\frac{1}{2}\opJ^0_{\chi,\eps}.
\ee
Thus~\eqref{eq:waveconffac} is equivalent (with a $-1/2$ factor) to the first Euler equation~\eqref{eq:T2-Euler-perp-def-a}. We separate it from the main first-order system~\eqref{eq:T2-1234-def}--\eqref{eq:T2-all-suite-def} because it is not an independent equation. In vacuum it follows from the other Einstein equations; with matter, its weak equality or inequality is determined by the selected mass-energy mathematical entropy formulation.


\subsection{Principal symbol and characteristic fields}
\label{section-A-3}

\subsubsection{Wave speeds}

This appendix supplies the characteristic and constraint-propagation calculations relative to  Propositions~\ref{theo-struct} and~\ref{theo-struct-bis}.  

The expressions of the wave speeds are obvious for the geometry variables, while concerning the fluid variables for sufficiently regular solutions we can write (modulo geometric and lower-order terms denoted by $\geo$ and $\lot$, respectively) 
\bel{equa-euleronly}
\aligned
\Bigl( \Ebf_0(\Jperp) + \Ebf_0(\Jpar) - \frac{q_\Jbb}{2}  \, \Jbb \cdot \Jbb \Bigr)_t
- \amdeux^{-1} \, \Bigl( \Ebf_1(\Jperp) \Bigr)_x
& = \geo + \lot,
\\
\Bigl( \Ebf_1(\Jperp) \Bigr)_t
- \amdeux^{-1} \, \Bigl( \Ebf_0(\Jperp) - \Ebf_0(\Jpar) + \frac{q_\Jbb}{2}  \, \Jbb \cdot \Jbb \Bigr)_x
& = \geo + \lot,
\endaligned
\ee
coupled with~\eqref{eq:T2-fluid23-def-a}--\eqref{eq:T2-fluid23-def-b} which can be expanded similarly. Here, the constitutive function $q$ only depends on the squared norm 
\be
y\coloneqq - \Jbb \cdot \Jbb = J_0^2 - J_1^2 - J_2^2 - J_3^2 = 2 (\mu + p(\mu)).
\ee
Recall that $y \geq 0$ and $J_0 \leq 0$, while $\Ebf_0(\Jperp) = (J_0^2 + J_1^2)/2$ and $\Ebf_0(\Jpar) = (J_2^2 + J_3^2)/2$.

Focusing on the principal part only, we obtain a system of four coupled partial differential equations:  
\bel{equa-Euler-Minko}
\begin{aligned} 
\frac{1}{2} \bigl( J_0^2 + J_1^2 + J_2^2 + J_3^2 + y q \bigr)_t
- \amdeux^{-1} ( J_0 J_1 )_x
& = 0, \quad
& ( J_0 J_2 )_t - \amdeux^{-1} ( J_1 J_2 )_x
& = 0,
\\
( J_0 J_1 )_t - \amdeux^{-1} \, \frac{1}{2} \bigl( J_0^2 + J_1^2 - J_2^2 - J_3^2 - y q \bigr)_x
& = 0,
& ( J_0 J_3 )_t - \amdeux^{-1} ( J_1 J_3 )_x
& = 0.
\end{aligned}
\ee 
We can write this system in the quasilinear form $A_0 \Jbb_t - \amdeux^{-1} A_1 \Jbb_x =0$ (for sufficiently regular solutions) with $A_0$ and $A_1$ given below in terms of
\be
(y q)' \coloneqq \frac{d(y q)}{dy}
= \frac{1-p'(\mu)}{1+p'(\mu)}
= \frac{1-k(\mu)^2}{1+k(\mu)^2}
\ee
in terms of the sound speed $k(\mu) = \sqrt{p'(\mu)}$.
The hyperbolicity conditions~\eqref{hyperbolic-eos} imply that $0<k(\mu)<1$ away from vacuum, namely $0<(y q)'<1$. Explicitly, the operators are
\be
\aligned
A_0 & \coloneqq
\begin{pmatrix}
(1+ (yq)') J_0 
& (1- (yq)') J_1
& (1- (yq)') J_2
& (1- (yq)') J_3
\\
J_1& J_0 & 0 & 0
\\
J_2 & 0 & J_0 & 0 
\\
J_3 & 0 & 0 & J_0 
\end{pmatrix}, 
\\
A_1 & \coloneqq 
\begin{pmatrix}
J_1& J_0 & 0 & 0
\\
(1- (yq)') J_0 & (1+(yq)') J_1 & (-1+(yq)') J_2 & (-1+(yq)') J_3
\\
0 & J_2 & J_1 & 0 
\\
0 & J_3 & 0 & J_1 
\end{pmatrix}.
\endaligned
\ee
The wave speeds are easily computed as the (real) roots $\xi$ of the polynomial $\det(\xi A_0 + \amdeux^{-1} A_1)$, namely the determinant of the matrix
\be
{\setlength{\arraycolsep}{3pt}\begin{pmatrix}
\xi (1+(yq)') J_0 + \amdeux^{-1} J_1 
& \xi (1-(yq)') J_1 + \amdeux^{-1} J_0 
& \xi (1-(yq)') J_2
& \xi (1-(yq)') J_3
\\
\xi J_1 + \amdeux^{-1} (1-(yq)') J_0 & \xi J_0 + \amdeux^{-1} (1+(yq)') J_1 & \amdeux^{-1} (-1+(yq)') J_2 & \amdeux^{-1} (-1+(yq)') J_3
\\
\xi J_2 & \amdeux^{-1} J_2 & \xi J_0 + \amdeux^{-1} J_1 & 0 
\\
\xi J_3 & \amdeux^{-1} J_3 & 0 & \xi J_0 + \amdeux^{-1} J_1
\end{pmatrix}},
\ee
which are $- \amdeux^{-1} J_1/J_0$ (double) and two other roots
\be
\aligned
& \xi_\pm \coloneqq
 - \amdeux^{-1} \frac{(1 - k^2) J_0 J_1  \pm k \sqrt{- \Jbb \cdot \Jbb} \sqrt{J_0^2 - J_1^2 - k^2 (J_2^2 + J_3^2)}}
{J_0^2 - k^2 (J_1^2 + J_2^2 + J_3^2)}.
\endaligned
\ee
Altogether the twelve wave speeds for $(\Pbb,\Qbb,\Kpar,\Omega,\amdeux,\Jbb)$ are (with the future convention $J_0<0$) $\pm \amdeux^{-1}$ (double each), $0$ (quadruple), 
$- \amdeux^{-1} J_1/J_0$ (double), and~$\xi_{\pm}$. The two acoustic speeds are real and distinct from each other and from the matter speed when $0<k<1$ and $-\Jbb\cdot\Jbb>0$. They remain strictly inside the geometric light cone. They may coincide with the zero-speed geometric block at isolated states, as may the matter speed; such coincidences occur between decoupled principal blocks and do not destroy diagonalizability. The system admits a full basis of eigenvectors and is therefore strongly hyperbolic on the non-vacuum state cone. 

\subsubsection{Non-hyperbolicity at vacuum}

In fact, at a non-zero null state the square-root term in $\xi_\pm$ vanishes and $J_0^2=J_1^2+J_2^2+J_3^2$, so
\bel{eq:null-speed-coal}
\xi_+=\xi_-=-\amdeux^{-1}\frac{J_1}{J_0}.
\ee
All four fluid speeds therefore coalesce at the null boundary. The timelike diagonalization is not uniform there, and diagonalizability may be lost. The zero current itself has no preferred fluid direction, so it is treated only as a continuous boundary state of the balance laws. When, in addition, $p'(0)=0$, the degeneration is explicit: for every non-zero null state the Euler equations~\eqref{equa-Euler-Minko} become
\bel{equa-Euler-Minko-trivial}
\begin{aligned} 
 \bigl(  J_0 J_0\bigr)_t
- \amdeux^{-1} ( J_0 J_1 )_x
& = 0, \quad
& ( J_0 J_2 )_t - \amdeux^{-1} ( J_1 J_2 )_x
& = 0,
\\
( J_0 J_1 )_t - \amdeux^{-1} \, \bigl(  J_1^2 \bigr)_x
& = 0,
& ( J_0 J_3 )_t - \amdeux^{-1} ( J_1 J_3 )_x
& = 0.
\end{aligned}
\ee
All four wave speeds are equal to $- \amdeux^{-1} J_1/J_0$, but the system fails to be diagonalizable; the flux Jacobian has a nontrivial Jordan structure, hence the system is not symmetrizable and not a well-posed first-order hyperbolic system in the standard sense.

Note that $\hNumb(\mu)\to 0$ as $\mu\to 0$ by the asymptotics~\eqref{near-vacuum-Numb}.
If the parallel momentum per particle $\widehat J_m$ is assumed to be bounded, as holds for finite-energy solutions, then $J_m = \hNumb(\mu) \widehat J_m$ vanishes at vacuum and the system reduces formally to a transport equation at light speed. Indeed, on either formal null branch $J_1=\pm J_0$ with $J_2=J_3=0$, system~\eqref{equa-Euler-Minko-trivial} reduces to the two equations
\bel{equa-Euler-Minko-trivial-2}
\begin{aligned} 
 \bigl(  J_0 J_0\bigr)_t
- \amdeux^{-1} ( J_0 J_1 )_x
& = 0, \quad
& ( J_0 J_1 )_t - \amdeux^{-1} ( J_1 J_1 )_x
& = 0. 
\end{aligned}
\ee 
The two displayed equations coincide on each branch $J_1= \pm J_0$ since $J_0^2=J_1^2$.  


\subsection{Propagation of the constraints}
\label{section-A-4}

\subsubsection{Twist constraints}
\bse
The propagation of constraints~\eqref{eq:theconstraints-def} on $\log\Omega, K_2, K_3$ is now checked. Consider first the functions
\bel{opL11-and-14}
\aligned
{ |t|^{3/2} \opK_{20}} & = \bigl(  |t|^{3/2} K_2 \bigr)_t - |t|^{3/2} \Omega \bigl( J_1 J_2 -  P_0 K_2/2 \bigr), 
\\
|t|^{3/2} \opK_{21} & = \bigl( |t|^{3/2} K_2 \bigr)_x - |t|^{3/2} \Omega \amdeux  ( J_0 J_2 - P_1 K_2 / 2),
\endaligned
\ee
which vanish precisely when the evolution equation~\eqref{eq:T2-9101112-evol-def-a} and the constraint equation for~$K_2$ hold, respectively. The Euler equation~\eqref{eq:T2-fluid23-def-a} for~$J_2$ identifies derivatives of the second terms in~\eqref{opL11-and-14}, from which we find
\be
\aligned
\bigl( |t|^{3/2} \opK_{21} \bigr)_t
& = \bigl(  |t|^{3/2} K_2 \bigr)_{xt} 
- \Bigl( |t|^{3/2} \Omega \amdeux  \bigl( J_0 J_2 - P_1 K_2 / 2 \bigr) \Bigr)_t 
\\
& = \bigl(  |t|^{3/2} K_2 \bigr)_{tx} - \Bigl( |t|^{3/2} \Omega \bigl( J_1 J_2 - P_0 K_2 / 2 \bigr) \Bigr)_x
= { \bigl( |t|^{3/2} \opK_{20} \bigr)_x} = 0.
\endaligned
\ee
Since $\opK_{21}$ vanishes identically on one slice, it must vanish for all times. The same argument applies to the constraint for $K_3$, using the Euler equation~\eqref{eq:T2-fluid23-def-b} for $J_3$ and evolution equation~\eqref{eq:T2-9101112-evol-def-b} for~$K_3$.

\subsubsection{Lapse constraint}

On the other hand, by considering
\be
2 \opO_1  = \bigl( \log \Omega^2\bigr)_x
- t \, \Omega^2 \amdeux  \bigl( \Ebf_1(\Pbb) + \Ebf_1(\Qbb) + \Ebf_1(\Jperp) \bigr), 
\ee
we can verify the constraint~\eqref{eq:T2-567-def} on the lapse. Using the evolution equation~\eqref{eq:evollambda-def}
we obtain
\be
\aligned
(2 \opO_1 )_t 
& = \Bigl( t \, \Omega^2 \amdeux \Mbf^{01}(\Jbb,\Kpar,\Pbb,\Qbb) \Bigr)_t
+ \bigl( \log \Omega^2\bigr)_{xt}
\\
& = \Bigl( t \, \Omega^2 \amdeux  \Mbf^{01}(\Jbb,\Kpar,\Pbb,\Qbb) \Bigr)_t
+ \Bigl( t \, \Omega^2 \Mbf^{11}(\Jbb,\Kpar,\Pbb,\Qbb) - \frac{1}{2t}\Bigr)_x
= 0, 
\endaligned
\ee
as it coincides with the second Euler equation in~\eqref{eq:T2-Euler-perp-def}. Therefore $\opO_1 $ is constant in time and thus vanishes for all times if it vanishes at some initial time. This completes the proof of Propositions~\ref{theo-struct} and~\ref{theo-struct-bis}. 
\ese
%


\subsection{A technical inequality on the source term}
\label{section-A-5}

\subsubsection{General geometry}

We claim that the source term of the Euler equations satisfies 
\be
\abs{\Msource}\leq 5 \, \Mbf^{00}. 
\ee
\bse
Indeed, by the elementary inequality $\bigl|X_0^2-X_1^2\bigr|\leq X_0^2+X_1^2=2\,\Ebf_0(\Xbb)$, valid for any $\Xbb=(X_0,X_1)$, we have
\be
|\Pbb\cdot\Pbb|\leq 2\,\Ebf_0(\Pbb),\qquad
|\Qbb\cdot\Qbb|\leq 2\,\Ebf_0(\Qbb),\qquad
|\Jperp\cdot\Jperp|=-\Jperp\cdot\Jperp\leq 2\,\Ebf_0(\Jperp),
\ee
since $-\Jperp\cdot\Jperp\ge0$. Using $\Jbb\cdot\Jbb\le0$ and $q_\Jbb\in(0,1)$, we obtain 
\be
\aligned
|\Msource|
& \leq \Ebf_0(\Jpar) + 5\,\Ebf_0(\Kpar)+2\,\Ebf_0(\Pbb)+2\,\Ebf_0(\Qbb)+2\,\Ebf_0(\Jperp)
   +\frac{1}{2} q_\Jbb\,(-\Jbb\cdot\Jbb)
\\
& \leq 5\Bigl(\Ebf_0(\Jpar,\Kpar)+\Ebf_0(\Pbb,\Qbb,\Jperp)
      +\frac{1}{2} q_\Jbb\,(-\Jbb\cdot\Jbb)\Bigr)
=5\,\Mbf^{00}.
\endaligned
\ee
\ese
\subsubsection{Gowdy improvement}

In the Gowdy case, since $\Kpar=0$ and $\Jpar=0$, the source term simplifies and the above estimate can be improved. More precisely, in the isothermal Gowdy-symmetry case we obtain
\bel{equa-Gowdy facteur2}
|\Msource|\leq 2\,\Mbf^{00} \quad \text{(isothermal, Gowdy-symmetry).} 
\ee

\section{Traces on spacelike and timelike hypersurfaces}
\label{section-C}

\subsubsection{Integrated inequalities between spacelike hypersurfaces}

The estimates in Sections~\ref{section-7} and \ref{section-8} were first derived for regular flows, but their integrated form only uses the divergence of an integrable current. We use the normal-trace theory for divergence-measure fields in the elementary codimension-one form needed below. Let $s$ be a torus-invariant Lipschitz time function with spacelike level hypersurfaces $\Sigma_s$, let $n_s$ be their future unit normal, and let $\mathcal D_{s_0,s_1}$ denote the spacetime region between two levels. For a future-directed current $\mathbb Q\in L^1_{\rm loc}$ whose distributional divergence is a Radon measure, its normal flux is defined for almost every level by\footnote{This definition applies to the areal foliation but is not restricted to it.}
\bel{eq:weak-normal-flux}
\mathfrak H_{\mathbb Q}(s)
\coloneqq
\int_{\Sigma_s}-\mathbb Q^\alpha(n_s)_\alpha\,d\Sigma_s.
\ee
The following observation explains how such \emph{traces on spacelike hypersurfaces} are defined.
These results apply both to the mass-energy and the particle-number currents for finite-energy Einstein--matter areal flows.

\begin{lemma}[Traces on spacelike hypersurfaces]
\label{prop-weak-integrated-current}
Let $\mathbb Q\in L^1_{\rm loc}$ be torus invariant and suppose that
\bel{eq:weak-current-production}
\nabla_\alpha\mathbb Q^\alpha=-\lambda+r\,dV_g,
\qquad
\lambda\in\Meas^+_{\rm loc},
\qquad
r\in L^1_{\rm loc}.
\ee
Then the almost-everywhere function in~\eqref{eq:weak-normal-flux} admits a right-continuous representative of locally bounded variation. For every pair of continuity levels $s_0<s_1$, one has
\bel{eq:weak-Gauss-Green}
\mathfrak H_{\mathbb Q}(s_1)
+\lambda(\mathcal D_{s_0,s_1})
=\mathfrak H_{\mathbb Q}(s_0)
+\int_{\mathcal D_{s_0,s_1}}r\,dV_g.
\ee
In particular, if $r=0$, the normal flux~$\mathfrak H_{\mathbb Q}$ is non-increasing toward the future. The same identity holds between two torus-invariant spacelike Lipschitz hypersurfaces, with the normal fluxes defined by one-sided approximation; these traces do not depend on the Lipschitz time function used between the hypersurfaces.
\end{lemma}

\begin{proof}
We test~\eqref{eq:weak-current-production} with functions of the Lipschitz time variable~$s$. The coarea formula identifies the distributional derivative of $\mathfrak H_{\mathbb Q}$ with the finite signed measure obtained by pushing forward $-\lambda+r\,dV_g$. Hence $\mathfrak H_{\mathbb Q}$ has a right-continuous $\BV_{\rm loc}$ representative. Approximation of the characteristic function of $(s_0,s_1)$ gives~\eqref{eq:weak-Gauss-Green}. Applying the same identity to the thin region between two nearby level hypersurfaces proves that the traces are \emph{independent} of the interpolating time function.
\end{proof}


\subsubsection{Traces on timelike hypersurfaces}

We next define traces on timelike hypersurfaces.  This is expressed on the full interval $(t_0,t_1)$ but applies likewise to subintervals.

\begin{lemma}[Traces on timelike hypersurfaces]
\label{prop:time-traces}
Let $I=(t_0,t_1)$ and let $A,B\in L^1(I\times\Sbb^1)$ satisfy
\bel{eq:rect-bal}
\del_tA+\del_xB=\mu
\ee
in distributions, where $\mu$ is a finite Radon measure. For every $\varphi\in C_c^1(I)$, the function
\bel{eq:rect-Bphi}
x\longmapsto B_\varphi(x)
\coloneqq
\int_I B(t,x)\varphi(t)\,dt
\ee
has a $\BV(\Sbb^1)$ representative. Its precise representative defines a distributional normal trace $\operatorname{Tr}_xB$ on every timelike line $x=\mathrm{const}$. If $B\geq0$, these traces are non-negative Radon measures in~$t$. If, in addition, $A$ has the one-sided traces on spacelike hypersurfaces supplied by \autoref{prop-weak-integrated-current}, then
\bel{eq:rect-trace-bound}
\sup_{x\in\Sbb^1}\|\operatorname{Tr}_xB\|_{\Meas(I)}
\leq C\left(
\|B\|_{L^1(I\times\Sbb^1)}
+\|A(t_0+)\|_{L^1(\Sbb^1)}
+\|A(t_1-)\|_{L^1(\Sbb^1)}
+|\mu|(I\times\Sbb^1)
\right),
\ee
where $C$ depends only on the period of the spatial circle. The same conclusion holds for a weighted balance law after its positive weights have been included in $A$ and~$B$.
\end{lemma}

\begin{proof}
\bse
For $\varphi\in C_c^1(I)$, testing~\eqref{eq:rect-bal} against $\varphi(t)\psi(x)$ gives
\bel{eq:rect-BV-der}
\del_xB_\varphi
=\int_I\varphi\,d\mu
+\int_I A(t,\cdot)\varphi'(t)\,dt
\ee
as a finite Radon measure on~$\Sbb^1$. Hence $B_\varphi\in\BV(\Sbb^1)$, and its precise representative gives the asserted normal trace by duality. Positivity follows by using non-negative temporal test functions and symmetric non-negative spatial mollifiers.

By Fubini's theorem, we choose $x_0$ such that $B(\cdot,x_0)\in L^1(I)$ and its norm is bounded by the spatial average of~$|B|$. We next test~\eqref{eq:rect-bal} with temporal cut-offs converging to the characteristic function of~$I$ and spatial cut-offs converging to the characteristic function of the oriented interval from $x_0$ to~$x$. We obtain
\bel{eq:rect-GG}
\operatorname{Tr}_xB(I)-\operatorname{Tr}_{x_0}B(I)
=\int_{[x_0,x]}A(t_0+,y)\,dy
-\int_{[x_0,x]}A(t_1-,y)\,dy
+\mu\bigl(I\times[x_0,x]\bigr),
\ee
with the signs changed for the opposite orientation. Since the traces are non-negative, their total variations equal their masses. Taking absolute values in this identity and using the choice of~$x_0$ proves~\eqref{eq:rect-trace-bound}. This argument approximates both timelike sides of the rectangle and therefore does not require the pointwise restriction of an $L^1$ current to a timelike line.
\ese
\end{proof}

\end{document}

%% file: macros-DEUX-v22.tex
\usepackage{hyperref,bookmark,footnotebackref}
\hypersetup{linktoc = all}                
\hypersetup{hidelinks}
\hypersetup{bookmarksnumbered}
\makeatletter
\let\oldaligned\aligned
\def\@alignedenvir{aligned}
\def\aligned{\ifx\@currenvir\@alignedenvir\expandafter\@firstoftwo\fi\oldaligned\relax}
\makeatother
\newcommand{\abs}[1]{\lvert#1\rvert}
\newcommand{\Ric}{\operatorname{Ric}}
\usepackage{aliascnt}
\theoremstyle{plain}
\newtheorem{theorem}{Theorem}[section]
\newcommand{\mynewtheorem}[2]{
  \newaliascnt{#1}{theorem}
  \newtheorem{#1}[#1]{#2}
  \aliascntresetthe{#1}
  \expandafter\providecommand\csname #1autorefname\endcsname{#2}
}
\mynewtheorem{definition}{Definition}
\mynewtheorem{proposition}{Proposition}
\mynewtheorem{lemma}{Lemma}
\mynewtheorem{corollary}{Corollary}
\mynewtheorem{remark}{Remark}
\mynewtheorem{claim}{Claim}

\numberwithin{equation}{section}
\DeclareMathAlphabet\mathbfcal{OMS}{cmsy}{b}{n} 

\newcommand \bse {\begin{subequations}}
\newcommand \ese {\end{subequations}}

\renewcommand \div {\operatorname{div}}
\newcommand \divN {\operatorname{div}_{\Ncal}}
\newcommand \curl {\operatorname{curl}}

\newcommand \Jperp  {\mathbb J^\perp}
\newcommand \Jpar {J^\parallel{}}

\newcommand \Jhatpar { \widehat{J}^\parallel}
\newcommand \Kpar {K}
\newcommand \Xpar {X}

\newcommand \lot {\textnormal{l.o.t.}}
\newcommand \geo {\textnormal{g.m.t.}}

\newcommand \Jbb {\mathbb J} 
\newcommand \Lbb {\mathbb L}
\newcommand \Nbb {\mathbb N} 
\newcommand \Pbb {\mathbb P}
\newcommand \Qbb {\mathbb Q}
 
\newcommand \Xbb {\mathbb X} 
\newcommand \Ybb {\mathbb Y}

\newcommand \sign {\operatorname{sgn}}
\newcommand \Tr {\operatorname{Tr}}
\renewcommand \ln \log

\definecolor{mybrown}{RGB}{128,64,0}

\definecolor{myviolet}{RGB}{148,0,211}

\newcommand \la 	\langle
\newcommand \ra 	\rangle

\newcommand \Ncal 	{\mathcal N}

\newcommand \Ebf 	E            

\newcommand \bei 	{\begin{itemize}}
\newcommand \eei 	{\end{itemize}}
\newcommand \del	\partial
\newcommand \auth 	{\textsc}   
\newcommand \Mcal 	{\mathcal M}

\newcommand \RR 	{\mathbb R}   
\newcommand \CC 	{\mathbb C}   

\newcommand \Tbb 	{\mathbb T}   
\newcommand \Sbb 	{\mathbb S}   
\newcommand \eps 	\epsilon  
\newcommand \be 	{\begin{equation}}
\newcommand \ee 	{\end{equation}} 
\newcommand \bel 	{\be \label}
\let\oldmarginpar\marginpar
\renewcommand\marginpar[1]{\ifhmode\unskip\fi\- \oldmarginpar[\raggedleft\footnotesize #1]%
{\raggedright\footnotesize #1}}
\newcommand \lam 	\lambda

\newcommand{\WARNONCE}{\gdef\WARNONCE{}\GenericWarning{}{There are still BLF/PLF comments in this file!}}

\renewcommand \th 	\theta 
\newcommand \sgn 	{\operatorname{sgn}}

\newcommand \ZZ       {\mathbb{Z}} 

\renewcommand \geq \geqslant
\renewcommand \ge \geqslant
\renewcommand \leq \leqslant 
\renewcommand \le \leqslant

\newcommand \bea  {\begin{eqnarray}}
\newcommand \eea  {\end{eqnarray}}

\newcommand \rhofluid \rho

\newcommand \lbrac \llbracket 
\newcommand \rbrac \rrbracket

\let\Re\relax\let\Im\relax
\DeclareMathOperator{\Im}{\mathfrak{Im}}
\DeclareMathOperator{\Re}{\mathfrak{Re}}

\newcommand \coloneqq {\mathrel{\mathop :}\mathrel{\mkern-1.2mu}=} 
\newcommand \eqqcolon {=\mathrel{\mkern-1.2mu}\mathrel{\mathop :}} 

\makeatletter
\newcommand*\smallbullet{\mathpalette\smallbullet@{.5}}
\newcommand*\smallbullet@[2]{\mathbin{\vcenter{\hbox{\scalebox{#2}{$\m@th#1\bullet$}}}}}
\makeatother

  \newcommand \amdeux  \lambda  
  
  \newcommand \Mbf   M  

\newcommand \Vol {\mathrm{vol} \hskip.03cm}
\newcommand \guntrois {g^{(1+3)}}
\newcommand \dVuntrois {\mathrm{d}\Vol^{(1+3)}}
\newcommand \divuntrois {\div^{(1+3)}}
\newcommand \curluntrois {\curl^{(1+3)}}

\newcommand \dVtrois {\mathrm{d}\Vol^{(3)}}

\newcommand \dVundeux {\mathrm{d}\Vol^{(1+2)}}

\newcommand \suit \eps

\newcommand \Msource {\Mbf_{\textnormal{\textbf{sour}}}}
\newcommand \Mwave {\Mbf_{\textnormal{\textbf{wave}}}}

\newcommand \JKL {\texorpdfstring{\ensuremath{\Jbb K\Lbb}}{JKL}}

\newcommand \Numb {N}
\newcommand \hNumb {\widehat{N}}
\newcommand \vNumb {{\mathbb N}}

\newcommand{\Obig}{\mathcal{O}}

\newcommand \Gammazero  {\widetilde K}      
\newcommand \Gammaun     {\widetilde J}         
\newcommand \Gammaunpar     {\widetilde J^\parallel}         
\newcommand \funmu {\underline{\smash{\ensuremath{\mu}}}}

\newcommand \BV {\ensuremath{BV}}
\newcommand \BVac {\ensuremath{BV_{\textnormal{ac}}}}
\newcommand \Meas {M}

\newcommand \Var {\textnormal{Var}}

\newcommand \muref {\mu_{\textnormal{ref}}}

\newcommand \opL {\mathfrak{L}}
\newcommand \opJ {\mathfrak{J}}
\newcommand \opO {\mathfrak{O}}
\newcommand \opC {\mathfrak{C}}
\newcommand \opK {\mathfrak{K}}

\newcommand \Interval {I}

\providecolor{mybrown}{RGB}{128,64,0}

\providecommand{\del}{\mathord{\partial}}

\makeatletter
\newif\ifPLFshowuncited
\PLFshowuncitedfalse
\def\citation#1{%
  \@for\PLF@citekey:=#1\do{%
    \expandafter\gdef\csname PLFcited@\PLF@citekey\endcsname{1}}}
\newcommand{\PLFifcited}[1]{%
  \@ifundefined{PLFcited@#1}{\@secondoftwo}{\@firstoftwo}}

\makeatother

\DeclareFontFamily{U}{mathx}{}
\DeclareFontShape{U}{mathx}{m}{n}{<-> mathx10}{}
\DeclareSymbolFont{mathx}{U}{mathx}{m}{n}
\DeclareMathAccent{\widehat}{0}{mathx}{"70}
\DeclareMathAccent{\widecheck}{0}{mathx}{"71}

\providecolor{cadmiumgreen}{RGB}{0,107,60}
\providecommand \cadmiumgreen {\color{cadmiumgreen}}
\providecolor{deepteal}{RGB}{0,96,104}

\newcommand \velo {\mathrm{v}}
\newcommand \Sym {\operatorname{Sym}}